\documentclass[english,draftcls,onecolumn, 12pt]{IEEEtran}

\usepackage[T1]{fontenc}
\usepackage[latin9]{inputenc}
\usepackage{color}
\usepackage{array}
\usepackage{float}
\usepackage{amsmath}
\usepackage{amsthm}
\usepackage{amssymb}
\usepackage{graphicx}

\makeatletter

\providecommand{\tabularnewline}{\\}
\floatstyle{ruled}
\newfloat{algorithm}{tbp}{loa}
\providecommand{\algorithmname}{Algorithm}
\floatname{algorithm}{\protect\algorithmname}

\theoremstyle{plain}
\newtheorem{thm}{\protect\theoremname}
\theoremstyle{plain}
\newtheorem{lem}[thm]{\protect\lemmaname}

\theoremstyle{example}
\newtheorem{assumption}{Assumption}

\theoremstyle{definition}
\newtheorem{remark}{Remark}

\theoremstyle{plain}
\newtheorem{mythm}{Theorem}

\theoremstyle{plain}
\newtheorem{myprop}{Proposition}

\usepackage{babel}
\providecommand{\lemmaname}{Lemma}
\providecommand{\theoremname}{Theorem}

\makeatother

\usepackage{babel}
\providecommand{\lemmaname}{Lemma}
\providecommand{\theoremname}{Theorem}

\begin{document}

\title{Distributed Derivative-free Learning Method\linebreak{}
 for Stochastic Optimization over a Network\linebreak{}
 with Sparse Activity}

\author{Wenjie~Li, Mohamad~Assaad,~\IEEEmembership{Senior Member, IEEE},
and Shiqi~Zheng \thanks{This paper has been presented in part at IEEE 57th Conference on Decision
and Control (CDC), Miami Beach, FL, USA, December, 2018 \cite{li2018cdc}.}\thanks{W. Li was with  Laboratoire des Signaux et Systèmes,
CentraleSupélec, when this manuscript was submitted. He is now with
Paris Research Center, Huawei Technologies, France (e-mail: liwenjie28@huawei.com) }\thanks{M. Assaad is with Laboratoire des Signaux et
Systèmes (L2S, UMR CNRS 8506), CentraleSupélec, France (e-mail: mohamad.assaad@centralesupelec.fr).}\thanks{S. Zheng is with the Hubei Key Laboratory of Advanced Control and
Intelligent Automation for Complex Systems, China University of Geosciences,
Wuhan 430074, China.}}
\maketitle
\begin{abstract}
This paper addresses a distributed optimization problem in a communication
network where nodes are active sporadically. Each active node applies
some learning method to control its action to maximize the global
utility function, which is defined as the sum of the local utility
functions of active nodes. We deal with stochastic optimization problem
with the setting that utility functions are disturbed by some non-additive
stochastic process. We consider a more challenging situation where
the learning method has to be performed only based on a scalar approximation
of the utility function, rather than its closed-form expression, so
that the typical gradient descent method cannot be applied. This setting
is quite realistic when the network is affected by some stochastic
and time-varying process, and that each node cannot have the full
knowledge of the network states. We propose a distributed optimization
algorithm and prove its almost surely convergence to the optimum.
Convergence rate is also derived with an additional assumption that
the objective function is strongly concave. Numerical results are
also presented to justify our claim. 
\end{abstract}

\begin{IEEEkeywords}
Stochastic optimization, derivative-free learning, sparse network,
convergence analysis, distributed algorithm 
\end{IEEEkeywords}

\section{Introduction}

We consider the distributed optimization of a network with sparse
communication, \emph{i.e.}, nodes are active occasionally in a discrete-time
system, so that only a small number of nodes are active at the same
time-slot. For example, in the modern communication system, independent
mobile phone users communicate with the base station at different
time. This model is also important in Internet of Things such as underwater
wireless sensor networks, where sensor nodes keep asleep frequently
to save battery.

Suppose that the performance of the network is characterized by a
global utility function, which is defined as the sum of the local
utility functions of all active nodes at one time-slot. Each active
node aims to properly control its own action to maximize the global
utility. The local utility of any active node is a function of the
action of all the active nodes, as well as some stochastic environment
state that can be seen as a non-additive stochastic process, \emph{e.g.},
stochastic and time-varying channel gain in wireless communication
system. Such stochastic optimization problem is important for the
improvement of network performance and has attracted much attention
in various field, \emph{e.g.}, radio resource management~\cite{gaie2008distributed},
power control~\cite{chiang2008power,hassan2013distributed}, and
beamforming allocation~\cite{sheng2017joint}.

The convex optimization problem is well investigated by applying the
typical gradient descent/ascent method \cite{snyman2005practical},
under the condition that each node is able to calculate the partial
derivative related to its action. Sub-gradient based methods have
been proposed to solve distributed optimization of the sum of several
convex function, over time-varying \cite{nedic2015distributed,nedic2016stochastic,nedic2017achieving}
or asynchronous \cite{nedic2011asynchronous,srivastava2011distributed}
networks. In these previous work, each node/agent requires the gradient
information of its local function to perform the optimization algorithm.

Stochastic learning schemes based on stochastic gradient descent have
been widely studied. In this work we consider a more challenging framework
that nodes are unaware of any gradient information. Since the network
is distributed by some non-additive stochastic process, our setting
is quite practical in the following situations: 
\begin{itemize}
\item the system is so complex that the closed-form expression of any utility
function is unavailable; 
\item the computation of gradient requires much informational exchange and
introduces a huge burden to the network. 
\end{itemize}
A detailed motivating example is presented in Section~\ref{sec:Motivating-Example}
to highlight the interest of our setting. We assume that each active
node only has a numerical observation of its local utility, our optimization
algorithm should be performed \emph{only} based on this zeroth-order
information. Moreover, we consider a distributed setting such that
nodes can only exchange the local utilities with their neighbors in
order to estimate the global utility, which make the problem more
challenging.

\subsection{Related work}

Our derivative-free optimization problem is known as zero-order stochastic
optimization and bandit optimization. There are numbers of work based
on two-point gradient estimator, \emph{e.g.,} \cite{agarwal2010optimal,duchi2015optimal,liu2018zeroth,hajinezhad2019zone},
under the assumption that two values of the objective function $f(\boldsymbol{a}_{k}^{\left(1\right)};\boldsymbol{s}_{k})$
and $f(\boldsymbol{a}_{k}^{\left(2\right)};\boldsymbol{s}_{k})$ are
available under the same stochastic parameter $\boldsymbol{s}_{k}$.
However such assumption is unrealistic in our setting, e.g., in i.i.d.
channel, the value of $\boldsymbol{s}_{k}$ change fast, it is impossible
to observe two network utilities using different action $\boldsymbol{a}_{k}$
while under the same environment state. Therefore, we should propose
some gradient estimator only based on a single realization of objection
function to estimate the gradient. A classical method was proposed
in \cite{flaxman2005online} of which the algorithm is \emph{near-optimal}:
for general convex and Lipschitz objective functions, the resulted
optimization error is $O(K^{-0.25})$ after a total number of $K$
iterations. From then on, several advanced methods were proposed (e.g.,
\cite{shamir2013complexity,hazan2014bandit,bubeck2017kernel}) to
accelerate the convergence speed of the algorithm for the general
convex functions or the convex functions with additional assumptions,
e.g., smooth or strongly convex. However, the optimal algorithm to
address bandit optimization is still unknown. It is worth mentioning
that, the optimization error cannot be better than $O(K^{-0.5})$
after $K$ iterations, according to the lower bounds of the convergence
rate derived in \cite{jamieson2012query,shamir2013complexity,duchi2015optimal}.

Although bandit optimization has attracted much attention in recent
years, the existed algorithms are usually centralized and hard be
decentralized. {In fact, in all the above mentioned
references, their algorithms contain the operations of vectors and
matrices that require a control center to handle. In our setting,
each node only controls its local variable (a coordinate) and may
not have the full knowledge of the objective function due to the distributed
setting. For example, in the algorithm proposed in \cite{flaxman2005online},
the core is to generate a random } {\emph{unit}} {{}
perturbation vector at each iteration, which is the key to ensure
that the expectation of the resulted gradient estimator is } {\emph{equal}} {{}
to the gradient of a } {\emph{smoothed}} {{}
version of the objective function by applying Stokes Theorem. This
requires a control center as the resulted perturbation vector cannot
have a unit norm without such a control center. In our distributed
network, each node can only generate its own random perturbation independently.
Different tools are needed to obtain our analytically results: we
derive upper bounds for the bias of gradient estimator. Moreover,
the existing work in learning community usually focused on the performance
after a given number of iterations. However, finite-time horizon is
not adapted to wireless networks, as it is usually hard to predict
the duration of connection and the total number of iterations. For
the above reason, in this work, we aim to propose some optimal solution
with } {\emph{asymptotic}} {{} performance
guarantee. }

In our recent work \cite{li2017distributed}, we have proposed a learning
algorithm named DOSP (distributed optimization algorithm using stochastic
perturbation) to solve the above derivative-free optimization problem,
however, in a synchronized network with small number of nodes, \emph{i.e.},
nodes are always active and update their action at each time-slot.
The basic idea of the DOSP algorithm is to estimate the gradient of
the objective function only based the numerical measurement of the
objective function. It has been shown that the estimation bias of
gradient is vanishing as the number of nodes is finite. The convergence
of the DOSP can be proved with the tools of stochastic approximation
\cite{kushner2012stochastic}. This technique is closely related to
simultaneous perturbation gradient approximation in \cite{spall1992multivariate,spall1997one}
and extremum seeking with stochastic perturbation proposed in \cite{manzie2009extremum}.
Please refer to \cite{li2017distributed} for the detailed discussion.
It is worth mentioning that, sine perturbation based extremum seeking
method \cite{ariyur2003real,frihauf2012nash,hanif2012spawc} can be
another option to solve the derivative-free optimization problem.
However, it is impractical to ensure that the sine function used by
each node is orthogonal in a distributed setting, especially when
the number of nodes is large.

\subsection{Our contribution}

This paper extends our previous results in \cite{li2017distributed}
by considering a more realistic network model, \emph{i.e}., nodes
are sporadically active and the entire network may be of large scale.
The achievable value of action of each node is considered as constrained,
\emph{i.e.}, belonging to some closed-interval. We present a modified
DOSP algorithm with two major differences compared with the original
DOSP algorithm: nodes can update their action only when they are active;
each node updates its step-size \emph{asynchronously},\emph{ independently},
and \emph{randomly}, according to its times of being active.

This paper focuses on the convergence analysis of the proposed learning
algorithm. Convergence rate has also be investigated with an additional
assumption that the utility functions are strongly concave. Compared
with that in \cite{li2017distributed}, the analysis is much more
challenging because of the additional random terms. The network is
dynamic as nodes have random activity, its global utility function
is harder to be characterized than a fixed network that nodes are
always active. As we try to estimate the gradient using the numerical
value of utility function, an essential term to be analyzed is the
estimation bias of gradient. In \cite{li2017distributed}, an upper
bound of such bias term is proved to be proportional to the vanishing
step-size, which is deterministic and identical for all nodes at each
iteration.  {Due to the random activity of each node,
the algorithm is performed in an asynchronous manner, i.e., the times
of update of each node is random. As a consequence, the step-size
of each node (function of times of update) is random and independent,
which makes the problem further challenging. We have to resort to
some new tools such as concentration inequalities to show that the
bias term is vanishing as well. }  {It is notable that
our proposed solution can achieve the optimal convergence rate when
the objective function is smooth and strongly concave: our achievable
optimization error is proved to be $O(K^{-0.5})$, which is the same
compared with the lower bounds in \cite{shamir2013complexity,duchi2015optimal}
in terms of the decreasing order. }

The rest of the paper is organized as follows. Section~\ref{sec:System-Model}
describes the problem as well as some basic assumptions. Section~\ref{sec:Motivating-Example}
provides examples to motivate the interest of the problem. Section~\ref{sec:Main-results}
presents our distributed optimization algorithm, of which the almost
sure convergence discussed in Section~\ref{subsec:Convergence-results}.
The convergence rate of the proposed learning algorithm is derived
in Section~\ref{sec:Convergence-Rate}. Section~\ref{sec:Numerical-illustration}
presents some numerical illustrations and Section~\ref{sec:Conclusion}
concludes this paper. Main notations in this paper are listed in Table~1.

\begin{table}
\caption{Main notations and their interpretation}
\centering{}%
\begin{tabular}{|c|>{\raggedright}p{0.72\columnwidth}|}
\hline 
$\mathcal{N}$  & set of nodes\tabularnewline
\hline 
$\mathcal{N}^{\left(k\right)}$  & set of active nodes at time-slot~$k$\tabularnewline
\hline 
$q_{\mathrm{a}}$  & the probability of a node being active at each time-slot\tabularnewline
\hline 
$\mathcal{I}^{\left(i,k\right)}$  & set of active nodes which have successfully sent their local utilities
to another active node~$i$ at time-slot~$k$\tabularnewline
\hline 
$q_{\mathrm{r}}$  & the probability of the successful reception of a local utility from
an active node to another\tabularnewline
\hline 
$\delta_{i,k}$  & a binary variable indicating whether node~$i$ is active at time-slot~$k$\tabularnewline
\hline 
$n_{k}$  & number of active nodes at time-slot~$k$\tabularnewline
\hline 
$\lambda$  & expected value of $n_{k}$\tabularnewline
\hline 
\hline 
$a_{i,k}$  & value of the action performed by node~$i$ at time-slot~$k$\tabularnewline
\hline 
$\mathbf{S}_{k}$  & stochastic environment matrix\tabularnewline
\hline 
$u_{i}$  & local utility function of node~$i$\tabularnewline
\hline 
$\widetilde{u}_{i,k}$  & numerical observation of $u_{i}$ at time-slot~$k$\tabularnewline
\hline 
$\eta_{i,k}$  & additive noise, the difference between $\widetilde{u}_{i,k}$ and
$u_{i}\left(\boldsymbol{a}_{k},\boldsymbol{\delta}_{k},\mathbf{S}_{k}\right)$\tabularnewline
\hline 
$f$  & global utility function of the network \tabularnewline
\hline 
$F$  & average global utility function of the network \tabularnewline
\hline 
$G$  & expected value of $f$ with a given realization of $\boldsymbol{\delta}_{k}$\tabularnewline
\hline 
\hline 
$\Phi_{i,k}$  & random perturbation used by node~$i$ at time-slot~$k$\tabularnewline
\hline 
$\left\{ \gamma_{.}\right\} ,\left\{ \beta_{.}\right\} $  & vanishing sequences from which step-sizes take values\tabularnewline
\hline 
$\ell_{i,k}$  & index of $\left\{ \gamma_{.}\right\} $ and $\left\{ \beta_{.}\right\} $
to be used as step-sizes \tabularnewline
\hline 
$\widetilde{\gamma}_{i,k},\widetilde{\beta}_{i,k}$  & step-sizes used by node~$i$ at time-slot~$k$\tabularnewline
\hline 
\end{tabular}
\end{table}

\section{System Model and Assumptions \label{sec:System-Model}}

\subsection{Network model}

Consider a network $\mathcal{N}$ with $N=\left|\mathcal{N}\right|$
nodes and a time-varying directed graph $\mathcal{G}^{\left(k\right)}=(\mathcal{N},\mathcal{E}^{\left(k\right)})$
at each discrete time-slot~$k$. Note that the edge set $\mathcal{E}^{\left(k\right)}$
is a set of pairs of nodes that are able to have direct communication.
We can use a communication matrix $\mathbf{E}(k)=[E_{i,j}(k)]_{i,j\in\mathcal{N}}$
to describe the network connectivity, with $E_{i,j}(k)\neq0$ if and
only if $(i,j)\in\mathcal{E}^{\left(k\right)}$. In this work, the
network topology is assumed to be stochastic, such that any two different
nodes can become neighbors with a non-zero constant probability, i.e.,
$\mathbb{P}(E_{i,j}(k))>0$, $\forall i,j\in\mathcal{N}$. It is worth
mentioning that such assumption can be naturally satisfied when nodes
are moving freely in some closed area.

Suppose that at each discrete time-slot~$k$, only a random subset
$\mathcal{N}^{\left(k\right)}\subseteq\mathcal{N}$ of nodes are \emph{active},
\emph{i.e.}, perform some action. Introduce a binary variable $\delta_{i,k}$
to indicate whether node~$i$ is active or not at time-slot~$k$,
\emph{i.e.}, 
\[
\delta_{i,k}=\begin{cases}
1, & \textrm{if }i\in\mathcal{N}^{\left(k\right)},\\
0, & \textrm{else}.
\end{cases}
\]
Define $a_{i,k}$ as the \emph{value} of the action of node~$i$
at time-slot~$k$ under the condition that $\delta_{i,k}=1$. Suppose
that the value of $a_{i,k}$ is bounded, \emph{i.e.}, $a_{i,k}\in\mathcal{A}_{i}=\left[a_{i,\min},a_{i,\max}\right]$.
Denote $\mathcal{A}=\mathcal{A}_{1}\times\ldots\times\mathcal{A}_{N}$
as the feasible set of the action vector $\boldsymbol{a}_{k}=\left[a_{1,k},\ldots,a_{N,k}\right]^{T}$.
Denote 
\begin{equation}
\sigma_{\boldsymbol{a}}^{2}=\max_{i\in\mathcal{N}}\left\{ a_{i,\min}^{2},a_{i,\max}^{2}\right\} .\label{eq:amax-1}
\end{equation}

Introduce $n_{k}=\left|\mathcal{N}^{\left(k\right)}\right|=\sum_{i=1}^{N}\delta_{i,k}$
the number of active nodes at time-slot~$k$. Mathematically, we
assume that:

\begin{assumption}\label{Assumption:net}The binary variables $\delta_{i,k}$
are \emph{i.i.d.} with $\mathbb{P}(\delta_{i,k}=1)=q_{\mathrm{a}}$.
Then $n_{k}$ follows a binomial distribution with $\mathbb{E}(n_{k})=Nq_{\mathrm{a}}=\lambda$.
In the situation where $N\rightarrow\infty$, the value of $q_{\mathrm{a}}$
is small such that $\lambda<\infty$, $n_{k}$ follows a Poisson distribution
with parameter $\lambda$.\end{assumption}

Note that we can have a large network with $N\rightarrow\infty$,
while our results hold for any value of $N$ as long as $\lambda=Nq_{\mathrm{a}}<\infty$.

\subsection{Utility functions}

We assume that each active node~$i$ with $\delta_{i,k}=1$ is able
to evaluate a pre-defined local utility function $u_{i}(\boldsymbol{a}_{k},\boldsymbol{\delta}_{k},\mathbf{S}_{k})$,
which depends on the action vector $\boldsymbol{a}_{k}$, the activity
vector $\boldsymbol{\delta}_{k}=\left[\delta_{1,k},\ldots,\delta_{N,k}\right]^{T}$,
and is also disturbed by a non-additive stochastic process $\mathbf{S}_{k}$
of the whole network, \emph{e.g.}, stochastic channels in wireless
networks. Consider $\mathbf{S}_{k}\in\mathcal{S}$ as a stochastic
matrix to describe the environment state of the network at any time-slot~$k$,
which is assumed to be independent and identically distributed (i.i.d.)
in this paper. The local utilities of the non-active nodes are not
meaningful, thus we define $u_{i}=0$ if $\delta_{i,k}=0$.

The \emph{global} utility $f\left(\boldsymbol{a}_{k},\boldsymbol{\delta}_{k},\mathbf{S}_{k}\right)$
of the entire network is defined as the sum of local utilities of
the active nodes at each time-slot~$k$, \emph{i.e.}, 
\begin{align}
 & f\left(\boldsymbol{a}_{k},\boldsymbol{\delta}_{k},\mathbf{S}_{k}\right)=\sum_{i\in\mathcal{N}^{\left(k\right)}}u_{i}(\boldsymbol{a}_{k},\boldsymbol{\delta}_{k},\mathbf{S}_{k}).\label{f}
\end{align}
We are interested in the configuration of the value of $a_{i,k}$
for each node~$i$ such that $i\in\mathcal{N}^{\left(k\right)}$
at each time-slot~$k$, in order to the maximize of the\emph{ average
global} utility function 
\begin{equation}
F\left(\boldsymbol{a}_{k}\right)=\mathbb{E}_{\boldsymbol{\delta},\mathbf{S}}\left(f\left(\boldsymbol{a}_{k},\boldsymbol{\delta}_{k},\mathbf{S}_{k}\right)\right).\label{eq:F}
\end{equation}
It is also necessary to define the average global utility function
with a given realization of the activity vector $\boldsymbol{\delta}_{k}$,\emph{
i.e.}, 
\begin{equation}
G\left(\boldsymbol{a}_{k},\boldsymbol{\delta}_{k}\right)=\mathbb{E}_{\mathbf{S}}\left(f\left(\boldsymbol{a}_{k},\boldsymbol{\delta}_{k},\mathbf{S}\right)\right).\label{eq:G}
\end{equation}
According to Assumption~\ref{Assumption:net}, it is easy to deduce
that 
\begin{equation}
F\left(\boldsymbol{a}_{k}\right)=\sum_{\boldsymbol{\delta}_{k}\in\mathcal{D}}q_{\mathrm{a}}^{n_{k}}\left(1-q_{\mathrm{a}}\right)^{N-n_{k}}G\left(\boldsymbol{a}_{k},\boldsymbol{\delta}_{k}\right)\label{eq:F_G}
\end{equation}
with $\mathcal{D}=\{\boldsymbol{\delta}=\left[\delta_{1},\ldots,\delta_{N}\right]^{T}:\delta_{i}\in\left\{ 0,1\right\} ,\forall i\}$.

Assume that at time-slot~$k$ each active node~$i\in\mathcal{N}^{\left(k\right)}$
is able to have a numerical observation $\widetilde{u}_{i,k}$ of
$u_{i}(\boldsymbol{a}_{k},\boldsymbol{\delta}_{k},\mathbf{S}_{k})$:
\begin{equation}
\widetilde{u}_{i,k}=u_{i}(\boldsymbol{a}_{k},\boldsymbol{\delta}_{k},\mathbf{S}_{k})+\eta_{i,k},
\end{equation}
where $\eta_{i,k}$ is the additive random noise caused by observation
of $u_{i}$. Such noise is assumed to be statistically independent
and have zero mean and bound variance.

\begin{assumption}\label{Assumption:add_noise} For any integer $k$
and $i\in\mathcal{N}^{\left(k\right)}$, we have $\mathbb{E}\left(\eta_{i,k}\right)=0$
and $\mathbb{E}\left(\eta_{i,k}^{2}\right)=\sigma_{\eta}^{2}<\infty$.
Besides, for any $i\neq j$ and $k\neq k'$, we have $\mathbb{E}\left(\eta_{i,k}\eta_{j,k}\right)=\mathbb{E}\left(\eta_{i,k}\eta_{i,k'}\right)=0$.\end{assumption}

In order to approximate the global utility of the network, active
nodes have to broadcast their observation of local utilities to their
active neighbors (other active nodes within transmission range). Without
any communication, an active node only knows its local utility. We
consider a realistic situation where an active node~$i$ can receive
$\widetilde{u}_{j,k}$ from another active node~$j$ only if both
of the following events occur: \emph{E1.} node~$j$ is a neighbor
of node~$i$ at time-slot~$k$;\emph{ E2.} there is no collision
or packet loss during the transmission. In other words, node~$i$
receives $\widetilde{u}_{j,k}$ from a subset $\mathcal{I}^{\left(i,k\right)}$
of its active neighbors, with $\mathcal{I}^{\left(i,k\right)}\subseteq\mathcal{N}^{\left(k\right)}\setminus\left\{ i\right\} $.
Mathematically:

\begin{assumption}\label{Assumption:incomp} At any time-slot~$k$,
any active node~$i\in\mathcal{N}^{\left(k\right)}$ knows the utility
$\widetilde{u}_{j,k}$ of another active node~$j\in\mathcal{N}^{\left(k\right)}$
with a constant probability $q_{\mathrm{r}}\in\left(0,1\right]$,
$i.e.$, 
\begin{equation}
\mathbb{P}\left(j\in\mathcal{I}^{\left(i,k\right)}\right)=q_{\mathrm{r}},\:\mathbb{P}\left(j\notin\mathcal{I}^{\left(i,k\right)}\right)=1-q_{\mathrm{r}},\:\forall j\neq i.\label{eq:p}
\end{equation}
Note that $q_{\mathrm{r}}$ is in fact a joint probability of events
E1 and E2.\end{assumption}

In Section~\ref{subsec:Estimation_f}, we will present an efficient
way to estimate the global utility $\widetilde{f}$ using incomplete
information of $\widetilde{u}_{i,k}$.

\begin{remark}Note that it is straightforward to extend the results
in this work to a more general case where $\mathbb{P}(j\in\mathcal{I}^{\left(i,k\right)})$
is not identical. We assume that $\mathbb{P}(j\in\mathcal{I}^{\left(i,k\right)})=q_{\mathrm{r}}$
mainly to lighten the expressions of this paper. \end{remark}

 {It is worth mentioning that our aforementioned network
model can hold in wireless settings. In fact, the wireless link between
any two nodes in such a network is affected by fast fading, modeled
usually by Rayleigh or Nakagami distribution. This implies that the
link changes from one slot to another in an i.i.d. way. If the link
is good, then the nodes can communicate and if the link is bad they
cannot communicate. As a result, the link qualities in such a time-varying
network are reshuffled at each slot. } 

\subsection{Problem formulation}

With the above definition, our problem can be written as 
\begin{equation}
\begin{cases}
\underset{\boldsymbol{a}}{\textrm{max}} & F\left(\boldsymbol{a}\right)=\mathbb{E}_{\boldsymbol{\delta}}\left(G\left(\boldsymbol{a},\boldsymbol{\delta}\right)\right)=\mathbb{E}_{\boldsymbol{\delta},\mathbf{S}}\left(f\left(\boldsymbol{a},\boldsymbol{\delta},\mathbf{S}\right)\right)\\
\textrm{ s.t.} & \boldsymbol{a}\in\mathcal{A}
\end{cases}
\end{equation}
We consider a situation where nodes do not have the knowledge of $\mathbf{S}$
to get the closed-form expression of the utility functions. This setting
is quite realistic as $\mathbf{S}$ may have large dimension and be
constantly time-varying. In this paper, the proposed learning algorithm
is performed only with the numerical value of utility function. An
motivating example is introduced in Section~\ref{sec:Motivating-Example}.

Denote $\boldsymbol{a}^{*}=\left[a_{1}^{*},\ldots,a_{N}^{*}\right]$
as the optimum solution of the problem. To ensure the existence of
$\boldsymbol{a}^{*}$, we assume that:

\begin{assumption}\label{Assumption:F}Both $G\left(\boldsymbol{a},\boldsymbol{\delta}\right)$
and $F\left(\boldsymbol{a}\right)$ are first order and second order
differentialable functions of $\boldsymbol{a}\in\mathcal{A}$. The
optimal point $\boldsymbol{a}^{*}$ exists such that $\partial F\left(\boldsymbol{a}^{*}\right)/\partial a_{i}=0$
and $\partial^{2}F\left(\boldsymbol{a}^{*}\right)/\partial a_{i}^{2}<0$,
$\forall i\in\mathcal{N}$. Besides, $\boldsymbol{a}^{*}$ is not
on the boundary of $\mathcal{A}$, i.e., $a_{i}^{*}\in(a_{1,\min},a_{1,\max})$,
$\forall i\in\mathcal{N}$. The objective function $F$ is strictly
concave such that 
\begin{equation}
\left(\boldsymbol{a}-\boldsymbol{a}'\right)^{T}\!\cdot\!\left(\nabla\!F\left(\boldsymbol{a}\right)-\nabla\!F\left(\boldsymbol{a}'\right)\right)\!<\!0,\:\forall\boldsymbol{a},\boldsymbol{a}'\in\mathcal{A}:\boldsymbol{a}\neq\boldsymbol{a}'.\label{eq:concave}
\end{equation}
\end{assumption}

\begin{remark}It is worth mentioning that we assumed $\left(a_{i,k}-a_{i}^{*}\right)^{T}\frac{\partial}{\partial a_{i,k}}F\left(\boldsymbol{a}_{k}\right)\leq0,$
$\forall i\in\mathcal{N}$ in \cite{li2018cdc}, which has been relaxed
by Assumption~\ref{Assumption:F} in this paper.\end{remark}

We have a further assumption to ensure the performance of the proposed
derivative-free learning algorithm.

\begin{assumption}\label{Assumption:lipc} There exists $\alpha_{G}\in\left(0,+\infty\right)$
such that 
\begin{equation}
\left|\frac{\partial^{2}}{\partial a_{i}\partial a_{j}}G\left(\boldsymbol{a},\boldsymbol{\delta}\right)\right|\leq\alpha_{G},\quad\forall i,j\in\mathcal{N}^{\left(k\right)}
\end{equation}
The function $\boldsymbol{a}\longmapsto u_{i}\left(\boldsymbol{a},\boldsymbol{\delta},\mathbf{S}\right)$
is Lipschitz for any $\boldsymbol{\delta}$ and $\mathbf{S}$, 
\begin{equation}
\left\Vert u_{i}\left(\boldsymbol{a},\boldsymbol{\delta},\mathbf{S}\right)-u_{i}\left(\boldsymbol{a}',\boldsymbol{\delta},\mathbf{S}\right)\right\Vert \leq L_{\mathbf{S}}\left\Vert \left(\boldsymbol{a}-\boldsymbol{a}'\right)\circ\boldsymbol{\delta}\right\Vert ,\label{eq:lipsc}
\end{equation}
with constant $L_{\mathbf{S}}<\infty$. Besides, $L=\sqrt{\mathbb{E}_{\mathbf{S}}\left(L_{\mathbf{S}}^{2}\right)}<\infty$.\end{assumption}

\section{Motivating Example\label{sec:Motivating-Example}}

Recently, derivative-free optimization is of interest in various applications,
e.g., management of fog computing in IoT \cite{chen2018bandit}, sensor
selection for parameter estimation \cite{liu2017zeroth}, and adversarial
machine learning \cite{liu2018zeroth2}. In this section we provide
another motivating example, which of particular interest for the problem
considered in this paper.

Consider a power allocation problem in a network with $N$ transmitter-receiver
links. As shown in Figure~1, each link corresponds to a node in our
system model. Transmitter~$i$ sends some packet to receiver~$i$
when $\delta_{i,k}=1$. Let $\mathbf{S}_{k}=\left[s_{ij,k}\right]_{i,j\in\mathcal{N}}$
denote the time-varying stochastic channel matrix, each element $s_{ij,k}\in\mathbb{R}^{+}$
represents the channel gain between transmitter $i$ and receiver
$j$ at time $k$. Each active transmitter~$i$ sets its transmission
power $p_{i,k}$, the Shannon achievable rate of the link is then
given by \cite{tan2013fast} 
\begin{equation}
r_{i,k}=\log\left(1+\frac{s_{ii,k}p_{i,k}}{\sigma^{2}+\sum_{j\neq i}\delta_{j,k}s_{ji,k}p_{i,k}}\right).
\end{equation}
At each time-slot $k$, define the global utility, which is widely
used in wireless systems, as $y_{k}\left(\boldsymbol{p}_{k},\boldsymbol{\delta}_{k},\mathbf{S}_{k}\right)=\sum_{i\in\mathcal{N}^{\left(k\right)}}(\omega_{1}r_{i,k}-\omega_{2}p_{i,k})$,
where $\omega_{1},\omega_{2}\in\mathbb{R}^{+}$ are constants and
$\omega_{2}p_{i,k}$ denotes the energy costs of the packet transmission.

However, $y_{k}$ is not concave of $p_{i,k}$, $\forall i\in\mathcal{N}^{\left(k\right)}$.
For this reason, we have to consider the approximation of $r_{i,k}$
and some variable change to make the objective function concave, which
is a well known problem in the sum rate maximization problem in wireless
network \cite{tan2013fast}. It is common to use change of variable
$p_{i,k}=\textrm{e}^{a_{i,k}}$ and consider the approximation $y_{k}\approx f_{k}=\sum_{i\in\mathcal{N}^{\left(k\right)}}u_{i}(\boldsymbol{a}_{k},\boldsymbol{\delta}_{k},\mathbf{S}_{k})$
with \cite{tan2013fast} 
\begin{equation}
u_{i}=\!\omega_{1}\!\log\!\left(\!\frac{s_{ii,k}\textrm{e}^{a_{i,k}}}{\sigma^{2}+\sum_{j\neq i}\delta_{j,k}s_{ji,k}\textrm{e}^{a_{j,k}}}\!\right)\!-\omega_{2}\textrm{e}^{a_{i,k}}.\label{eq:u_function-1}
\end{equation}
It is straightforward to show that $\partial^{2}f_{k}/\partial a_{i,k}^{2}<0$,
$\forall i\in\mathcal{N}^{\left(k\right)}$, thus the condition (\ref{eq:concave})
in Assumption~\ref{Assumption:F} is satisfied.

In order to perform classical gradient methods, each transmitter should
be able to evaluate 
\begin{align}
 & \frac{\partial f}{\partial a_{i,k}}\!=\omega_{1}\!-\!\!\sum_{n\in\mathcal{N}^{\left(k\right)}}\!\!\frac{\omega_{1}s_{in,k}\textrm{e}^{a_{i,k}}}{\sigma^{2}\!+\!\sum_{j\neq n}\!\delta_{j,k}s_{jn,k}\textrm{e}^{a_{j,k}}}\!-\!\omega_{2}\textrm{e}^{a_{i,k}},\label{eq:deri_theo}
\end{align}
of which the calculation requires much information, such as the cross-channel
gain $s_{in,k}$ $\forall n\in\mathcal{N}^{\left(k\right)}\setminus\left\{ i\right\} $,
as well as all the interference estimated by each active receiver.
All the channel information has to be estimated and exchanged by each
active node, which is a huge burden for the network. Therefore, we
desire to propose a distributed optimization algorithm only with the
numerical observation of utilities. The framework is distributed such
that each node can only know the local utilities of its neighbors
and of itself, see Figure~1 for more details.

\begin{figure}[h]
\begin{centering}
\includegraphics[width=0.85\columnwidth]{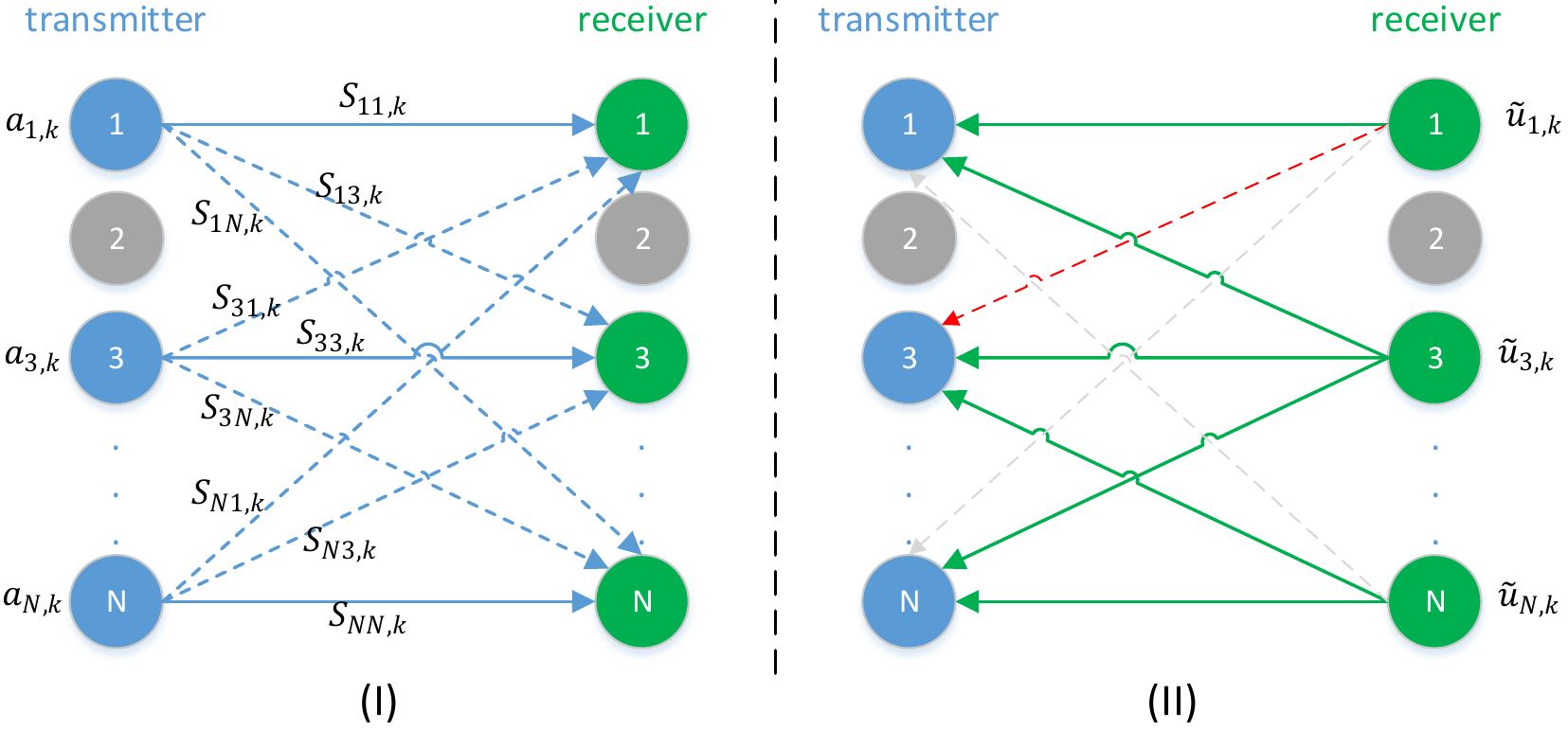} 
\par\end{centering}
\caption{\label{fig:example_ex}(I) At time $k$, link~2 is inactive, each
active transmitter $i$ communicates with its receiver with transmission
power $e^{a_{i,k}}$ and introduces interference to the other links;
(II) Each active receiver $i$ broadcast $\widetilde{u}_{i,k}$ to
its neighbors: the green links mean that there is a successful transmission
of $\widetilde{u}_{i,k}$; the red links represent the transmission
failure caused by collision or packet loss; the gray links mean that
link~1 and link~$N$ are not neighbors so exchange of $\widetilde{u}_{i,k}$
between these nodes.}
\end{figure}

\section{DOSP Learning Algorithm with Sporadic Updates\label{sec:Main-results}}

In this section, we describe our DOSP-S learning algorithm, namely,
distributed optimization algorithm using stochastic perturbation with
sporadic updates. We start with the approximation of the value of
global utility based on the collected local utilities by each active
node in Section~\ref{subsec:Estimation_f}, before the presentation
of DOSP-S in Section~\ref{subsec:Algorithm}.

\subsection{Estimation of global utility $\widetilde{f}$\label{subsec:Estimation_f}}

Recall that the global utility is 0 if no nodes are active, hence
we focus on the opposite situation. For any time-slot~$k$ such that
$n_{k}\geq1$, we consider an arbitrary active node~$i$ as reference
and denote $\widetilde{f}_{i,k}$ as the numerical value of global
utility approximated by node~$i$. If node~$i$ knows the constant
probability $q_{\textrm{r}}$ of successfully receiving $\widetilde{u}_{j,k}$
from another node~$j$, we can have an unbiased estimation of $f$
according to the following proposition.

\begin{myprop}\label{lem:aver_f_in}Suppose that Assumption~\ref{Assumption:incomp}
holds and $q_{\mathrm{r}}$ is known by all nodes, then each active
node~$i$ can estimate 
\begin{align}
{\color{blue}} & \widetilde{f}_{i,k}\left(\boldsymbol{a}_{k},\boldsymbol{\delta}_{k},\mathbf{S}_{k},\mathcal{I}^{\left(i,k\right)}\right)=\widetilde{u}_{i,k}+\frac{1}{q_{\mathrm{r}}}\sum_{j\in\mathcal{I}^{\left(i,k\right)}}\widetilde{u}_{j,k}\label{eq:f_est}
\end{align}
of which the expected value over all possible sets $\mathcal{I}^{\left(i,k\right)}$
and the additive noise $\boldsymbol{\eta}_{k}$ equals to the global
utility function, i.e., 
\begin{equation}
\mathbb{E}_{\mathcal{I},\boldsymbol{\eta}_{k}}\left(\widetilde{f}_{i,k}\left(\boldsymbol{a}_{k},\boldsymbol{\delta}_{k},\mathbf{S},\mathcal{I}^{\left(i,k\right)}\right)\right)=f\left(\boldsymbol{a}_{k},\boldsymbol{\delta}_{k},\mathbf{S}\right).\label{eq:f_inc}
\end{equation}
\end{myprop} 
\begin{IEEEproof}
Introduce $\kappa_{i,j,k}\in\left\{ 0,1\right\} $ with $\kappa_{i,j,k}=1$
if $j\in\mathcal{I}^{\left(i,k\right)}$, otherwise $\kappa_{i,j,k}=0$.
Then (\ref{eq:f_est}) can be re-written as 
\begin{align}
 & \widetilde{f}_{i,k}=\widetilde{u}_{i,k}+\frac{1}{q_{\mathrm{r}}}\sum_{j\in\mathcal{N}^{\left(k\right)}\setminus\left\{ i\right\} }\kappa_{i,j,k}\widetilde{u}_{j,k}\label{eq:f_est-1}
\end{align}
By Assumption~\ref{Assumption:incomp}, we have $\mathbb{E}(\kappa_{i,j,k})=\mathbb{P}(\kappa_{i,j,k}=1)=q_{\mathrm{r}}$.
Based on (\ref{eq:f_est-1}), we evaluate 
\begin{align}
 & \mathbb{E}_{\mathcal{I}}\left(\widetilde{f}_{i,k}\right)=\widetilde{u}_{i,k}+\frac{1}{q_{\mathrm{r}}}\mathbb{E}_{\mathcal{I}}\left(\sum_{j\in\mathcal{N}^{\left(k\right)}\setminus\left\{ i\right\} }\kappa_{i,j,k}\widetilde{u}_{j,k}\right)\nonumber \\
 & =\widetilde{u}_{i,k}+\frac{1}{q_{\mathrm{r}}}\sum_{j\in\mathcal{N}^{\left(k\right)}\setminus\left\{ i\right\} }\widetilde{u}_{j,k}\mathbb{E}(\kappa_{i,j,k})=\sum_{j\in\mathcal{N}^{\left(k\right)}}\widetilde{u}_{j,k}.\label{eq:f_aver_i}
\end{align}
Since $\mathbb{E}_{\boldsymbol{\eta}}(\widetilde{u}_{j,k})=u_{j}(\boldsymbol{a}_{k},\boldsymbol{\delta}_{k},\mathbf{S}_{k})+\mathbb{E}_{\boldsymbol{\eta}}(\eta_{j,k})=u_{j}(\boldsymbol{a}_{k},\boldsymbol{\delta}_{k},\mathbf{S}_{k})$
by Assumption~\ref{Assumption:add_noise}, we can easily get $\mathbb{E}_{\mathcal{I},\boldsymbol{\eta}_{k}}(\widetilde{f}_{i,k})=f(\boldsymbol{a}_{k},\boldsymbol{\delta}_{k},\mathbf{S}_{k})$,
which concludes the proof. 
\end{IEEEproof}
\begin{remark} In a more general case where the probability of receiving
$\widetilde{u}_{j,k}$ from different nodes is not the same, we can
have an similar estimator of $f$ with trivial extension.\end{remark}

\begin{remark} In our work, we use only the current information of
local utilities to estimate $f$ without considering any previous
utility values. Due to the stochastic environment considered in this
work, there could be a significant difference between $f\left(\boldsymbol{a},\mathbf{S}_{k-1}\right)$
and $f\left(\boldsymbol{a},\mathbf{S}_{k}\right)$ as $\mathbf{S}_{k-1}$
and $\mathbf{S}_{k}$ are independent. Hence we cannot use the previous
values of utilities and apply the compensating scheme as in \cite{defazio2014saga}.
\end{remark}

\subsection{Learning Algorithm\label{subsec:Algorithm}}

This section presents our learning algorithm DOSP-S, which is a modified
version of the DOSP algorithm in \cite{li2017distributed}. We first
introduce some important parameters to be used in our algorithm, as
presented in the following assumption.

\begin{assumption}\label{Assumption:para} (I). $\left\{ \beta_{\ell}\right\} _{\ell\geq0}$
and $\left\{ \gamma_{\ell}\right\} _{\ell\geq0}$ are positive vanishing
sequences, \emph{i.e.}, $\beta_{\ell}=\beta_{0}\ell^{-c_{1}}$ and
$\gamma_{\ell}=\gamma_{0}\ell^{-c_{2}}$, with $\beta_{0}>0$, $\gamma_{0}>0$,
$c_{1}\in\left(0.5,1\right)$, and $c_{2}\in\left(0,1-c_{1}\right]$,
such that 
\begin{equation}
\sum_{\ell=1}^{\infty}\beta_{\ell}\gamma_{\ell}=\infty\textrm{ and }\sum_{\ell=1}^{\infty}\beta_{\ell}^{2}<\infty;
\end{equation}
(II). $\left\{ \Phi_{i,k}\right\} _{i\in\mathcal{N},k\geq1}$ are
i.i.d. zero-mean random variables, there exist $\sigma_{\Phi}>0$
and $\alpha_{\Phi}>0$ such that $\mathbb{E}(\Phi_{i,k}^{2})=\sigma_{\Phi}^{2}$
and $\left|\Phi_{i,k}\right|\leq\alpha_{\Phi}$. (III). There exists
$K_{0}<\infty$ such that 
\begin{equation}
\alpha_{\Phi}\gamma_{\ell}\leq\max_{i\in\mathcal{N}}\left\{ \left|a_{i,\max}-a_{i}^{*}\right|,\left|a_{i,\min}-a_{i}^{*}\right|\right\} ,\:\forall\ell\geq K_{0}.\label{eq:small_pert}
\end{equation}
\end{assumption}

Since we have $a_{i}^{*}\in\left(a_{i,\min},a_{i,\max}\right)$, $\forall i\in\mathcal{N}$
in Assumption~\ref{Assumption:F}, such $K_{0}<\infty$ always exists
to ensure (\ref{eq:small_pert}).

Denote $\widetilde{a}_{i,k}$ as an intermediate variable. For any
active node~$i$ at time-slot~$k$, the learning algorithm is given
by 
\begin{align}
\widetilde{a}_{i,k+1} & =a_{i,k}+\widetilde{\beta}_{i,k}\Phi_{i,k}\widetilde{f}_{i,k}\left(\boldsymbol{a}_{k}+\widetilde{\boldsymbol{\gamma}}_{k}\circ\boldsymbol{\Phi}_{k},\boldsymbol{\delta}_{k},\mathbf{S}_{k}\right),\label{eq:update_i}\\
a_{i,k+1} & =\mathtt{Proj}_{i,k+1}\left(\widetilde{a}_{i,k+1}\right),\label{eq:projection}
\end{align}
in which we use the equivalent step-sizes 
\begin{equation}
\widetilde{\beta}_{i,k}=\delta_{i,k}\beta_{\ell_{i,k}},\qquad\widetilde{\gamma}_{i,k}=\delta_{i,k}\gamma_{\ell_{i,k}},
\end{equation}
where $\ell_{i,k}$ denotes the index of the step-sizes $\gamma_{\cdot}$
and $\beta_{\cdot}$ to be applied by node~$i$ at iteration~$k$
during the algorithm. In this paper, $\ell_{i,k}$ is supposed to
be generated independently and randomly by each node with 
\begin{equation}
\ell_{i,k}=\widetilde{\ell}_{i,k}+\delta_{i,k}:\textrm{ }\widetilde{\ell}_{i,k}\sim\mathcal{B}\left(k-1,q_{\mathrm{a}}\right).
\end{equation}
Notice that $\mathcal{B}$ represents Binomial distribution. We denote
$\widetilde{\boldsymbol{\beta}}_{k}=\left[\widetilde{\beta}_{1,k},\ldots,\widetilde{\beta}_{N,k}\right]^{T}$,
$\widetilde{\boldsymbol{\gamma}}_{k}=\left[\widetilde{\gamma}_{1,k},\ldots,\widetilde{\gamma}_{N,k}\right]^{T}$
and $\circ$ represents the element-wise production of two vectors.

We have to apply the projection of $\widetilde{a}_{i,k}$ as in (\ref{eq:projection}),
to ensure that the actually performed action $a_{i,k}+\widetilde{\gamma}_{i,k}\Phi_{i,k}$
belongs to the feasible set $\mathcal{A}_{i}$. The operator $\mathtt{Proj}_{i,k}$
is defined as 
\begin{align}
\mathtt{Proj}_{i,k}\left(\widetilde{a}_{i,k}\right) & =\min\left\{ \max\left\{ \widetilde{a}_{i,k},a_{i,\min}+\alpha_{\Phi}\widetilde{\gamma}_{i,k}\right\} ,\right.\nonumber \\
 & \qquad\qquad\qquad\qquad\left.a_{i,\max}-\alpha_{\Phi}\widetilde{\gamma}_{i,k}\right\} .\label{eq:projection-1}
\end{align}
Recall that $\left|\Phi_{i,k}\right|\leq\alpha_{\Phi}$ in Assumption
\ref{Assumption:para}, we have then 
\begin{equation}
a_{i,k}+\widetilde{\gamma}_{i,k}\Phi_{i,k}\in\left[a_{i,k}-\alpha_{\Phi}\widetilde{\gamma}_{i,k},a_{i,k}+\alpha_{\Phi}\widetilde{\gamma}_{i,k}\right]\subseteq\mathcal{A}_{i},
\end{equation}
which means that the actually performed action always belongs to the
feasible set.

The proposed learning algorithm is concluded in Algorithm~\ref{alg:ESSP-based-Algorithm}.
The main difference between the DOSP-S algorithm and the DOSP algorithm
in \cite{li2017distributed} comes from the network model. Since not
all nodes are active at the same time, the step-sizes $\widetilde{\beta}_{i,k}$
and $\widetilde{\gamma}_{i,k}$ are not updated simultaneously, the
analysis becomes more challenging as we will discuss in Section~\ref{subsec:Convergence-results}.

\begin{algorithm}
\caption{\label{alg:ESSP-based-Algorithm}DOSP-S for each node~$i$}
\begin{enumerate}
\item Initialize $k=1$, set step-sizes $\widetilde{\beta}_{i,k}=\delta_{i,k}\beta_{\delta_{i,k}}$
and $\widetilde{\gamma}_{i,k}=\delta_{i,k}\gamma_{\delta_{i,k}}$,
set the value $a_{i,1}$ randomly from the interval $\left[a_{i,\min}+\widetilde{\gamma}_{i,k}\alpha_{\Phi},a_{i,\max}+\widetilde{\gamma}_{i,k}\alpha_{\Phi}\right]$. 
\item If $\delta_{i,k}=1$ 
\begin{enumerate}
\item Generate a random variable $\Phi_{i,k}$, perform action with value
$\widehat{a}_{i,k}=a_{i,k}+\widetilde{\gamma}_{i,k}\Phi_{i,k}$; 
\item Estimate $\widetilde{u}_{i,k}$, broadcast this value to its active
neighbors, and receive $\widetilde{u}_{j,k}$ from active neighbors~$j\in\mathcal{I}^{\left(i,k\right)}$.
Calculate $\widetilde{f}_{i,k}$ according to (\ref{eq:f_est}), i.e.,
$\widetilde{f}_{i,k}=\widetilde{u}_{i,k}+q_{\mathrm{r}}^{-1}\sum_{j\in\mathcal{I}^{\left(i,k\right)}}\widetilde{u}_{j,k}$; 
\item Update $\widetilde{a}_{i,k+1}$ using (\ref{eq:update_i}), i.e.,
$\widetilde{a}_{i,k+1}=a_{i,k}+\widetilde{\beta}_{i,k}\Phi_{i,k}\widetilde{f}_{i,k}$. 
\end{enumerate}
\item If $\delta_{i,k}=0$, then $\widetilde{a}_{i,k+1}=a_{i,k}$. 
\item Generate $\widetilde{\ell}_{i,k+1}\sim\mathcal{B}(k,q_{\mathrm{a}})$,
set $\widetilde{\beta}_{i,k+1}=\delta_{i,k+1}\beta_{\delta_{i,k+1}+\widetilde{\ell}_{i,k+1}}$
and $\widetilde{\gamma}_{i,k+1}=\delta_{i,k+1}\gamma_{\delta_{i,k+1}+\widetilde{\ell}_{i,k+1}}$. 
\item Update $a_{i,k+1}$ using (\ref{eq:projection}), i.e., $a_{i,k+1}=\mathtt{Proj}_{i,k+1}(\widetilde{a}_{i,k+1})$. 
\item $k=k+1$, go to 2. 
\end{enumerate}
\end{algorithm}

\section{Almost Sure Convergence\label{subsec:Convergence-results}}

We investigate the convergence of Algorithm~1 in this section. We
mainly need to investigate the divergence
\begin{equation}
d_{k}=\frac{1}{N}\left\Vert \boldsymbol{a}_{k}-\boldsymbol{a}^{*}\right\Vert ^{2},\label{eq:diverg_def}
\end{equation}
which represents the distance between the actual $\boldsymbol{a}_{k}$
and the optimal point $\boldsymbol{a}^{*}$. Our aim is to prove that
$d_{k}\rightarrow0$ a.s. Compared with the original DOSP algorithm,
the main challenge of the analysis comes from the additional randomness
of the network topology, which makes the objective function completely
different and more complicated to be characterized. Moreover, the
fact that each nodes uses independent and random step-sizes also makes
the analysis challenging.

A fundamental step is to learn the relation between $d_{k+1}$ and
$d_{k}$. Similar to the analysis of stochastic approximation, we
can write (\ref{eq:update_i}) into the generalized Robbins-Monro
form \cite{kushner2012stochastic} by introducing two noise terms.
Denote $\widehat{g}_{i,k}=\widetilde{\beta}_{i,k}\Phi_{i,k}\widetilde{f}_{i,k}$
and $\overline{g}_{i,k}=\mathbb{E}_{\mathbf{S},\boldsymbol{\eta},\mathcal{I},\boldsymbol{\Phi},\boldsymbol{\delta},\boldsymbol{\ell}}(\widehat{g}_{i,k})$,
\emph{i.e}., the expected value of $\widehat{g}_{i,k}$ with respect
to (w.r.t.) $(\mathbf{S}_{k},\boldsymbol{\eta}_{k},\mathcal{I}^{\left(i,k\right)},\boldsymbol{\Phi}_{k},\boldsymbol{\delta}_{k},\boldsymbol{\ell}_{k})$,
conditioned by any $\boldsymbol{a}_{k}\in\mathcal{A}$. Rewrite (\ref{eq:update_i})
as 
\begin{align}
 & \widetilde{a}_{i,k+1}=a_{i,k}+\widehat{g}_{i,k}=a_{i,k}+\overline{g}_{i,k}+\left(\widehat{g}_{i,k}-\overline{g}_{i,k}\right)\nonumber \\
 & =a_{i,k}+\frac{\sigma_{\Phi}^{2}}{q_{\mathrm{a}}}\overline{\beta\gamma}_{i,k}\left(\frac{\partial}{\partial a_{i,k}}F\left(\boldsymbol{a}_{k}\right)+b_{i,k}\right)+e_{i,k},\label{eq:reform}
\end{align}
where we introduce 
\begin{align}
e_{i,k} & =\widehat{g}_{i,k}-\overline{g}_{i,k}.\label{eq:stoch_def}\\
b_{i,k} & =\frac{q_{\mathrm{a}}}{\sigma_{\Phi}^{2}\overline{\beta\gamma}_{k}}\overline{g}_{i,k}-\frac{\partial}{\partial a_{i,k}}F\left(\boldsymbol{a}_{k}\right);\label{eq:bias_def}\\
\overline{\beta\gamma}_{k} & =\mathbb{E}\left(\widetilde{\beta}_{i,k}\widetilde{\gamma}_{i,k}\right)=\mathbb{E}_{\boldsymbol{\delta},\boldsymbol{\ell}}\left(\delta_{i,k}\beta_{\ell_{i,k}}\gamma_{\ell_{i,k}}\right);\label{eq:step_aver}
\end{align}
Note that $e_{i,k}$ is in fact the stochastic noise indicating the
difference between the value of a single realization of $\widehat{g}_{i,k}$
and its average $\overline{g}_{i,k}$; $b_{i,k}$ represents the difference
between $\overline{g}_{i,k}$ and $\partial F/\partial a_{i,k}$.
The average step-size $\overline{\beta\gamma}_{k}$ can be evaluated
by 
\begin{align}
\overline{\beta\gamma}_{k} & =\mathbb{P}\left(\delta_{i,k}=1\right)\mathbb{E}_{\boldsymbol{\delta},\widetilde{\boldsymbol{\ell}}}\left(\delta_{i,k}\beta_{\delta_{i,k}+\widetilde{\ell}_{i,k}}\gamma_{\delta_{i,k}+\widetilde{\ell}_{i,k}}\mid\delta_{i,k}=1\right)\nonumber \\
 & =\sum_{\ell=1}^{k}\beta_{\ell}\gamma_{\ell}q_{\mathrm{a}}^{\ell}\left(1-q_{\mathrm{a}}\right)^{k-\ell}\binom{k-1}{\ell-1},\label{eq:step_aver_1}
\end{align}
which is identical for any node~$i$ at time-slot~$k$, since the
statistical property of $\delta_{i,k}$ and $\widetilde{\ell}_{i,k}$
is assumed to be same for all nodes. Similar to $\overline{\beta\gamma}_{k}$,
define the following average step-sizes that will be used in our analysis:
\begin{align}
\overline{\beta}_{k} & =\mathbb{E}_{\boldsymbol{\delta},\boldsymbol{\ell}}\left(\widetilde{\beta}_{i,k}\right),\:\overline{\gamma}_{k}=\mathbb{E}_{\boldsymbol{\delta},\boldsymbol{\ell}}\left(\widetilde{\gamma}_{i,k}\right),\:\overline{\beta^{2}}_{k}=\mathbb{E}_{\boldsymbol{\delta},\boldsymbol{\ell}}\left(\widetilde{\beta}_{i,k}^{2}\right),\nonumber \\
\overline{\gamma^{2}}_{k} & =\mathbb{E}_{\boldsymbol{\delta},\boldsymbol{\ell}}\left(\widetilde{\gamma}_{i,k}^{2}\right),\:\overline{\beta\gamma^{2}}_{k}=\mathbb{E}_{\boldsymbol{\delta},\boldsymbol{\ell}}\left(\widetilde{\beta}_{i,k}\widetilde{\gamma}_{i,k}^{2}\right)\label{eq:step_aver_def}
\end{align}
Denote $\widehat{\boldsymbol{g}}_{k}=[\widehat{g}_{1,k},\ldots,\widehat{g}_{N,k}]^{T}$,
$\overline{\boldsymbol{g}}_{k}=[\overline{g}_{1,k},\ldots,\overline{g}_{N,k},]^{T}$,
$\boldsymbol{b}_{k}=[b_{1,k},\ldots,b_{N,k}]^{T}$, $\boldsymbol{e}_{k}=[e_{1,k},\ldots,e_{N,k}]^{T}$
and $\nabla F\left(\boldsymbol{a}_{k}\right)=[\frac{\partial}{\partial a_{1}}F\left(\boldsymbol{a}_{k}\right),\ldots,\frac{\partial}{\partial a_{N}}F\left(\boldsymbol{a}_{k}\right)]^{T}$.
Then we rewrite (\ref{eq:update_i}) into $\widetilde{\boldsymbol{a}}_{k+1}=\boldsymbol{a}_{k}+\widehat{\boldsymbol{g}}_{k}$
with 
\begin{equation}
\widehat{\boldsymbol{g}}_{k}=\sigma_{\Phi}^{2}q_{\mathrm{a}}^{-1}\overline{\beta\gamma}_{k}\left(\nabla F\left(\boldsymbol{a}_{k}\right)+\boldsymbol{b}_{k}\right)+\boldsymbol{e}_{k}.\label{eq:reform_vec}
\end{equation}

Based on the above notations, we can find an upper bound of $d_{k+1}$
as a function of $d_{k}$:

\begin{myprop}\label{lem:gamma_sum-1}Introduce $\varDelta_{k}=\alpha_{\Phi}^{2}\gamma_{0}^{2}N^{-1}\sum_{i\in\mathcal{N}}\delta_{i,k}\iota_{i,k}$
with 
\begin{equation}
\iota_{i,k}=\begin{cases}
1, & \textrm{if }\widetilde{\ell}_{i,k}<K_{0}-1,\\
0, & \textrm{otherwise}.
\end{cases}\label{eq:dp}
\end{equation}
Then for any $k\geq K_{0}$, we have 
\begin{align}
d_{k+1} & \leq d_{k}+\varDelta_{k+1}+\frac{1}{N}\left\Vert \widehat{\boldsymbol{g}}_{k}\right\Vert ^{2}+\frac{2}{N}\left(\boldsymbol{a}_{k}-\boldsymbol{a}^{*}\right)^{T}\cdot\boldsymbol{e}_{k}\nonumber \\
{\color{blue}} & \:+\frac{2\sigma_{\Phi}^{2}}{q_{\mathrm{a}}N}\overline{\beta\gamma}_{k}\left(\boldsymbol{a}_{k}-\boldsymbol{a}^{*}\right)^{T}\cdot\left(\nabla F\left(\boldsymbol{a}_{k}\right)+\boldsymbol{b}_{k}\right).\label{eq:evo_d}
\end{align}
\end{myprop} 
\begin{IEEEproof}
By definition of $d_{k}$, we have 
\begin{align}
d_{k+1} & =\frac{\left\Vert \widetilde{\boldsymbol{a}}_{k+1}-\boldsymbol{a}^{*}\right\Vert ^{2}}{N}+\frac{\left\Vert \boldsymbol{a}_{k+1}-\boldsymbol{a}^{*}\right\Vert ^{2}-\left\Vert \widetilde{\boldsymbol{a}}_{k+1}-\boldsymbol{a}^{*}\right\Vert ^{2}}{N}\nonumber \\
 & \overset{\left(a\right)}{\leq}\frac{1}{N}\left\Vert \boldsymbol{a}_{k}+\widehat{\boldsymbol{g}}_{k}-\boldsymbol{a}^{*}\right\Vert ^{2}+\varDelta_{k+1}\nonumber \\
 & =d_{k}+\frac{1}{N}\left\Vert \widehat{\boldsymbol{g}}_{k}\right\Vert ^{2}+\frac{2}{N}\left(\boldsymbol{a}_{k}-\boldsymbol{a}^{*}\right)^{T}\cdot\widehat{\boldsymbol{g}}_{k}+\varDelta_{k+1}\label{eq:evo2}
\end{align}
where $(a)$ is by the following 
\begin{equation}
\frac{\left\Vert \boldsymbol{a}_{k}-\boldsymbol{a}^{*}\right\Vert ^{2}-\left\Vert \widetilde{\boldsymbol{a}}_{k}-\boldsymbol{a}^{*}\right\Vert ^{2}}{N}\leq\varDelta_{k},\quad\forall k\geq K_{0},\label{eq:proj_diff}
\end{equation}
with the proof detail in Appendix~\ref{subsec:Proof-of-inequality}.
We get (\ref{eq:evo_d}) by substituting (\ref{eq:reform_vec}) into
(\ref{eq:evo2}), which concludes the proof. 
\end{IEEEproof}
Our next step is to show the desirable properties of $\overline{\beta\gamma}_{k}$,
$\boldsymbol{b}_{k}$, $\boldsymbol{e}_{k}$, and $\varDelta_{k}$
respectively, before our main convergence result.

\begin{myprop}\label{lem:gamma_sum}We have 
\begin{equation}
\sum_{k=1}^{\infty}\overline{\beta\gamma}_{k}=\sum_{k=1}^{\infty}\beta_{k}\gamma_{k}\rightarrow\infty,\label{eq:gamma_aver_sum}
\end{equation}
\begin{equation}
\sum_{k=1}^{\infty}\overline{\beta}_{k}^{2}\leq\sum_{k=1}^{\infty}\overline{\beta^{2}}_{k}=\sum_{k=1}^{\infty}\beta_{k}^{2}<\infty.\label{eq:gamma_aversquar}
\end{equation}
\end{myprop} 
\begin{IEEEproof}
See Appendix~\ref{subsec:Proof-of-sum}. 
\end{IEEEproof}
Proposition~\ref{lem:gamma_sum} states that the average step-sizes
$\overline{\beta\gamma}_{k}$ and $\overline{\beta}_{k}$ inherit
the property of $\beta_{k}\gamma_{k}$ and $\beta_{k}$, which is
essential for the convergence of our DOSP-S learning algorithm.

\begin{mythm}\label{prop:bias}If all the assumptions are satisfied,
then for any node~$i$ and any time-slot $k$, we have 
\begin{align}
\left|b_{i,k}\right| & \leq\left(2\sigma_{\Phi}^{2}\right)^{-1}\alpha_{G}\alpha_{\Phi}^{3}w_{i,k}.\label{eq:bound_d-1}
\end{align}
with 
\begin{equation}
w_{i,k}=q_{\mathrm{a}}\overline{\beta\gamma}_{k}^{-1}\sum_{j_{1},j_{2}\in\mathcal{N}}\!\mathbb{E}_{\boldsymbol{\delta},\boldsymbol{\ell}}\!\left(\widetilde{\beta}_{i,k}\widetilde{\gamma}_{j_{1},k}\widetilde{\gamma}_{j_{2},k}\right).\label{eq:w}
\end{equation}
Furthermore, $w_{i,k}\rightarrow0$ as $k\rightarrow\infty$. Thus
$\left|b_{i,k}\right|\rightarrow0$.\end{mythm}

\begin{proof}See Appendix~\ref{subsec:Proof-of-bias}. The proof
of (\ref{eq:bound_d-1}) is mainly by the application of Taylor's
theorem and the mean value theorem. We can see that the estimation
bias of gradient $b_{i,k}$ comes from the second order term of the
objective function. The value of $\left|b_{i,k}\right|$ can be bounded
as $\mid\frac{\partial^{2}G}{\partial a_{i}\partial a_{j}}\mid$ is
bounded by Assumption~\ref{Assumption:lipc}. The proof of $w_{i,k}\rightarrow0$
is challenging, we have used Chernoff's bound to show that $\mathbb{E}_{\boldsymbol{\delta},\boldsymbol{\ell}}(\widetilde{\beta}_{i,k}\widetilde{\gamma}_{j_{1},k}\widetilde{\gamma}_{j_{2},k})$
decreases much faster than $\overline{\beta\gamma}_{k}$.\end{proof}

\begin{remark}In the case where nodes are always active, we get $w_{i,k}=N^{2}\gamma_{k}$
in our previous work \cite{li2017distributed}. We can directly have
$w_{i,k}\rightarrow0$ as $\gamma_{k}\rightarrow0$, given that $N<\infty$.
While in this paper, the analysis of $w_{i,k}$ is much more complicated
due to the asynchronous feature of the algorithm: each node maintains
a random and individual step-size $\widetilde{\gamma}_{i,k}$. In
(\ref{eq:w}), $w_{i,k}$ has complicated form of which the closed
expression are hard to derive. \end{remark}

The property of $\boldsymbol{e}_{k}$ is stated as follows:

\begin{myprop}\label{lem:stochastic_noise}If all the assumptions
are satisfied, then we have $N^{-1}\left|\sum_{k=1}^{\infty}\left(\boldsymbol{a}_{k}-\boldsymbol{a}^{*}\right)^{T}\cdot\boldsymbol{e}_{k}\right|<\infty$
a.s.\end{myprop}

\begin{proof}See Appendix~\ref{subsec:Proof-sketch-of_e}. The proof
is by applying Doob's martingale inequality, which is a suitable tool
in the framework of stochastic approximation.\end{proof}

The property of $\varDelta_{k}$ is similar to $\boldsymbol{e}_{k}$:

\begin{myprop}\label{lem:bias_projection} There exist bounded constants
$K_{1}\geq K_{0}$ and $\widetilde{C}>0$, such that $\mathbb{E}(\varDelta_{k})\leq\widetilde{C}\overline{\beta^{2}}_{k-1}$
for any $K\geq K_{1}$. Meanwhile, we have $\left|\sum_{k=K_{1}}^{\infty}\varDelta_{k}\right|<\infty$
a.s.\end{myprop}

\begin{proof}See Appendix~\ref{subsec:Proof-proj}.\end{proof}

Based on all the above results, we can finally prove the a.s. convergence
of our DOSP-S algorithm.

\begin{mythm}\label{prop:converge_rp}If all the assumptions are
satisfied, then $\boldsymbol{a}_{k}\rightarrow\boldsymbol{a}^{*}$
as $k\rightarrow\infty$ almost surely by applying Algorithm~1. \end{mythm}

\begin{proof}See Appendix \ref{subsec:Proof-sketch-of_th}. Based
on (\ref{eq:evo_d}) and our results that $N^{-1}\left|\sum_{k=1}^{\infty}\left(\boldsymbol{a}_{k}-\boldsymbol{a}^{*}\right)^{T}\cdot\boldsymbol{e}_{k}\right|<\infty$
a.s. and $\left|\sum_{k=K_{1}}^{\infty}\varDelta_{k}\right|<\infty$
a.s., we can get $\lim_{k\rightarrow\infty}(\nabla F\left(\boldsymbol{a}_{k}\right)+\boldsymbol{b}_{k})=\mathbf{0}$
a.s. with basic steps as in the framework of stochastic approximation.
Meanwhile, we have shown that $\lim_{k\rightarrow\infty}\left\Vert \boldsymbol{b}_{k}\right\Vert \rightarrow0$
in Theorem~\ref{prop:bias}. Thus, $\nabla F\left(\boldsymbol{a}_{k}\right)\rightarrow\mathbf{0}$
a.s. when $k\rightarrow\infty$. \end{proof}

\section{\label{sec:Convergence-Rate}Convergence Rate}

In this section, we investigate the speed of convergence to optimum
of the proposed learning algorithm. Specifically, we derive an upper
bound of the average divergence
\begin{equation}
D_{k}=N^{-1}\mathbb{E}\left(\left\Vert \boldsymbol{a}_{k}-\boldsymbol{a}^{*}\right\Vert ^{2}\right).\label{eq:diverg_aver}
\end{equation}
Note that the expectation is taken over all the random terms including
$\boldsymbol{a}_{k}$. An additional assumption is made as follows,
which is a common setting in the analysis of the convergence rate
\cite{nemirovski2009robust}.

\begin{assumption}\label{Assumption:strong_concave} $F\left(\boldsymbol{a}\right)$
is strongly concave, there exists $\alpha_{F}>0$ such that 
\begin{equation}
\left(\boldsymbol{a}-\boldsymbol{a}^{*}\right)^{T}\nabla F\left(\boldsymbol{a}_{k}\right)\leq-\alpha_{F}\left\Vert \boldsymbol{a}-\boldsymbol{a}^{*}\right\Vert _{2}^{2},\quad\forall\boldsymbol{a}\in\mathcal{A}.\label{eq:order1_2}
\end{equation}

\end{assumption}

As a starting point, we need to find the recurrence relation between
$D_{k+1}$ and $D_{k}$. 
\begin{lem}
\label{lem:inequality_D}Under Assumptions 1-7, $D_{k+1}$ is upper
bounded by a function of $D_{k}$ as $k\geq K_{1}$, i.e., 
\begin{align}
 & D_{k+1}\leq\left(1-A\theta_{k}\right)D_{k}+B\psi_{k}\sqrt{D_{k}}+C\upsilon_{k},\label{eq:dk+1_dk}
\end{align}
with bounded constants $A=2\sigma_{\Phi}^{2}\alpha_{F}$, $B=\alpha_{G}\alpha_{\Phi}^{3}$,
$C=\widetilde{C}+(1+q_{\mathrm{r}}^{-1})\lambda\sigma_{\Phi}^{2}\sigma_{\eta}^{2}+(1+(2q_{\mathrm{r}}^{-1}+5)\lambda+(q_{\mathrm{r}}^{-1}+5)\lambda^{2}+\lambda^{3}L^{2})\sigma_{\Phi}^{2}\sigma_{\boldsymbol{a}}^{2}$
and vanishing sequences 
\begin{equation}
\theta_{k}=q_{\mathrm{a}}^{-1}\overline{\beta\gamma}_{k},\qquad\upsilon_{k}=\overline{\beta^{2}}_{k},\label{eq:para_others}
\end{equation}
\begin{equation}
\psi_{k}=2N\overline{\beta\gamma}_{k}\overline{\gamma}_{k}+\overline{\beta\gamma^{2}}_{k}+\left(N-1\right)^{2}\overline{\beta}_{k}\overline{\gamma^{2}}_{k}.\label{eq:PHI_def}
\end{equation}
\end{lem}
\begin{IEEEproof}
See Appendix \ref{subsec:Proof-of-Lemma_D}. 
\end{IEEEproof}
Based on (\ref{eq:dk+1_dk}), we can derive the upper bounds of $D_{k}$,
as stated as follows.

\begin{mythm}\label{prop:rate} Introduce $K_{2}$ the minimum value
of $k\geq K_{1}$ such that $\theta_{k}<A^{-1}$. Define the following
parameters: 
\begin{equation}
\chi_{k}=\frac{1}{\theta_{k}}-\frac{\psi_{k+1}^{2}\theta_{k}}{\psi_{k}^{2}\theta_{k+1}^{2}},\epsilon_{1}=\max_{k\geq K_{0}}\chi_{k},\epsilon_{2}=\max_{k\geq K_{0}}\frac{\theta_{k}\upsilon_{k}}{\psi_{k}^{2}},\label{eq:chi_w}
\end{equation}
\begin{equation}
\varpi_{k}=\frac{1}{\theta_{k}}-\frac{\upsilon_{k+1}}{\upsilon_{k}\theta_{k+1}},\epsilon_{3}=\max_{k\geq K_{0}}\varpi_{k},\epsilon_{4}=\max_{k\geq K_{0}}\frac{\psi_{k}^{2}}{\theta_{k}\upsilon_{k}}.\label{eq:small_p}
\end{equation}
If $\epsilon_{1}<A$ and $\epsilon_{2}<\infty$, then 
\begin{equation}
D_{k}\leq\vartheta^{2}\psi_{k}^{2}\theta_{k}^{-2},\:\forall k\geq K_{0},\label{eq:Dk_up1}
\end{equation}
with 
\begin{equation}
\vartheta\!\geq\!\max\!\left\{ \!\frac{\theta_{K_{0}}\sqrt{D_{K_{0}}}}{\psi_{K_{0}}},\frac{B+\!\sqrt{B^{2}\!+4C\epsilon_{2}\left(A-\epsilon_{1}\right)}}{2\left(A-\epsilon_{1}\right)}\!\right\} .\label{eq:para_0}
\end{equation}
If $\epsilon_{3}<A$ and $\epsilon_{4}<\infty$, then 
\begin{equation}
D_{k}\leq\varrho^{2}\upsilon_{k}\theta_{k}^{-1},\:\forall k\geq K_{0},\label{eq:Dk_up2}
\end{equation}
with 
\begin{equation}
\varrho\!\geq\!\max\!\left\{ \!\sqrt{\!\frac{D_{K_{0}}\!\theta_{K_{0}}}{\upsilon_{K_{0}}}},\frac{B\sqrt{\epsilon_{4}}\!+\!\sqrt{B^{2}\epsilon_{4}\!+\!4C\!\left(A-\epsilon_{3}\right)}}{2\left(A-\epsilon_{3}\right)}\!\right\} .\label{eq:para_0-1}
\end{equation}
\end{mythm} 
\begin{IEEEproof}
See Appendix~\ref{subsec:proof_convergencerate}. 
\end{IEEEproof}
The general form of the upper bounds of $D_{k}$ looks complicated
mainly due to the averaged parameters $\theta_{k}$, $\upsilon_{k}$,
and $\psi_{k}$. The conditions that $\epsilon_{1}<A$ and $\epsilon_{3}<A$
can be checked numerically for any fix value of $N$, $q_{\mathrm{a}}$,
and any sequences $\left\{ \beta_{\ell}\right\} _{\ell\geq0}$ and
$\left\{ \gamma_{\ell}\right\} _{\ell\geq0}$. Here we focus on the
theoretical analysis of: \emph{i}) decreasing order of $D_{k}$; \emph{ii})
convergence of $\epsilon_{2}$ and \emph{$\epsilon_{4}$}; \emph{iii})
convergence of $\epsilon_{1}$ and \emph{$\epsilon_{3}$}.

We propose first the upper bounds of $\upsilon_{k}\theta_{k}^{-1}$,
$\psi_{k}^{2}\theta_{k}^{-2}$,$\theta_{k}\upsilon_{k}\psi_{k}^{-2}$,
and $\psi_{k}^{2}\theta_{k}^{-1}\upsilon_{k}^{-1}$ in the following
lemma. 
\begin{lem}
\label{lem:bounds_collect}Consider $\beta_{\ell}=\beta_{0}k^{-c_{1}}$
and $\gamma_{\ell}=\gamma_{0}k^{-c_{2}}$. For any $\xi\in\left(0,1\right)$
and $\xi'>0$, there exists $K'$ such that $\forall k\geq K'$ ,\textup{
\begin{align}
\frac{\upsilon_{k}}{\theta_{k}} & <\left(1+\xi'\right)\left(1-\xi\right)^{-2c_{1}}q_{\mathrm{a}}\beta_{0}\gamma_{0}^{-1}\left(q_{\mathrm{a}}k\right)^{-c_{1}+c_{2}},\label{eq:rate_1}\\
\frac{\psi_{k}^{2}}{\theta_{k}^{2}} & <\frac{\left(1+\xi'\right)^{2}}{\left(1-\xi\right)^{2c_{1}+4c_{2}}}\left(\lambda+1\right)^{4}\gamma_{0}^{2}\left(q_{\mathrm{a}}k\right)^{-2c_{2}},\label{eq:rate_2}\\
\frac{\psi_{k}^{2}}{\theta_{k}\upsilon_{k}} & <\frac{\left(1+\xi'\right)^{2}\left(\lambda+1\right)^{4}\gamma_{0}^{3}}{\left(1-\xi\right)^{2c_{1}+4c_{2}}q_{\mathrm{a}}\beta_{0}}\left(q_{\mathrm{a}}\left(k-1\right)+1\right)^{c_{1}-3c_{2}},\label{eq:const_2}\\
\frac{\theta_{k}\upsilon_{k}}{\psi_{k}^{2}} & <\frac{\left(1+\xi'\right)^{2}q_{\mathrm{a}}\beta_{0}}{\left(1-\xi\right)^{3c_{1}+c_{2}}\lambda^{4}\gamma_{0}^{3}}\left(q_{\mathrm{a}}\left(k-1\right)+1\right)^{-c_{1}+3c_{2}}.\label{eq:const_1}
\end{align}
}Both $\xi$ and $\xi'$ can be arbitrarily close to $0$ as $K'\rightarrow\infty$. 
\end{lem}
\begin{IEEEproof}
See Appendix~\ref{subsec:Proof-of-Lemma_order}. 
\end{IEEEproof}
From Lemma~\ref{lem:bounds_collect}, we can clearly see that the
decreasing order of $\upsilon_{k}\theta_{k}^{-1}$ and of $\psi_{k}^{2}\theta_{k}^{-2}$
is the same as that of $\beta_{k}\gamma_{k}\propto k^{-c_{1}+c_{2}}$
and of $\gamma_{k}^{2}\propto k^{-2c_{2}}$, respectively. According
to (\ref{eq:const_2}) and (\ref{eq:const_1}), we find that $\lim_{k\rightarrow\infty}\theta_{k}\upsilon_{k}\psi_{k}^{-2}<\infty$
and $\epsilon_{2}<\infty$ if and only if $c_{1}\geq3c_{2}$, whereas
$\lim_{k\rightarrow\infty}\psi_{k}^{2}\theta_{k}^{-1}\upsilon_{k}^{-1}<\infty$
and $\epsilon_{4}<\infty$ if and only if $c_{1}\leq3c_{2}$.

The convergence of $\chi_{k}$ and $\varpi_{k}$ are discussed in
the following lemma, which is more challenging to be justified. 
\begin{lem}
\label{lem:constant_term}Consider $\beta_{\ell}=\beta_{0}k^{-c_{1}}$
and $\gamma_{\ell}=\gamma_{0}k^{-c_{2}}$, then both $\chi_{k}$ and
$\varpi_{k}$ are bounded. There always exist $\beta_{0}<\infty$
and $\gamma_{0}<\infty$ such that $\epsilon_{1}=\max_{k\geq K_{0}}\chi_{k}<A$
and $\epsilon_{3}=\max_{k\geq K_{0}}\varpi_{k}<A$. 
\end{lem}
\begin{IEEEproof}
See Appendix~\ref{subsec:Proof_last}. 
\end{IEEEproof}
The following theorem concludes our discussion.

\begin{mythm}\label{thm:rate_example}Consider $\beta_{\ell}=\beta_{0}k^{-c_{1}}$
and $\gamma_{\ell}=\gamma_{0}k^{-c_{2}}$, if the value of $\beta_{0}\gamma_{0}<\infty$
is large enough, then there exists $\Xi<\infty$, such that 
\begin{equation}
D_{k}\leq\Xi{\color{blue}q_{\mathrm{r}}^{-1}}\left(q_{\mathrm{a}}k\right)^{-\min\left\{ 2c_{2},c_{1}-c_{2}\right\} },\quad\forall k\geq K_{2}.\label{eq:bound_example}
\end{equation}
As $c_{1}=0.75$ and $c_{2}=0.25$, the upper bound of $D_{k}$ has
the optimum decreasing order, i.e\emph{.}, $D_{k}=O(q_{\mathrm{r}}^{-1}\left(q_{\mathrm{a}}k\right)^{-0.5})$.
\end{mythm} 
\begin{IEEEproof}
 {From Lemma~\ref{lem:inequality_D} we find that
$q_{\mathrm{r}}$ only affects the constant term $C$, such that $C=O\left(q_{\mathrm{r}}^{-1}\right)$.
We also have the upper bound of $D_{k}$ is dominated by a linear
function of $C$ when $C$ is large by Theorem~\ref{prop:rate}.
Thus $D_{k}=O\left(q_{\mathrm{r}}^{-1}\right)$.} Then we consider
three situations separately.

Case 1: $3c_{2}<c_{1}$. We have $\epsilon_{2}<\infty$ and $\epsilon_{4}=\infty$.
Then only (\ref{eq:Dk_up1}) is valid with $\vartheta<\infty$. We
have $D_{k}\rightarrow O(q_{\mathrm{r}}^{-1}\left(q_{\mathrm{a}}k\right)^{-2c_{2}})$.

Case 2: $3c_{2}>c_{1}$. We have $\epsilon_{4}<\infty$ and $\epsilon_{2}=\infty$.
Then only (\ref{eq:Dk_up2}) is valid with $\varrho<\infty$. We have
$D_{k}\rightarrow O(q_{\mathrm{r}}^{-1}\left(q_{\mathrm{a}}k\right)^{-c_{1}+c_{2}})$.

Case 3: $3c_{2}=c_{1}$. Both (\ref{eq:Dk_up1}) and (\ref{eq:Dk_up2})
are valid, we have $D_{k}\rightarrow O(\left(q_{\mathrm{a}}k\right)^{-2c_{2}})$
or $D_{k}\rightarrow O(q_{\mathrm{r}}^{-1}\left(q_{\mathrm{a}}k\right)^{-c_{1}+c_{2}})$.

As $c_{1}+c_{2}\leq1$ and $c_{2}>0.5$, it is easy to deduce that
$\min\left\{ 2c_{2},c_{1}-c_{2}\right\} \leq0.5$, where the equality
holds only if $c_{1}=0.75$ and $c_{2}=0.25$. 
\end{IEEEproof}
\begin{remark}  {From $\left|\frac{\partial^{2}}{\partial a_{i}\partial a_{j}}G\left(\boldsymbol{a},\boldsymbol{\delta}\right)\right|\leq\alpha_{G}$
in Assumption~\ref{Assumption:lipc}, one have $\left|\frac{\partial^{2}}{\partial a_{i}\partial a_{j}}F\left(\boldsymbol{a}\right)\right|\leq\alpha_{G}$
by definition (\ref{eq:G}), which means that $\left\Vert \nabla F\left(\boldsymbol{a}\right)-\nabla F\left(\boldsymbol{a}'\right)\right\Vert \leq N\alpha_{G}\left\Vert \boldsymbol{a}-\boldsymbol{a}'\right\Vert $
and $\left|F\left(\boldsymbol{a}\right)-F\left(\boldsymbol{a}'\right)\right|\leq N\alpha_{G}\left\Vert \boldsymbol{a}-\boldsymbol{a}'\right\Vert ^{2}/2$
for any $\boldsymbol{a},\boldsymbol{a}'\in\mathcal{A}$. Applying
Jensen's inequality, we can then derive the upper bound of optimization
error 
\begin{align*}
 & F\left(\boldsymbol{a}^{*}\right)-\mathbb{E}(F(\frac{1}{K}\sum_{k=1}^{K}\boldsymbol{a}_{k}))\leq\frac{1}{K}\sum_{k=1}^{K}\left(F\left(\boldsymbol{a}^{*}\right)-\mathbb{E}\left(F\left(\boldsymbol{a}_{k}\right)\right)\right)\\
 & \leq\frac{N\alpha_{G}}{2K}\sum_{k=1}^{K}\left\Vert \boldsymbol{a}_{k}-\boldsymbol{a}^{*}\right\Vert ^{2}\leq\frac{N\alpha_{G}}{2K}\sum_{k=1}^{K}\Xi'k^{-0.5}\\
 & \leq N\alpha_{G}\Xi'K^{-0.5}=O\left(K^{-0.5}\right).
\end{align*}
Clearly, the optimization error achieved by our proposed solution
is $O\left(K^{-0.5}\right)$ when the objective function is smooth
and strongly concave.} \end{remark}

\section{Numerical illustration\label{sec:Numerical-illustration}}

This section presents some numerical examples to further illustrate
our results.

We consider the power control problem described in Section~\ref{sec:Motivating-Example}
Recall that the network is composed of $N$ transmitter-receiver link,
each link has a probability $q_{\mathrm{a}}$ to be active at any
time-slot, with the the local utility function defined in (\ref{eq:u_function-1}).
The power gain is $s_{ij}=\left|h_{ij}\right|^{2}$, where $h_{ij}$,
the channel between transmitter~$i$ and receiver~$j$, follows
Gaussian distribution with variance $\sigma_{ii}^{2}=1$ (direct channel)
and $\sigma_{ij}^{2}=0.1$ (cross channel). The rest of the system
parameters are set as $\sigma^{2}=0.2$, $\omega_{1}=20$ and $\omega_{2}=1$.
In the proposed learning algorithm, the random perturbation $\Phi_{i,k}\in\left\{ -1,1\right\} $
is generated as a symmetric Bernoulli random variable.

First, we set $\beta_{\ell}=0.025\ell^{-0.75}$, $\gamma_{\ell}=10\ell^{-0.25}$
and consider $N=50$, $q_{\mathrm{a}}=0.05$ and $q_{\mathrm{r}}=1$.
We perform a single simulation to show the convergence of the action
$a_{i,k}$ performed by all nodes. The result is shown in Figure~\ref{fig:a_1},
which contains $N=50$ curves. We can see that all the curves turn
to be close to each other and converge after sufficient number of
iterations. Note that the optimum value $a_{i}^{*}$ should be identical
for all nodes in this example, as the global utility function has
a symmetric shape and the random coefficients are generated using
the same mechanism. Because of the sparse activity of nodes, the final
index of iteration look large. In fact, the average times of update
performed by each node is $2500$ when $k=5\times10^{4}$ and $q_{\mathrm{a}}=0.05$.

Second, we investigate the influence of fact that nodes have incomplete
knowledge of local utilities. We set $q_{\mathrm{r}}=\left\{ 1,0.5,0.1\right\} $
and the other parameters remain unchanged. In order to show that our
algorithm converges to optimum, we consider also the ideal gradient
descent method as a reference, with the exact partial derivative obtained
by (\ref{eq:deri_theo}). As we have discussed in Section~\ref{sec:Motivating-Example},
this ideal method requires much informational exchange and may be
infeasible in practice. Figure~\ref{fig:F_50} shows the evolution
of the average global utility by 100 independent simulations. From
the oscillation of the curves, we can see that the objective function
is quite sensitive to the stochastic channel and not easy to optimize.
We find that the global utility tends to the maximum value in average
in all cases. The value of $q_{\mathrm{r}}$ does not seriously affect
the convergence speed, when an active node has only $10\%$ opportunity
to know the local utility of another active node. The two curves corresponding
to $q_{\mathrm{r}}=1$ and $q_{\mathrm{r}}=0.5$ are quite close.

\begin{figure}[h]
\begin{centering}
\includegraphics[width=0.9\columnwidth]{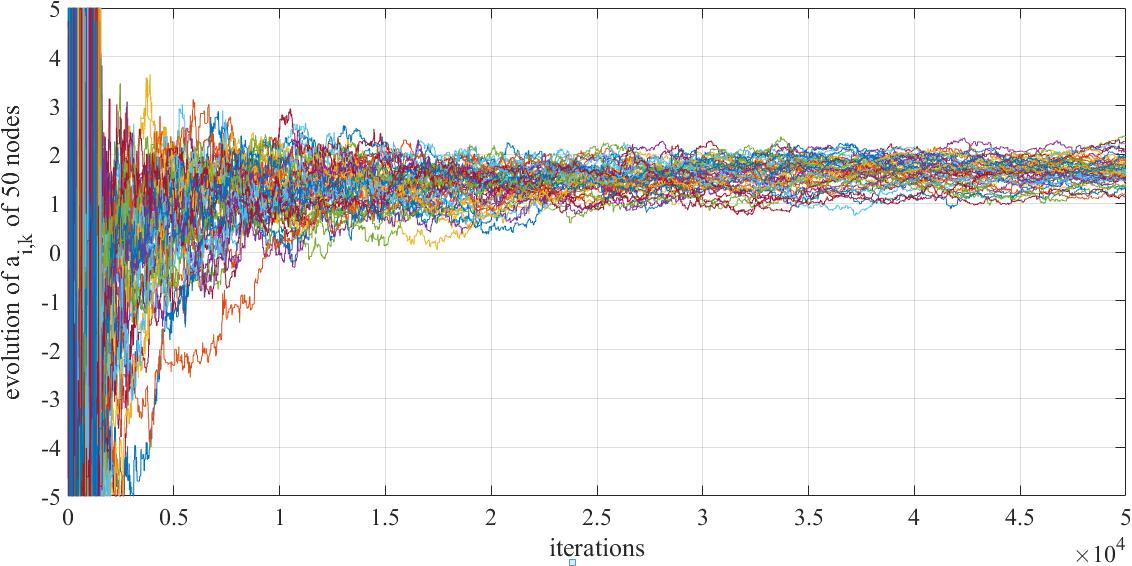} 
\par\end{centering}
\caption{Evolution of action $a_{i,k}$ of 50 nodes, obtained by a single simulation\label{fig:a_1} }
\end{figure}

\begin{figure}[h]
\begin{centering}
\includegraphics[width=0.9\columnwidth]{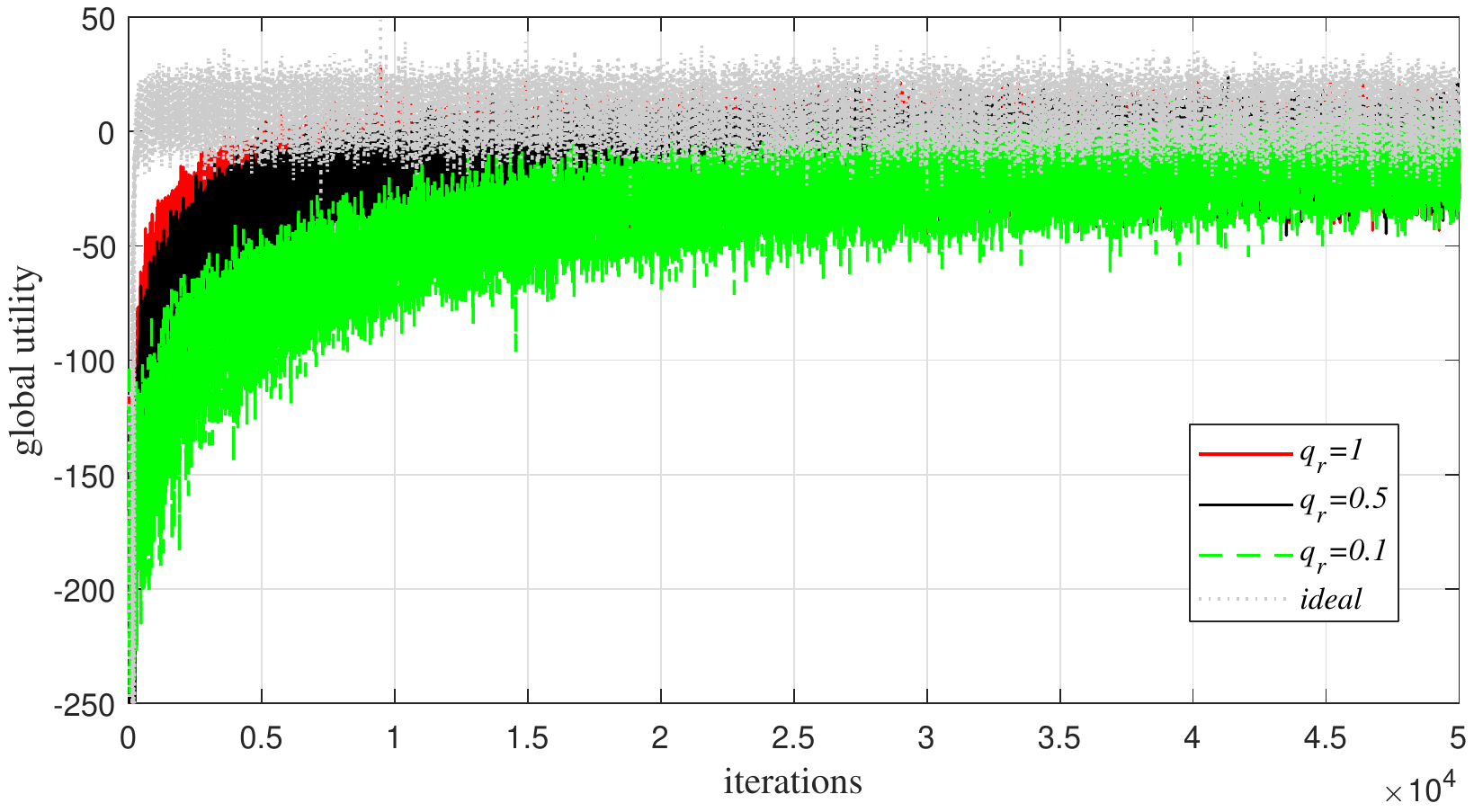} 
\par\end{centering}
\caption{Evolution of the global utility function $F$, with $N=50$, $q_{\mathrm{a}}=0.05$
and $q_{\mathrm{r}}\in\left\{ 1,0.5,0.1\right\} $, average results
by 100 simulations\label{fig:F_50} }
\end{figure}

Finally, we are interested in the evolution of the average divergence
$D_{k}=N^{-1}\mathbb{E}(\left\Vert \boldsymbol{a}_{k}-\boldsymbol{a}^{*}\right\Vert ^{2})$.
We still use $\beta_{\ell}=\frac{\ell^{-0.75}}{10Nq_{\mathrm{a}}}$
and $\gamma_{\ell}=10\ell^{-0.25}$, while consider various values
of $N$, $q_{\mathrm{a}}$, and $q_{\mathrm{r}}$. The result is presented
in Figure~\ref{fig:diverg_1}. Note that the optimal point $\boldsymbol{a}^{*}$
is approximately obtained by applying the ideal gradient method. We
plot an additional curve $\Xi k^{-0.5}$ in Figure~\ref{fig:diverg_1},
which represents the theoretical convergence rate when $\beta_{\ell}\propto\ell^{-0.75}$
and $\gamma_{\ell}\propto\ell^{-0.25}$, under the condition that
the objective function is strongly concave. Note that $\Xi=50$ is
set to facilitate the visual comparison of different curves, as we
only focus on the asymptotic decreasing speed.

We can see that all the tails of the curves in Figure~\ref{fig:diverg_1}
are approximately parallel, which means that $D_{k}\rightarrow O\left(k^{-0.5}\right)$
with different values of $N$, $q_{\mathrm{a}}$, and $q_{\mathrm{r}}$.
We can also see the influence of $q_{\mathrm{r}}$ on $D_{k}$ with
fixed $N=50$ and $q_{\mathrm{a}}=0.5$: compared with the case where
$q_{\mathrm{r}}=1$, $D_{k}$ converges slightly slower as $q_{\mathrm{r}}=0.5$,
which confirms our discussion of Figure~\ref{fig:F_50}.

\begin{figure}[h]
\begin{centering}
\includegraphics[width=0.9\columnwidth]{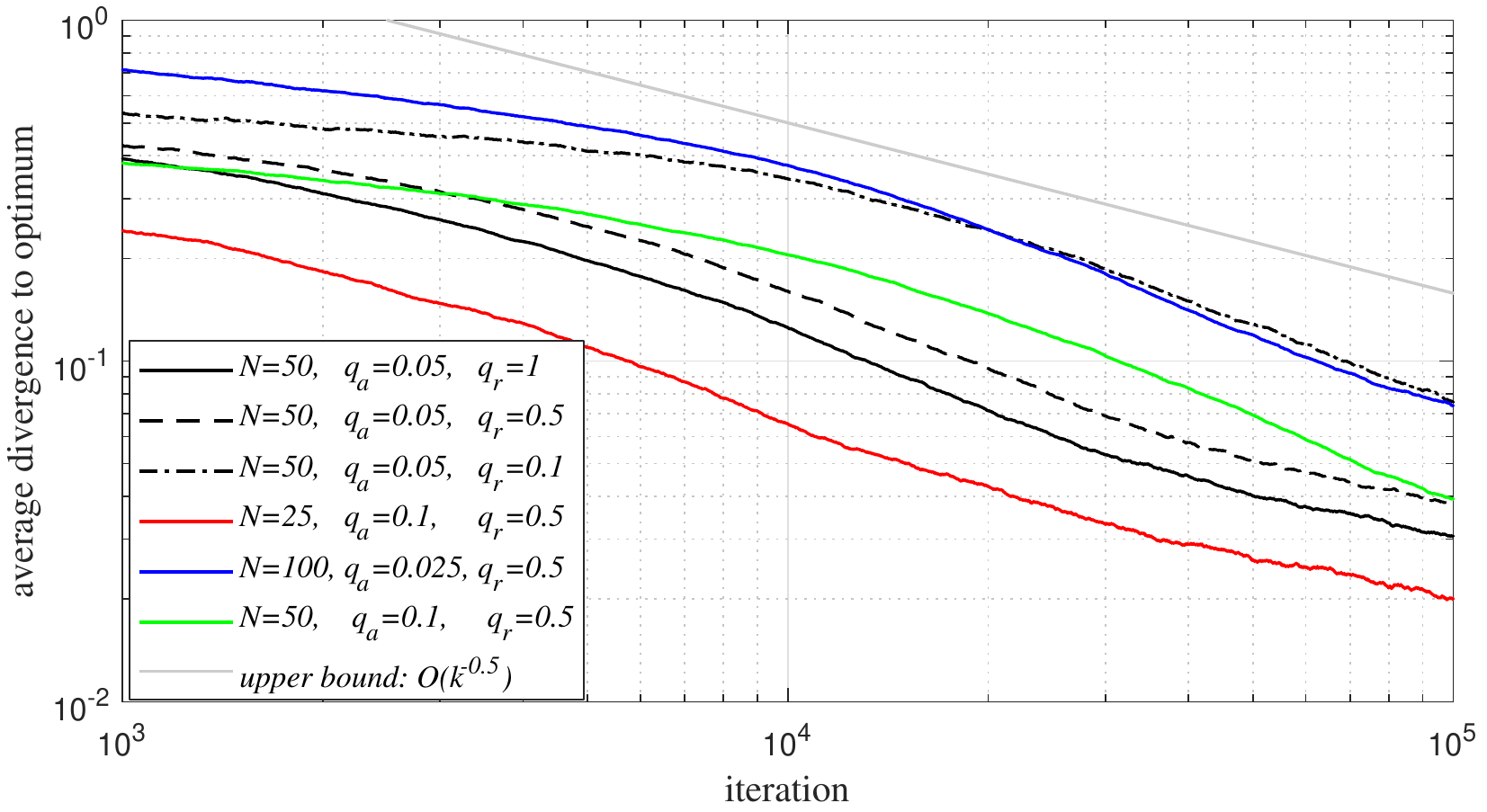} 
\par\end{centering}
\caption{Evolution of $D_{k}$ by 50 simulations. We use $\beta_{\ell}=\frac{\ell^{-0.75}}{10Nq_{\mathrm{a}}}$
and $\gamma_{\ell}=10k^{-0.25}$ and consider various setting of $(N,q_{\mathrm{a}},q_{\mathrm{r}})$.
We also use an addition curve $O(k^{-0.5})$ to present the theoretical
upper bound of $D_{k}$.  {\label{fig:diverg_1} }}
\end{figure}

\section{Conclusion\label{sec:Conclusion}}

In this paper we consider a distributed derivative-free optimization
problem in a large network with sparse activity. We propose a learning
algorithm to make each active node control its action to maximize
the global utility function of the network, which is also affected
by some stochastic process. The algorithm is performed only based
on the numerical observation of the global utility rather than its
gradient. We prove the almost surely convergence of the algorithm
with the tools of stochastic approximation and concentration inequalities.
The analysis is challenging because of the asynchronous feature of
the network. We have also derived the convergence rate of the proposed
algorithm. We provide simulation results to corroborate our claim.
Both theoretical and numerical results show that our derivative-free
learning algorithm can converge at a rate $O\left(k^{-0.5}\right)$.

\appendix

\subsection{\label{subsec:Proof-of-inequality}Proof of inequality (\ref{eq:proj_diff})}

In this proof, we investigate the property of the projection (\ref{eq:projection-1}).
Define $\mathcal{C}_{i}^{*}=\left[a_{i,\min}+\alpha_{\Phi}\gamma_{K_{0}},a_{i,\max}-\alpha_{\Phi}\gamma_{K_{0}}\right]$,
then (\ref{eq:small_pert}) implies that $a_{i}^{*}\in\mathcal{C}_{i}^{*}$,
$\forall i\in\mathcal{N}$. Similarly, let $\mathcal{C}_{i,k}=\left[a_{i,\min}+\alpha_{\Phi}\widetilde{\gamma}_{i,k},a_{i,\max}-\alpha_{\Phi}\widetilde{\gamma}_{i,k}\right]$
for any $i\in\mathcal{N}$ and $k\geq K_{0}$. Due to the fact that
$\widetilde{\gamma}_{i,k}=\delta_{i,k}\gamma_{\delta_{i,k}+\widetilde{\ell}_{i,k}}$
is random, there is not always $a_{i}^{*}\in\mathcal{C}_{i,k}$. Different
cases must be considered depending on the values of $\delta_{i,k}$
and $\widetilde{\ell}_{i,k}$.

Case 1, $\delta_{i,k}=0$. We have $\mathcal{C}_{i,k}=[a_{i,\min}+0,a_{i,\max}+0]$,
thus $a_{i}^{*}\in\mathcal{C}_{i,k}$. By definition, we also have
$a_{i,k}=\mathtt{Proj}_{i,k}(\widetilde{a}_{i,k})\in\mathcal{C}_{i,k}$.
Since the projection decreases the Euclidean distance between $\widetilde{a}_{i,k}$
and $a_{i}^{*}$ if $\widetilde{a}_{i,k}\notin\mathcal{C}_{i,k}$.
it is easy to show that $\left|\mathtt{Proj}_{i,k}\left(\widetilde{a}_{i,k}\right)-a_{i}^{*}\right|\leq\left|\widetilde{a}_{i,k}-a_{i}^{*}\right|$.
Hence $\left(a_{i,k}-a_{i}^{*}\right)^{2}-\left(\widetilde{a}_{i,k}-a_{i}^{*}\right)^{2}\leq0.$

Case 2, $\delta_{i,k}=1$ and $\widetilde{\ell}_{i,k}\geq K_{0}-1$.
Then $\widetilde{\gamma}_{i,k}\leq\gamma_{K_{0}}$ and $\mathcal{C}_{i}^{*}\subseteq\mathcal{C}_{i,k}$.
Thus $a_{i}^{*}\in\mathcal{C}_{i,k}$ as in Case 1 and $\left(a_{i,k}-a_{i}^{*}\right)^{2}-\left(\widetilde{a}_{i,k}-a_{i}^{*}\right)^{2}\leq0.$

Case 3, $\delta_{i,k}=1$ and $\widetilde{\ell}_{i,k}<K_{0}-1$. This
case is complicated as $a_{i}^{*}\notin\mathcal{C}_{i,k}$. There
exist two possible situations: $a_{i}^{*}\in[a_{i,\min},a_{i,\min}+\alpha_{\Phi}\widetilde{\gamma}_{i,k}]$
or $a_{i}^{*}\in[a_{i,\max},a_{i,\max}-\alpha_{\Phi}\widetilde{\gamma}_{i,k}]$.
Here we mainly consider $a_{i}^{*}\in[a_{i,\min},a_{i,\min}+\alpha_{\Phi}\widetilde{\gamma}_{i,k}]$
as the analysis of the other situation is similar. We still need to
discuss the possible value of $a_{i,k}$ in separate situations: \emph{i}).
if $\widetilde{a}_{i,k}\in\mathcal{C}_{i,k}$, we have $\mathtt{Proj}_{i,k}(\widetilde{a}_{i,k})=\widetilde{a}_{i,k}$,
then $(a_{i,k}-a_{i}^{*})^{2}-(\widetilde{a}_{i,k}-a_{i}^{*})^{2}=0$;
\emph{ii}). if $\widetilde{a}_{i,k}>a_{i,\max}-\alpha_{\Phi}\widetilde{\gamma}_{i,k}$,
then $\mathtt{Proj}_{i,k}(\widetilde{a}_{i,k})=a_{i,\max}-\alpha_{\Phi}\widetilde{\gamma}_{i,k}<\widetilde{a}_{i,k}$.
Meanwhile $a_{i}^{*}\leq a_{i,\min}+\alpha_{\Phi}\widetilde{\gamma}_{i,k}<a_{i,\max}-\alpha_{\Phi}\widetilde{\gamma}_{i,k}$.
We get $(a_{i,k}-a_{i}^{*})^{2}-(\widetilde{a}_{i,k}-a_{i}^{*})^{2}<0$;
\emph{iii}). if $\widetilde{a}_{i,k}<a_{i,\min}+\alpha_{\Phi}\widetilde{\gamma}_{i,k}$,
then $\mathtt{Proj}_{i,k}(\widetilde{a}_{i,k})=a_{i,\min}+\alpha_{\Phi}\widetilde{\gamma}_{i,k}$.
We have $\left(a_{i,k}-a_{i}^{*}\right)^{2}-\left(\widetilde{a}_{i,k}-a_{i}^{*}\right)^{2}\leq\left(a_{i,k}-a_{i}^{*}\right)^{2}=\left(a_{i,\min}+\alpha_{\Phi}\widetilde{\gamma}_{i,k}-a_{i}^{*}\right)^{2}\leq\alpha_{\Phi}^{2}\widetilde{\gamma}_{i,k}^{2}\leq\alpha_{\Phi}^{2}\gamma_{0}^{2}$.
In summary, $\left(a_{i,k}-a_{i}^{*}\right)^{2}-\left(\widetilde{a}_{i,k}-a_{i}^{*}\right)^{2}\leq\alpha_{\Phi}^{2}\gamma_{0}^{2}$
in Case 3.

Based on the above discussions, we have the following bound to conclude
the three cases $\left(a_{i,k}-a_{i}^{*}\right)^{2}-\left(\widetilde{a}_{i,k}-a_{i}^{*}\right)^{2}\leq\alpha_{\Phi}^{2}\gamma_{0}^{2}\delta_{i,k}\iota_{i,k}$
with $\iota_{i,k}$ defined in (\ref{eq:dp}). Finally, we get $\frac{\left\Vert \boldsymbol{a}_{k}-\boldsymbol{a}^{*}\right\Vert ^{2}\!-\!\left\Vert \widetilde{\boldsymbol{a}}_{k}-\boldsymbol{a}^{*}\right\Vert ^{2}}{N}\!\leq\!\frac{\alpha_{\Phi}^{2}\gamma_{0}^{2}}{N}\!\sum_{i\in\mathcal{N}}\delta_{i,k}\iota_{i,k}=\varDelta_{k}$which
concludes the proof.

\subsection{\label{subsec:Proof-of-sum}Proof of Lemma~\ref{lem:gamma_sum}}

We first present an important lemma with its proof in Appendix~F
of our previous work \cite{li2019matrix}. 
\begin{lem}
\label{lem:sum_equi}Consider any sequence $\left\{ x_{k}\right\} $
and let $\overline{x}_{k}=\sum_{\ell=1}^{k}x_{\ell}\left(1-p\right)^{\ell}p^{k-\ell}\binom{k-1}{\ell-1}$
with $p\in\left[0,1\right]$, we always have $\sum_{k=1}^{\infty}\overline{x}_{k}=\sum_{k=1}^{\infty}x_{k}.$
\end{lem}
Replace $x_{k}$ by $\beta_{k}\gamma_{k}$ and $p$ by $1-q_{\textrm{a}}$,
we get that $\sum_{k=1}^{\infty}\overline{\beta\gamma}_{k}=\sum_{k=1}^{\infty}\beta_{k}\gamma_{k}$,
then (\ref{eq:gamma_aver_sum}) can be proved as $\sum_{k=1}^{\infty}\beta_{k}\gamma_{k}\rightarrow\infty$;
Replace $x_{k}$ by $\beta_{k}^{2}$ and $p$ by $1-q_{\textrm{a}}$,
we have $\sum_{k=1}^{\infty}\overline{\beta^{2}}_{k}=\sum_{k=1}^{\infty}\beta_{k}^{2}$.
We can finally justify (\ref{eq:gamma_aversquar}) with the assumption
that $\sum_{k=1}^{\infty}\beta{}_{k}^{2}<\infty$.

\subsection{\label{subsec:Proof-of-bias}Proof of Theorem~\ref{prop:bias}}

This proof contains two parts. We first show that (\ref{eq:bound_d-1})
is an upper bound of $\left|b_{i,k}\right|$ in Appendix~\ref{subsec:Proof-c1},
then we prove that this upper bound is vanishing in Appendix~\ref{subsec:Proof-c2}.

\subsubsection{Proof of (\ref{eq:bound_d-1}) \label{subsec:Proof-c1}}

As $b_{i,k}$ describes the difference between $\overline{g}_{i,k}$
and $\partial F\left(\boldsymbol{a}_{k}\right)/\partial a_{i}$, we
start with the derivation of $\overline{g}_{i,k}$ by successively
taking the expectation of $\widehat{g}_{i,k}$ w.r.t. multiple stochastic
terms $\left(\mathbf{S},\mathcal{I},\boldsymbol{\eta},\boldsymbol{\Phi},\boldsymbol{\delta},\boldsymbol{\ell}\right)$,
which makes the analysis complicated. By definition, we have 
\begin{align}
 & \overline{g}_{i,k}\overset{}{=}\mathbb{E}_{\mathbf{S},\mathcal{I},\boldsymbol{\eta},\boldsymbol{\Phi},\boldsymbol{\delta},\boldsymbol{\ell}}\left(\widetilde{\beta}_{i,k}\Phi_{i,k}\widetilde{f}_{i,k}\left(\boldsymbol{a}_{k}+\widetilde{\boldsymbol{\gamma}}_{k}\circ\boldsymbol{\Phi}_{k},\boldsymbol{\delta}_{k},\mathbf{S}_{k}\right)\right)\nonumber \\
 & \overset{\left(a\right)}{=}\mathbb{E}_{\boldsymbol{\Phi},\boldsymbol{\delta},\boldsymbol{\ell}}\left(\widetilde{\beta}_{i,k}\Phi_{i,k}\mathbb{E}_{\mathbf{S},\mathcal{I},\boldsymbol{\eta}}\left(\widetilde{f}_{i,k}\left(\boldsymbol{a}_{k}+\widetilde{\boldsymbol{\gamma}}_{k}\circ\boldsymbol{\Phi}_{k},\boldsymbol{\delta}_{k},\mathbf{S}_{k}\right)\right)\right)\nonumber \\
 & \overset{\left(b\right)}{=}\mathbb{E}_{\boldsymbol{\Phi},\boldsymbol{\delta},\boldsymbol{\ell}}\left(\widetilde{\beta}_{i,k}\Phi_{i,k}\mathbb{E}_{\mathbf{S}}\left(f\left(\boldsymbol{a}_{k}+\widetilde{\boldsymbol{\gamma}}_{k}\circ\boldsymbol{\Phi}_{k},\boldsymbol{\delta}_{k},\mathbf{S}_{k}\right)\right)\right)\nonumber \\
 & \overset{(c)}{=}\mathbb{E}_{\mathbf{\boldsymbol{\Phi}},\boldsymbol{\delta},\boldsymbol{\ell}}\left(\widetilde{\beta}_{i,k}\Phi_{i,k}G\left(\boldsymbol{a}_{k}+\widetilde{\boldsymbol{\gamma}}_{k}\circ\boldsymbol{\Phi}_{k},\boldsymbol{\delta}_{k}\right)\right)\nonumber \\
 & \overset{\left(d\right)}{=}\mathbb{E}_{\mathbf{\boldsymbol{\Phi}},\boldsymbol{\delta},\boldsymbol{\ell}}\!\!\left(\!\widetilde{\beta}_{i,k}\Phi_{i,k}\!\!\left(\!\!G\left(\boldsymbol{a}_{k},\boldsymbol{\delta}_{k}\right)+\!\sum_{j\in\mathcal{N}}\!\widetilde{\gamma}_{j,k}\Phi_{j,k}\frac{\partial G}{\partial a_{j}}\!\left(\boldsymbol{a}_{k},\boldsymbol{\delta}_{k}\right)\!\!\right)\right.\nonumber \\
 & +\!\left.\frac{\widetilde{\beta}_{i,k}\Phi_{i,k}}{2}\!\!\sum_{j_{1},j_{2}\in\mathcal{N}}\!\widetilde{\gamma}_{j_{1},k}\Phi_{j_{1},k}\widetilde{\gamma}_{j_{2},k}\Phi_{j_{2},k}\frac{\partial^{2}G\left(\mathring{\boldsymbol{a}}_{k},\boldsymbol{\delta}_{k}\right)}{\partial a_{j_{1}}\partial a_{j_{2}}}\!\!\right)\!\!,\label{eq:g_aver1}
\end{align}
where $\left(a\right)$ holds as the stochastic term $\widetilde{\beta}_{i,k}\Phi_{i,k}$
generated during the DOSP-S algorithm is independent of $\left(\mathbf{S}_{k},\mathcal{I}^{\left(i,k\right)},\boldsymbol{\eta}_{k}\right)$
caused by the system environment, the unsuccessful packet transmission
and the measurement noise; $(b)$ is by taking expectation of $\widetilde{f}_{i,k}$
w.r.t. $\left(\mathcal{I}^{\left(i,k\right)},\boldsymbol{\eta}_{k}\right)$,
which has already been solved in Proposition~\ref{lem:aver_f_in};
$\left(c\right)$ is by taking expectation of $f$ w.r.t. $\mathbf{S}$,
recall that $G\left(\boldsymbol{a},\boldsymbol{\delta}\right)=\mathbb{E}_{\mathbf{S}}\left(f\left(\boldsymbol{a},\boldsymbol{\delta},\mathbf{S}\right)\right)$
by definition (\ref{eq:G}); $(d)$ comes from the extension of $G\left(\boldsymbol{a}_{k}+\widetilde{\boldsymbol{\gamma}}_{k}\circ\boldsymbol{\Phi}_{k},\boldsymbol{\delta}_{k}\right)$
by applying Taylor's theorem and mean-valued theorem, \emph{i.e}.,
there exists $\mathring{\boldsymbol{a}}_{k}=\left[\mathring{a}_{1,k},\ldots\mathring{a}_{N,k}\right]^{T}$
with $\mathring{a}_{i,k}\in\left(a_{i,k},a_{i,k}+\widetilde{\gamma}_{i,k}\Phi_{i,k}\right)$,
$\forall i\in\mathcal{N}$, such that $\left(d\right)$ can be satisfied.

We should continue the derivation in (\ref{eq:g_aver1}) by considering
the expectation w.r.t. $\left(\boldsymbol{\Phi}_{k},\boldsymbol{\delta}_{k},\boldsymbol{\ell}_{k}\right)$.
We have 
\begin{equation}
\mathbb{E}_{\mathbf{\boldsymbol{\Phi}},\boldsymbol{\delta},\boldsymbol{\ell}}\left(\widetilde{\beta}_{i,k}\Phi_{i,k}G\left(\boldsymbol{a}_{k},\boldsymbol{\delta}_{k}\right)\right)=0,\label{eq:g_aver2}
\end{equation}
as $\Phi_{i,k}$ is independent of $(\boldsymbol{a}_{k},\boldsymbol{\delta}_{k},\widetilde{\beta}_{i,k})$
and $\mathbb{E}_{\boldsymbol{\Phi}}(\Phi_{i,k})=0$ by Assumption~\ref{Assumption:para}.
Meanwhile, 
\begin{align}
 & \mathbb{E}_{\mathbf{\boldsymbol{\Phi}},\boldsymbol{\delta},\boldsymbol{\ell}}\left(\widetilde{\beta}_{i,k}\Phi_{i,k}\sum_{j\in\mathcal{N}}\widetilde{\gamma}_{j,k}\Phi_{j,k}\frac{\partial G}{\partial a_{j}}\left(\boldsymbol{a}_{k},\boldsymbol{\delta}_{k}\right)\right)\nonumber \\
 & \overset{\left(a\right)}{=}\sigma_{\Phi}^{2}\mathbb{E}_{\boldsymbol{\delta},\boldsymbol{\ell}}\left(\delta_{i,k}^{2}\beta_{\ell_{i,k}}\gamma_{\ell_{i,k}}\frac{\partial G}{\partial a_{i}}\left(\boldsymbol{a}_{k},\boldsymbol{\delta}_{k}\right)\right)+0\nonumber \\
 & \overset{\left(b\right)}{=}\sigma_{\Phi}^{2}\mathbb{P}\left(\delta_{i,k}=1\right)\mathbb{E}_{\boldsymbol{\delta},\boldsymbol{\ell}}\!\left(\!\beta_{\ell_{i,k}}\gamma_{\ell_{i,k}}\left.\frac{\partial G}{\partial a_{i}}\left(\boldsymbol{a}_{k},\boldsymbol{\delta}_{k}\right)\right|\delta_{i,k}=1\!\right)\!+0\nonumber \\
 & \overset{\left(c\right)}{=}\sigma_{\Phi}^{2}\mathbb{E}_{\widetilde{\boldsymbol{\ell}}}\!\left(\!\beta_{1+\widetilde{\ell}_{i,k}}\gamma_{1+\widetilde{\ell}_{i,k}}\!\right)\!\mathbb{E}_{\boldsymbol{\delta}}\!\!\left(\!\!\left.\frac{\partial G}{\partial a_{i}}\!\left(\boldsymbol{a}_{k},\!\boldsymbol{\delta}_{k}\right)\right|\!\delta_{i,k}\!=\!1\!\!\right)\!\mathbb{P}\!\left(\delta_{i,k}=1\right)\nonumber \\
 & \overset{\left(d\right)}{=}\sigma_{\Phi}^{2}\mathbb{E}_{\widetilde{\boldsymbol{\ell}}}\!\left(\!\beta_{1+\widetilde{\ell}_{i,k}}\gamma_{1+\widetilde{\ell}_{i,k}}\!\right)\!\mathbb{E}_{\boldsymbol{\delta}}\!\!\left(\!\frac{\partial G}{\partial a_{i}}\left(\boldsymbol{a}_{k},\boldsymbol{\delta}_{k}\right)\!\!\right)\nonumber \\
 & \overset{\left(e\right)}{=}\sigma_{\Phi}^{2}q_{\mathrm{a}}^{-1}\overline{\beta\gamma}_{k}\frac{\partial F}{\partial a_{i}}\left(\boldsymbol{a}_{k}\right),\label{eq:g_aver3}
\end{align}
in which $\left(a\right)$ is again by Assumption~\ref{Assumption:para},
\emph{i.e.}, $\mathbb{E}_{\boldsymbol{\Phi}}(\Phi_{i,k}^{2})=\sigma_{\Phi}^{2}$
and $\mathbb{E}_{\boldsymbol{\Phi}}(\Phi_{i,k}\Phi_{j,k})=0$ $\forall j\neq i$;
$\left(b\right)$ comes from $\mathbb{E}_{\boldsymbol{\delta},\boldsymbol{\ell}}\left(\delta_{i,k}^{2}\beta_{\ell_{i,k}}\gamma_{\ell_{i,k}}\frac{\partial G}{\partial a_{i}}\left(\boldsymbol{a}_{k},\boldsymbol{\delta}_{k}\right)\mid\delta_{i,k}=0\right)=0$;
$\left(c\right)$ is by the independence of $\boldsymbol{\delta}_{k}$
and $\widetilde{\boldsymbol{\ell}}$; $\left(d\right)$ holds as $\mathbb{E}_{\boldsymbol{\delta}}\!\!\left(\!\frac{\partial G}{\partial a_{i}}\!\left(\boldsymbol{a}_{k},\boldsymbol{\delta}_{k}\right)\!\right)\!\!=\mathbb{E}_{\boldsymbol{\delta}}\left(\left.\frac{\partial G}{\partial a_{i}}\left(\boldsymbol{a}_{k},\boldsymbol{\delta}_{k}\right)\right|\delta_{i,k}=1\right)\mathbb{P}\left(\delta_{i,k}=1\right),$
note that $\frac{\partial G}{\partial a_{i}}\left(\boldsymbol{a}_{k},\boldsymbol{\delta}_{k}\right)=0$
in the case where $\delta_{i,k}=0$ meaning that $G$ is not a function
of $a_{i,k}$; $\left(e\right)$ comes from $\overline{\beta\gamma}_{k}=\mathbb{P}\left(\delta_{i,k}=1\right)\mathbb{E}_{\widetilde{\boldsymbol{\ell}}}(\beta_{1+\widetilde{\ell}_{i,k}}\gamma_{1+\widetilde{\ell}_{i,k}})$
and from the relation between $F$ and $G$ discussed in (\ref{eq:F_G}),
we have $\frac{\partial F}{\partial a_{i}}\left(\boldsymbol{a}_{k}\right)=\frac{\partial}{\partial a_{i}}\left(\sum_{\boldsymbol{\delta}_{k}\in\mathcal{D}}q_{\mathrm{a}}^{n_{k}}\left(1-q_{\mathrm{a}}\right)^{N-n_{k}}G\left(\boldsymbol{a}_{k},\boldsymbol{\delta}_{k}\right)\right)=\!\sum_{\boldsymbol{\delta}_{k}\in\mathcal{D}}\!q_{\mathrm{a}}^{n_{k}}\!\left(1-q_{\mathrm{a}}\right)^{N-n_{k}}\!\frac{\partial}{\partial a_{i}}G\left(\boldsymbol{a}_{k},\boldsymbol{\delta}_{k}\right)\!=\!\mathbb{E}_{\boldsymbol{\delta}}\!\left(\!\frac{\partial G}{\partial a_{i}}\left(\boldsymbol{a}_{k},\boldsymbol{\delta}_{k}\right)\!\right).$Substituting
(\ref{eq:g_aver2}) and (\ref{eq:g_aver3}) into (\ref{eq:g_aver1}),
we get $\overline{g}_{i,k}=\sigma_{\Phi}^{2}q_{\mathrm{a}}^{-1}\overline{\beta\gamma}_{k}\left(\frac{\partial F}{\partial a_{i}}\left(\boldsymbol{a}_{k}\right)+b_{i,k}\right)$
with the bias term $b_{i,k}=$ 
\begin{equation}
\sum_{j_{1},j_{2}\in\mathcal{N}}\!\!\mathbb{E}_{\mathbf{\boldsymbol{\Phi}},\boldsymbol{\delta},\boldsymbol{\ell}}\!\left(\!\frac{q_{\mathrm{a}}\widetilde{\beta}_{i,k}\widetilde{\gamma}_{j_{1},k}\widetilde{\gamma}_{j_{2},k}\Phi_{i,k}\Phi_{j_{1},k}\Phi_{j_{2},k}}{2\sigma_{\Phi}^{2}\overline{\beta\gamma}_{k}}\frac{\partial^{2}G\left(\widetilde{\boldsymbol{a}}_{k},\boldsymbol{\delta}_{k}\right)}{\partial a_{j_{1}}\partial a_{j_{2}}}\!\right)\label{eq:bound_bik}
\end{equation}
As $\left|\frac{\partial^{2}G\left(\widetilde{\boldsymbol{a}}_{k},\boldsymbol{\delta}_{k}\right)}{\partial a_{j_{1}}\partial a_{j_{2}}}\right|\leq\alpha_{G}$
(by Assumption~3) and $\left|\Phi_{i,k}\right|\leq\alpha_{\Phi}$,
$\forall i\in\mathcal{N}$ (by Assumption~5), it is straightforward
to get 
\begin{align*}
 & \left|b_{i,k}\right|\leq\frac{\alpha_{\Phi}^{3}\alpha_{G}}{2\sigma_{\Phi}^{2}}\!\frac{\sum_{j_{1},j_{2}\in\mathcal{N}}\!\mathbb{E}_{\boldsymbol{\delta},\boldsymbol{\ell}}\!\left(\widetilde{\beta}_{i,k}\widetilde{\gamma}_{j_{1},k}\widetilde{\gamma}_{j_{2},k}\right)}{q_{\mathrm{a}}^{-1}\overline{\beta\gamma}_{k}}=\frac{\alpha_{\Phi}^{3}\alpha_{G}}{2\sigma_{\Phi}^{2}}w_{i,k}.
\end{align*}
Therefore, $b_{i,k}$ in (\ref{eq:bound_bik}) can be bounded by (\ref{eq:bound_d-1})
with $w_{i,k}$ defined in (\ref{eq:w}), which concludes the first
part of the proof.

\subsubsection{Proof of $\left|b_{i,k}\right|\rightarrow0$ \label{subsec:Proof-c2}}

Our next target is to show $w_{i,k}\rightarrow0$, from which we can
directly get $\left|b_{i,k}\right|\rightarrow0$. The proof is quite
challenging, as $w_{i,k}$ contains a summation of $N^{2}$ terms
of expectation whose closed form expression are hard to obtain. Moreover,
the denominator of $w_{i,k}$ is vanishing, \emph{i.e.}, $\overline{\beta\gamma}_{k}\rightarrow0$.

Denote $\mathcal{N}_{-i}=\mathcal{N}\setminus\left\{ i\right\} $,
we evaluate the numerator of $w_{i,k}$: 
\begin{align}
 & \sum_{j_{1},j_{2}\in\mathcal{N}}\!\mathbb{E}_{\boldsymbol{\delta},\boldsymbol{\ell}}\!\left(\widetilde{\beta}_{i,k}\widetilde{\gamma}_{j_{1},k}\widetilde{\gamma}_{j_{2},k}\right)\!=\!\sum_{\stackrel{j_{1},j_{2}\in\mathcal{N}_{-i}}{j_{1}\neq j_{2}}}\!\mathbb{E}_{\boldsymbol{\delta},\boldsymbol{\ell}}\!\left(\widetilde{\beta}_{i,k}\widetilde{\gamma}_{j_{1},k}\widetilde{\gamma}_{j_{2},k}\right)\nonumber \\
 & +\sum_{j\in\mathcal{N}_{-i}}\mathbb{E}_{\boldsymbol{\delta},\boldsymbol{\ell}}\left(2\widetilde{\beta}_{i,k}\widetilde{\gamma}_{i,k}\widetilde{\gamma}_{j,k}+\widetilde{\beta}_{i,k}\widetilde{\gamma}_{j,k}^{2}\right)+\mathbb{E}_{\boldsymbol{\delta},\boldsymbol{\ell}}\!\left(\widetilde{\beta}_{i,k}\widetilde{\gamma}_{i,k}^{2}\right)\nonumber \\
 & =\!\left(N\!-1\right)\!\left(\!\left(N\!-2\right)\overline{\beta}_{k}\overline{\gamma}_{k}^{2}\!+\!2\overline{\beta\gamma}_{k}\overline{\gamma}_{k}\!+\overline{\beta}_{k}\overline{\gamma^{2}}_{k}\!\right)\!+\overline{\beta\gamma^{2}}_{k},\label{eq:step_aver_sum}
\end{align}
where $\overline{\beta\gamma^{2}}_{k}$, $\overline{\gamma}_{k}$,
$\overline{\gamma^{2}}_{k}$, and $\overline{\beta}_{k}$ are defined
in (\ref{eq:step_aver_def}). From (\ref{eq:step_aver_sum}) and the
fact that $\overline{\gamma}_{k}^{2}\leq\overline{\gamma^{2}}_{k}$,
$w_{i,k}$ can be bounded by 
\begin{align}
w_{i,k} & \leq\frac{\overline{\beta\gamma^{2}}_{k}+2\left(N-1\right)\overline{\beta\gamma}_{k}\overline{\gamma}_{k}+\left(N-1\right)^{2}\overline{\beta}_{k}\overline{\gamma^{2}}_{k}}{q_{\mathrm{a}}^{-1}\overline{\beta\gamma}_{k}}\nonumber \\
 & <2\lambda\overline{\gamma}_{k}+\frac{\overline{\beta\gamma^{2}}_{k}+\left(N-1\right)^{2}\overline{\beta}_{k}\overline{\gamma^{2}}_{k}}{q_{\mathrm{a}}^{-1}\overline{\beta\gamma}_{k}},\label{eq:w_bound1}
\end{align}
note that $\left(N-1\right)q_{\mathrm{a}}<Nq_{\mathrm{a}}=\lambda$.
The following lemma is useful to find upper bounds of $\overline{\gamma}_{k}$,
$\overline{\beta\gamma^{2}}_{k}$ and $\overline{\beta}_{k}\overline{\gamma^{2}}_{k}$. 
\begin{lem}
\label{lem:bound_chern}Consider an arbitrary positive decreasing
sequence $\left\{ z_{k}\right\} $ and an arbitrary $0<\xi<1$. Denote
\begin{align}
p_{k,\xi} & =\exp\left(-2^{-1}\xi^{2}q_{\mathrm{a}}\left(k-1\right)\right),\\
\overline{k}_{\xi} & =\left\lfloor \left(1-\xi\right)q_{\mathrm{a}}\left(k-1\right)\right\rfloor +2.
\end{align}
Then we have 
\begin{align}
\mathbb{E}_{\boldsymbol{\delta},\boldsymbol{\ell}}\left(\delta_{i,k}z_{\ell_{i,k}}\right) & \leq q_{\mathrm{a}}\left(p_{k,\xi}z_{1}+z_{\overline{k}_{\xi}}\right).
\end{align}
\end{lem}
\begin{proof}We have, 
\begin{align}
 & \mathbb{E}_{\boldsymbol{\delta},\boldsymbol{\ell}}\left(\delta_{i,k}z_{\ell_{i,k}}\right)=\mathbb{P}\left(\delta_{i,k}=1\right)\mathbb{E}_{\boldsymbol{\delta},\boldsymbol{\ell}}(z_{\widetilde{\ell}_{i,k}+1}\mid\delta_{i,k}=1)\nonumber \\
 & =q_{\mathrm{a}}\sum_{\ell=0}^{k-1}\mathbb{P}\left(\widetilde{\ell}_{i,k}=\ell\right)z_{\ell+1}\overset{\left(a\right)}{\leq}q_{\mathrm{a}}z_{1}\mathbb{P}\left(\widetilde{\ell}_{i,k}\leq\overline{k}_{\xi}-2\right)\nonumber \\
 & +q_{\mathrm{a}}\gamma_{\overline{k}_{\xi}}\mathbb{P}\left(\widetilde{\ell}_{i,k}\geq\overline{k}_{\xi}-1\right)\overset{\left(b\right)}{<}q_{\mathrm{a}}\left(p_{k,\xi}z_{1}+z_{\overline{k}_{\xi}}\right),\label{eq:aver2-1}
\end{align}
in which $(a)$ is by the fact that $\gamma_{\ell}$ is a decreasing
sequence; $(b)$ is obtained by using Chernoff Bound, \emph{i.e.},
\begin{align}
\mathbb{P}\left(\widetilde{\ell}_{i,k}\leq\overline{k}_{\xi}-2\right) & =\mathbb{P}\left(\widetilde{\ell}_{i,k}\leq\left\lfloor \left(1-\xi\right)q_{\mathrm{a}}\left(k-1\right)\right\rfloor \right)\nonumber \\
 & \leq e^{-\frac{1}{2}\xi^{2}\mathbb{E}\left(\widetilde{\ell}_{i,k}\right)}=p_{k,\xi},
\end{align}
and by $\mathbb{P}(\widetilde{\ell}_{i,k}\geq\overline{k}_{\xi}-1)<1$,
which concludes the proof. \end{proof}

Applying Lemma~\ref{lem:bound_chern}, we can obtain the following
bounds 
\begin{equation}
\begin{array}{cc}
\overline{\gamma}_{k}<q_{\mathrm{a}}(p_{k,\xi}\gamma_{1}+\gamma_{\overline{k}_{\xi}}); & \overline{\gamma^{2}}_{k}<q_{\mathrm{a}}(p_{k,\xi}\gamma_{1}^{2}+\gamma_{\overline{k}_{\xi}}^{2});\\
\overline{\beta}_{k}<q_{\mathrm{a}}(p_{k,\xi}\beta_{1}+\beta_{\overline{k}_{\xi}}); & \overline{\beta\gamma^{2}}_{k}<q_{\mathrm{a}}(p_{k,\xi}\beta_{1}\gamma_{1}^{2}+\beta_{\overline{k}_{\xi}}\gamma_{\overline{k}_{\xi}}^{2}).
\end{array}\label{eq:bounds_aver1}
\end{equation}
As $p_{k,\xi}$, $\beta_{k}$, and $\gamma_{k}$ are vanishing, (\ref{eq:bounds_aver1})
implies that $\overline{\gamma}_{k}\rightarrow0$, $\overline{\gamma^{2}}_{k}\rightarrow0$,
$\overline{\beta}_{k}\rightarrow0$, and $\overline{\beta\gamma^{2}}_{k}\rightarrow0$.

Applying the upper bounds in (\ref{eq:bounds_aver1}), we have 
\begin{align}
 & \overline{\beta\gamma^{2}}_{k}+\left(N-1\right)^{2}\overline{\beta}_{k}\overline{\gamma^{2}}_{k}<q_{\mathrm{a}}\left(p_{k,\xi}\beta_{1}\gamma_{1}^{2}+\beta_{\overline{k}_{\xi}}\gamma_{\overline{k}_{\xi}}^{2}\right)\nonumber \\
 & \quad\qquad+\left(N-1\right)^{2}q_{\mathrm{a}}^{2}\left(p_{k,\xi}\beta_{1}+\beta_{\overline{k}_{\xi}}\right)\left(p_{k,\xi}\gamma_{1}^{2}+\gamma_{\overline{k}_{\xi}}^{2}\right)\nonumber \\
 & <\left(\lambda^{2}\left(p_{k,\xi}+2\right)+q_{\mathrm{a}}\right)\beta_{1}\gamma_{1}^{2}p_{k,\xi}+\left(\lambda^{2}+q_{\mathrm{a}}\right)\beta_{\overline{k}_{\xi}}\gamma_{\overline{k}_{\xi}}^{2}\nonumber \\
 & <\left(3\lambda^{2}+q_{\mathrm{a}}\right)\beta_{1}\gamma_{1}^{2}p_{k,\xi}+\left(\lambda^{2}+q_{\mathrm{a}}\right)\beta_{\overline{k}_{\xi}}\gamma_{\overline{k}_{\xi}}^{2},\label{eq:w_bound2}
\end{align}
where the upper bound is by $\gamma_{\overline{k}_{\xi}}<\gamma_{1}$
and $\beta_{\overline{k}_{\xi}}<\beta_{1}$, as $\overline{k}_{\xi}=\left\lfloor \left(1-\xi\right)q\left(k-1\right)\right\rfloor +2>1$.

Meanwhile, thanks to the fact that $\beta_{k}\gamma_{k}$ is a convex
function of $k$, we can apply Jensen's inequality to get the lower
bound 
\begin{align}
\overline{\beta\gamma}_{k} & =q_{\mathrm{a}}\mathbb{E}_{\widetilde{\ell}}\left(\beta_{1+\widetilde{\ell}_{i,k}}\gamma_{1+\widetilde{\ell}_{i,k}}\right)\geq q_{\mathrm{a}}\beta_{\overline{k}'}\gamma_{\overline{k}'},\label{eq:low_step}
\end{align}
in which we denote $\overline{k}'=1+\mathbb{E}\left(\widetilde{\ell}_{i,k}\right)=1+q_{\mathrm{a}}\left(k-1\right)$.
Note that $\beta_{\overline{k}'}$ and $\gamma_{\overline{k}'}$ represent
functions of $\overline{k}'\in\mathbb{R}^{+}$, e.g., $\beta_{\overline{k}'}=\beta_{0}(\overline{k}')^{-c_{1}}$.
Here we slightly abuse the notation as $\{\beta_{\ell}\}$ and $\{\gamma_{\ell}\}$
are initially defined as sequences with integer index.

From (\ref{eq:w_bound1}), (\ref{eq:bounds_aver1}), (\ref{eq:w_bound2}),
and (\ref{eq:low_step}), we have $w_{i,k}<\Omega_{k}$ with 
\begin{align}
\Omega_{k} & =\left(3\lambda^{2}+q_{\mathrm{a}}\right)\beta_{1}\gamma_{1}^{2}\frac{p_{k,\xi}}{\beta_{\overline{k}'}\gamma_{\overline{k}'}}+\left(\lambda^{2}+q_{\mathrm{a}}\right)\frac{\beta_{\overline{k}_{\xi}}\gamma_{\overline{k}_{\xi}}^{2}}{\beta_{\overline{k}'}\gamma_{\overline{k}'}}\nonumber \\
 & \qquad\qquad\qquad\qquad+2\lambda q_{\mathrm{a}}\left(\gamma_{1}p_{k,\xi}+\gamma_{\overline{k}_{\xi}}\right).\label{eq:omega}
\end{align}

The last step is to show that $\Omega_{k}\rightarrow0$ considering
$\beta_{k}=\beta_{0}k^{-c_{1}}$ and $\gamma_{k}=\gamma_{0}k^{-c_{2}}$.
Since $\lambda<\infty$, $q_{\mathrm{a}}\leq1$, $p_{k,\xi}\rightarrow0$
and $\gamma_{\overline{k}_{\xi}}\rightarrow0$, we mainly need to
check whether $\frac{p_{k,\xi}}{\beta_{\overline{k}'}\gamma_{\overline{k}'}}$
and $\frac{\beta_{\overline{k}_{\xi}}\gamma_{\overline{k}_{\xi}}^{2}}{\beta_{\overline{k}'}\gamma_{\overline{k}'}}$
are vanishing. In fact, we have 
\begin{align*}
\lim_{k\rightarrow\infty}\frac{p_{k,\xi}}{\beta_{\overline{k}'}\gamma_{\overline{k}'}} & =\lim_{k\rightarrow\infty}\frac{\exp\left(-2^{-1}\xi^{2}q_{\mathrm{a}}k\right)}{\beta_{0}\gamma_{0}\left(1+\left\lfloor q_{\mathrm{a}}\left(k-1\right)\right\rfloor \right)^{-c_{1}-c_{2}}}=0,
\end{align*}
since the exponential term decreases much faster than $k^{-c_{1}-c_{2}}$.
Meanwhile, we have 
\begin{align}
 & \frac{\beta_{\overline{k}_{\xi}}\gamma_{\overline{k}_{\xi}}^{2}}{\beta_{\overline{k}'}\gamma_{\overline{k}'}}=\frac{\beta_{0}\gamma_{0}\left(\left\lfloor \left(1-\xi\right)q_{\mathrm{a}}\left(k-1\right)\right\rfloor +2\right)^{-c_{1}-c_{2}}}{\beta_{0}\gamma_{0}\left(q_{\mathrm{a}}\left(k-1\right)+1\right)^{-c_{1}-c_{2}}}\gamma_{\overline{k}_{\xi}}\nonumber \\
 & \overset{\left(a\right)}{<}\!\frac{\left(\left(1-\xi\right)q_{\mathrm{a}}\!\left(k-1\right)\!+1\right)^{-c_{1}-c_{2}}}{\left(q_{\mathrm{a}}\left(k-1\right)+1\right)^{-c_{1}-c_{2}}}\gamma_{\overline{k}_{\xi}}\!\!\overset{\left(b\right)}{<}\!\frac{\gamma_{\overline{k}_{\xi}}}{\left(1-\xi\right)^{c_{1}+c_{2}}},\label{eq:bound_ra}
\end{align}
where $\left(a\right)$ is by $\left\lfloor x\right\rfloor >x-1$,
$\forall x>0$; $\left(b\right)$ holds for any $\xi\in\left(0,1\right)$
and $k\geq1$, as 
\begin{align*}
\frac{\left(1-\xi\right)q_{\mathrm{a}}\left(k-1\right)+1}{q_{\mathrm{a}}\left(k-1\right)+1} & =1-\xi+\frac{\xi}{q_{\mathrm{a}}\left(k-1\right)+1}>1-\xi.
\end{align*}
From (\ref{eq:bound_ra}), we finally have 
\begin{equation}
\lim_{k\rightarrow\infty}\frac{\beta_{\overline{k}_{\xi}}\gamma_{\overline{k}_{\xi}}^{2}}{\beta_{\overline{k}'}\gamma_{\overline{k}'}}\leq\lim_{k\rightarrow\infty}\frac{\gamma_{\overline{k}_{\xi}}}{\left(1-\xi\right)^{c_{1}+c_{2}}}=0.
\end{equation}
We have shown that each term of $\Omega_{k}$ in (\ref{eq:omega})
is vanishing, hence $\Omega_{k}\rightarrow0$ implying that $w_{i,k}\rightarrow0$
and $\left|b_{i,k}\right|\rightarrow0$.

\subsection{\label{subsec:Proof-sketch-of_e}Proof Proposition~\ref{lem:stochastic_noise}}

We first show that $\{\sum_{k=K}^{K'}\left(\boldsymbol{a}_{k}-\boldsymbol{a}^{*}\right)^{T}\cdot\boldsymbol{e}_{k}\}_{K'\geq K}$
is martingale, then apply Doob's martingale inequality \cite{doob1953stochastic}
to prove Proposition~\ref{lem:stochastic_noise}. In order to lighten
the notations, we introduce $\mathcal{F}_{k}=\left\{ \mathbf{S}_{k},\boldsymbol{\Phi}_{k},\mathcal{I}_{k},\boldsymbol{\eta}_{k},\boldsymbol{\delta}_{k},\boldsymbol{\ell}_{k}\right\} $
to denote the collection of all stochastic terms.

The noise term $\boldsymbol{e}_{k}$ has zero mean, since $\mathbb{E}_{\mathcal{F}}(\boldsymbol{e}_{k})=\mathbb{E}_{\mathcal{F}}(\widehat{\boldsymbol{g}}_{k}-\overline{\boldsymbol{g}}_{k})=\overline{\boldsymbol{g}}_{k}-\overline{\boldsymbol{g}}_{k}=\mathbf{0}$,
$\forall\boldsymbol{a}_{k}\in\mathcal{A}$. Due to the independence
of $\mathcal{F}_{k}$ and $\mathcal{F}_{k'}$ for any $k\neq k'$,
$\boldsymbol{e}_{k}$ and $\boldsymbol{e}_{k'}$ are independent.
Hence, the sequence $\{\sum_{k=K}^{K'}\left(\boldsymbol{a}_{k}-\boldsymbol{a}^{*}\right)^{T}\cdot\boldsymbol{e}_{k}\}_{K'\geq K}$
is martingale. We apply Doob's martingale inequality to get, $\forall\rho>0$,
\begin{align}
 & \mathbb{P}\!\left(\!\!\sup_{K'\geq K}\!\left\Vert \frac{1}{N}\sum_{k=K}^{K'}\!\left(\boldsymbol{a}_{k}\!-\!\boldsymbol{a}^{*}\right)^{T}\!\!\cdot\!\boldsymbol{e}_{k}\right\Vert \!\geq\!\rho\!\right)\nonumber \\
 & \leq\frac{1}{\rho^{2}N^{2}}\mathbb{E}_{\mathcal{F}}\!\left(\left\Vert \sum_{k=K}^{K'}\!\left(\boldsymbol{a}_{k}\!-\!\boldsymbol{a}^{*}\right)^{T}\!\!\cdot\!\boldsymbol{e}_{k}\right\Vert ^{2}\right).\label{eq:temp1}
\end{align}
We need to evaluate 
\begin{align}
 & \mathbb{E}_{\mathcal{F}}\!\left(\left\Vert \sum_{k=K}^{K'}\!\left(\boldsymbol{a}_{k}\!-\!\boldsymbol{a}^{*}\right)^{T}\!\!\cdot\!\boldsymbol{e}_{k}\right\Vert ^{2}\right)\!\overset{(a)}{=}\!\sum_{k=K}^{K'}\!\mathbb{E}_{\mathcal{F}}\!\left(\!\left\Vert \left(\boldsymbol{a}_{k}\!-\!\boldsymbol{a}^{*}\right)^{T}\!\!\cdot\!\boldsymbol{e}_{k}\right\Vert ^{2}\right)\nonumber \\
 & \overset{\left(b\right)}{\leq}\!\sum_{k=K}^{K'}\!\mathbb{E}_{\mathcal{F}}\!\left(\!\left\Vert \boldsymbol{a}_{k}\!-\!\boldsymbol{a}^{*}\right\Vert ^{2}\!\left\Vert \boldsymbol{e}_{k}\right\Vert ^{2}\!\right)\!\overset{\left(c\right)}{\leq}\!Nd_{\max}^{2}\!\sum_{k=K}^{K'}\!\mathbb{E}_{\mathcal{F}}\!\left(\!\left\Vert \widehat{\boldsymbol{g}}_{k}\!-\!\overline{\boldsymbol{g}}_{k}\right\Vert ^{2}\!\right)\nonumber \\
 & \leq\!Nd_{\max}^{2}\!\sum_{k=K}^{K'}\!\mathbb{E}_{\mathcal{F}}\!\left(\!\left\Vert \widehat{\boldsymbol{g}}_{k}\right\Vert ^{2}\!\right)\overset{\left(d\right)}{\leq}N^{2}d_{\max}^{2}C'\sum_{k=K}^{K'}\overline{\beta^{2}}_{k}\label{eq:temp2}
\end{align}
where $(a)$ comes from $\mathbb{E}(e_{i,k_{1}}e_{i,k_{2}})=0$ for
any $k_{1}\neq k_{2}$; $\left(b\right)$ is by Cauchy\textendash Schwarz
inequality; in $\left(c\right)$ we denote $d_{\max}^{2}=\max_{i\in\mathcal{N}}\{(a_{i,\max}-a_{i,\min})^{2}$\},
then we have $\left\Vert \boldsymbol{a}_{k}-\boldsymbol{a}^{*}\right\Vert ^{2}\leq Nd_{\max}^{2}$,
recall that $a_{i,k}\in[a_{i,\min},a_{i,\max}]$, $\forall i\in\mathcal{N}$;
$\left(d\right)$ is by Lemma~\ref{lem:bound_g} stated in what follows,
of which the proof is given in Appendix~~\ref{subsec:Proof-of-gbound}. 
\begin{lem}
\label{lem:bound_g}If all the assumptions are satisfied, then $\mathbb{E}_{\mathbf{S},\boldsymbol{\Phi},\mathcal{I},\boldsymbol{\eta},\boldsymbol{\delta},\boldsymbol{\ell}}(\left\Vert \widehat{\boldsymbol{g}}_{k}\right\Vert ^{2})<NC'\overline{\beta^{2}}_{k}$,
with $C'=(1+q_{\mathrm{r}}^{-1}\lambda)\sigma_{\Phi}^{2}\sigma_{\eta}^{2}+(1+(2q_{\mathrm{r}}^{-1}+5)\lambda+(q_{\mathrm{r}}^{-1}+5)\lambda^{2}+\lambda^{3})L^{2}\sigma_{\Phi}^{2}\sigma_{\boldsymbol{a}}^{2}<\infty$. 
\end{lem}
Substituting (\ref{eq:temp2}) into (\ref{eq:temp1}), we get 
\begin{equation}
\mathbb{P}\!\left(\!\!\sup_{K'\geq K}\!\left\Vert \frac{1}{N}\!\sum_{k=K}^{K'}\!\left(\boldsymbol{a}_{k}\!-\!\boldsymbol{a}^{*}\right)^{T}\!\!\cdot\!\boldsymbol{e}_{k}\right\Vert \!\geq\!\rho\!\right)\!\!\leq\!\frac{d_{\max}^{2}C'}{\rho^{2}}\!\sum_{k=K}^{K'}\!\overline{\beta^{2}}_{k}.\label{eq:e_up_1}
\end{equation}
Since$\lim_{K\rightarrow\infty}\sum_{k=K}^{K'}\overline{\beta^{2}}_{k}=0$
by Lemma~\ref{lem:gamma_sum}, we can say that $N^{-1}\left\Vert \sum_{k=K}^{\infty}\left(\boldsymbol{a}_{k}-\boldsymbol{a}^{*}\right)^{T}\cdot\boldsymbol{e}_{k}\right\Vert $
is bounded a.s. according to (\ref{eq:e_up_1}), Proposition~\ref{lem:stochastic_noise}
is then proved.

\subsection{\label{subsec:Proof-of-gbound}Proof of Lemma~\ref{lem:bound_g}}

We evaluate the expectation of $\widehat{g}_{i,k}^{2}$ on all the
random terms, 
\begin{align}
 & \mathbb{E}_{\mathcal{F}}\left(\widehat{g}_{i,k}^{2}\right)=\mathbb{P}\left(\delta_{i,k}=1\right)\mathbb{E}_{\mathcal{F}}\left(\widehat{g}_{i,k}^{2}\mid\delta_{i,k}=1\right)\nonumber \\
 & =q_{\mathrm{a}}\mathbb{E}_{\boldsymbol{\Phi},\widetilde{\boldsymbol{\ell}}}\left(\beta_{1+\widetilde{\ell}_{i,k}}^{2}\Phi_{i,k}^{2}\mathbb{E}_{\mathbf{S},\mathcal{I},\boldsymbol{\eta},\boldsymbol{\delta}}\left(\widetilde{f}_{i,k}^{2}\mid\delta_{i,k}=1\right)\right),\label{eq:gs_e}
\end{align}
it is worth mentioning that $\widetilde{\boldsymbol{\ell}}_{k}$,
$\boldsymbol{\Phi}_{k}$, $\boldsymbol{\eta}_{k}$, $\mathbf{S}_{k}$,
and $\boldsymbol{\delta}_{k}$ are mutually independent. According
to definition of $\widetilde{f}_{i,k}$ (\ref{eq:f_est-1}), as $\delta_{i,k}=1$,
\begin{align*}
 & \widetilde{f}_{i,k}^{2}\!=\!(\widetilde{u}_{i,k}\!+\!\!\sum_{j\in\mathcal{N}^{\!\left(k\right)}\!\setminus\!\left\{ i\right\} }\!\!\!\frac{\kappa_{i,j,k}}{q_{\mathrm{r}}}\widetilde{u}_{j,k})^{2}\!=\widetilde{u}_{i,k}^{2}\!+\!\!\sum_{j\in\mathcal{N}^{\!\left(k\right)}\!\setminus\!\left\{ i\right\} }\!\!\!\frac{\kappa_{i,j,k}^{2}}{q_{\mathrm{r}}^{2}}\widetilde{u}_{j,k}^{2}\\
 & +\!\!\sum_{j\in\mathcal{N}^{\!\left(k\right)}\!\setminus\!\left\{ i\right\} }\!\!\!\frac{2\kappa_{i,j,k}}{q_{\mathrm{r}}}\widetilde{u}_{i,k}\widetilde{u}_{j,k}\!+\!\!\sum_{\stackrel{j_{1},j_{2}\in\mathcal{N}^{\!\left(k\right)}:}{j_{1}\neq j_{2}\neq i}}\!\!\!\frac{\kappa_{i,j_{1},k}\kappa_{i,j_{2},k}}{q_{\mathrm{r}}}\widetilde{u}_{j_{1},k}\widetilde{u}_{j_{2},k}.
\end{align*}
Recall that $\mathbb{E}(\kappa_{i,j,k})=\mathbb{E}(\kappa_{i,j,k}^{2})=q_{\mathrm{r}}$,
$\forall j\neq i$ and $\mathbb{E}(\kappa_{i,j_{1},k}\kappa_{i,j_{2},k})=q_{\mathrm{r}}^{2}$,
$\forall j_{1}\neq j_{2}\neq i$, as $\kappa_{i,j_{1},k}$ and $\kappa_{i,j_{2},k}$
are independent. Similarly, we have $\mathbb{E}_{\boldsymbol{\eta}}(\widetilde{u}_{i,k}^{2})=\mathbb{E}_{\boldsymbol{\eta}}((\widetilde{u}_{i,k}+\eta_{i,k})^{2})=u_{i,k}^{2}+\sigma_{\eta}^{2}$,
$\forall i$, and $\mathbb{E}_{\boldsymbol{\eta}}(\widetilde{u}_{i,k}\widetilde{u}_{j,k})=u_{i,k}u_{j,k}$,
$\forall i\neq j$. We can take expectation of $\widetilde{f}_{i,k}^{2}$
on $\mathcal{I}^{(i,k)}$ and $\boldsymbol{\eta}_{k}$ to get 
\begin{align}
 & \mathbb{E}_{\mathcal{I},\boldsymbol{\eta}}\left(\widetilde{f}_{i,k}^{2}\mid\delta_{i,k}=1\right)=\left(\!1+\frac{m}{q_{\mathrm{r}}}\!\right)\!\sigma_{\eta}^{2}+u_{i,k}^{2}\nonumber \\
{\color{blue}} & +\!\!\sum_{j\in\mathcal{N}^{\!\left(k\right)}\!\setminus\!\left\{ i\right\} }\!\!\left(\frac{1}{q_{\mathrm{r}}}u_{j,k}^{2}\!+2u_{i,k}u_{j,k}\!\right)\!+\!\!\sum_{\stackrel{j_{1},j_{2}\in\mathcal{N}^{\!\left(k\right)}:}{j_{1}\neq j_{2}\neq i}}\!\!u_{j_{1},k}u_{j_{2},k}\label{eq:gs_e1}
\end{align}
where we denote $m=\sum_{j\in\mathcal{N}^{\left(k\right)}\setminus\left\{ i\right\} }1=n_{k}-1$.

We then need to find an upper bound of $u_{i,k}^{2}$ and $u_{j_{1},k}u_{j_{2},k}$.
For any $\boldsymbol{\delta}_{k}$, $\mathbf{S}_{k}$ and $j\in\mathcal{N}^{\left(k\right)}$,
we have, 
\begin{align}
u_{j,k}^{2} & =u_{j}^{2}\left(\widehat{\boldsymbol{a}}_{k},\boldsymbol{\delta}_{k},\mathbf{S}_{k}\right)\overset{\left(a\right)}{\leq}\!\left(\left\Vert u_{j}\left(\mathbf{0},\boldsymbol{\delta}_{k},\mathbf{S}_{k}\right)\right\Vert \!+\!L_{\mathbf{S}_{k}}\!\left\Vert \widehat{\boldsymbol{a}}_{k}\circ\boldsymbol{\delta}_{k}\right\Vert \right)^{2}\nonumber \\
 & \overset{\left(b\right)}{\leq}L_{\mathbf{S}_{k}}^{2}\left\Vert \widehat{\boldsymbol{a}}_{k}\circ\boldsymbol{\delta}_{k}\right\Vert ^{2}\overset{\left(c\right)}{\leq}L_{\mathbf{S}_{k}}^{2}\left(m+1\right)\sigma_{\boldsymbol{a}}^{2}<\infty,\label{eq:u_bound}
\end{align}
where $\left(a\right)$ is by (\ref{eq:lipsc}),\emph{ i.e.}, the
assumption that $u_{i}$ is Lipschitz; $\left(b\right)$ comes from
$u_{j}\left(\mathbf{0},\boldsymbol{\delta}_{k},\mathbf{S}_{k}\right)=u_{j}\left(\mathbf{0},\mathbf{0},\mathbf{S}_{k}\right)=0$,
as $\widehat{\boldsymbol{a}}_{k}=\mathbf{0}$ also means no nodes
perform action; $\left(c\right)$ is by $\left\Vert \widehat{\boldsymbol{a}}_{k}\circ\boldsymbol{\delta}_{k}\right\Vert ^{2}\leq\sum_{j\in\mathcal{N}}\delta_{j,k}\sigma_{\boldsymbol{a}}^{2}=\left(m+1\right)\sigma_{\boldsymbol{a}}^{2}$,
where $\sigma_{\boldsymbol{a}}^{2}$ is the upper bound of $\widehat{a}_{i,k}^{2}$
defined in (\ref{eq:amax-1}). Based on (\ref{eq:u_bound}), we can
also deduce 
\begin{equation}
u_{j_{1},k}u_{j_{2},k}\leq\left|u_{j_{1},k}\right|\left|u_{j_{2},k}\right|\leq L_{\mathbf{S}_{k}}^{2}\left(m+1\right)\sigma_{\boldsymbol{a}}^{2}\label{eq:u_bound1}
\end{equation}
for any $j_{1},j_{2}\in\mathcal{N}^{\left(k\right)}$ such that $j_{1}\neq j_{2}$.

By substituting (\ref{eq:u_bound}) and (\ref{eq:u_bound1}) into
(\ref{eq:gs_e1}), we get 
\begin{align}
 & \mathbb{E}_{\mathcal{I},\boldsymbol{\eta}}\!\left(\widetilde{f}_{i,k}^{2}\mid\delta_{i,k}=1\right)\leq\left(q_{\mathrm{r}}^{-1}m+1\right)\sigma_{\eta}^{2}\nonumber \\
 & \quad+\left(1+\left(q_{\mathrm{r}}^{-1}+2\right)\left(m+m^{2}\right)+m^{3}\right)L_{\mathbf{S}_{k}}^{2}\sigma_{\boldsymbol{a}}^{2}.\label{eq:gs_e2}
\end{align}
Meanwhile, we have $L^{2}=\mathbb{E}_{\mathbf{S}}\left(L_{\mathbf{S}_{k}}^{2}\right)<\infty$
by Assumption~\ref{Assumption:lipc}. In both cases where the random
variable $m=\sum_{j\in\mathcal{N}\setminus\left\{ i\right\} }\delta_{j,k}$
follows a binomial distribution or Poisson distribution, it is easy
to show that $\mathbb{E}(m)=(N\!-1)q_{\mathrm{a}}\leq\lambda$, $\mathbb{E}(m^{2})\leq\lambda^{2}+\lambda$,
and $\mathbb{E}(m^{3})\leq\lambda^{3}+3\lambda^{2}+\lambda$. Thus
we can further take the expectation of both sides of (\ref{eq:gs_e2})
on $\mathbf{S}_{k}$ and $\boldsymbol{\delta}_{k}$ to get 
\begin{align}
 & \mathbb{E}_{\mathbf{S},\mathcal{I},\boldsymbol{\eta},\boldsymbol{\delta}}(\widetilde{f}_{i,k}^{2}\mid\delta_{i,k}=1)\!\leq\!\left(1+q_{\mathrm{r}}^{-1}\lambda\right)\sigma_{\eta}^{2}+\left(1+\lambda^{3}\right)L^{2}\sigma_{\boldsymbol{a}}^{2}\nonumber \\
 & \qquad+\left((5+2q_{\mathrm{r}}^{-1})\lambda+(5+q_{\mathrm{r}}^{-1})\lambda^{2}\right)L^{2}\sigma_{\boldsymbol{a}}^{2}=\sigma_{\Phi}^{-2}C',\label{eq:gs_e3}
\end{align}
with $C'$ defined in Lemma~\ref{lem:bound_g}.

Finally, by substituting (\ref{eq:gs_e3}) into (\ref{eq:gs_e}),
we get 
\begin{align}
\mathbb{E}_{\mathcal{F}}(\widehat{g}_{i,k}^{2}) & \leq q_{\mathrm{a}}\mathbb{E}_{\boldsymbol{\Phi},\widetilde{\boldsymbol{\ell}}}(\beta_{1+\widetilde{\ell}_{i,k}}^{2}\Phi_{i,k}^{2}\sigma_{\Phi}^{-2}C')\nonumber \\
 & =C'q_{\mathrm{a}}\mathbb{E}_{\widetilde{\boldsymbol{\ell}}}(\beta_{1+\widetilde{\ell}_{i,k}}^{2})\sigma_{\Phi}^{-2}\mathbb{E}_{\boldsymbol{\Phi}}(\Phi_{i,k}^{2})=C'\overline{\beta^{2}}_{k},\label{eq:C_def}
\end{align}
note that $\mathbb{E}_{\boldsymbol{\Phi}}(\Phi_{i,k}^{2})=\sigma_{\Phi}^{2}$
and $\overline{\beta^{2}}_{k}=\mathbb{E}(\delta_{i,k}\beta_{\ell_{i,k}}^{2})=\mathbb{P}(\delta_{i,k}=1)\mathbb{E}_{\widetilde{\boldsymbol{\ell}}}(\beta_{1+\widetilde{\ell}_{i,k}}^{2})$.
In the end, Lemma~\ref{lem:bound_g} can be proved since $\mathbb{E}_{\mathcal{F}}(\left\Vert \widehat{\boldsymbol{g}}_{k}\right\Vert ^{2})=\sum_{i=1}^{N}\mathbb{E}_{\mathcal{F}}(\widehat{g}_{i,k}^{2})\leq NC'\overline{\beta^{2}}_{k}.$

\subsection{\label{subsec:Proof-proj}Proof of Proposition~\ref{lem:bias_projection}}

By definition, $\delta_{i,k}\iota_{i,k}$ takes binary value, we can
evaluate 
\begin{align}
 & \mathbb{E}(\delta_{i,k}\iota_{i,k})=\mathbb{P}(\delta_{i,k}=1,\widetilde{\ell}_{i,k}<K_{0}-1)=\mathbb{P}(\delta_{i,k}=1)\nonumber \\
 & \times\mathbb{P}(\widetilde{\ell}_{i,k}<K_{0}-1)\overset{\left(a\right)}{\leq}q_{\mathrm{a}}\exp\!\left(\!-\frac{\left(q_{\mathrm{a}}\left(k-1\right)-\left(K_{0}-1\right)\right)^{2}}{2q_{\mathrm{a}}\left(k-1\right)}\!\right)\nonumber \\
 & \leq q_{\mathrm{a}}\exp\left(-\frac{q_{\mathrm{a}}}{2}\left(k-1\right)+K_{0}-1\right),\label{eq:proj_b1}
\end{align}
where $\left(a\right)$ is by Chernoff's bound, note that $\widetilde{\ell}_{i,k}\sim\mathcal{B}(k-1,q_{\mathrm{a}})$.
From (\ref{eq:proj_b1}) and the definition of $\varDelta_{k}$ in
Proposition~\ref{lem:gamma_sum-1}, we get 
\begin{align}
\mathbb{E}\left(\varDelta_{k}\right) & \leq\alpha_{\Phi}^{2}\gamma_{0}^{2}q_{\mathrm{a}}\exp\left(-\frac{q_{\mathrm{a}}}{2}\left(k-1\right)+K_{0}-1\right).
\end{align}
Meanwhile, we obtain $\overline{\beta^{2}}_{k-1}\geq q_{\mathrm{a}}\left(q_{\mathrm{a}}\left(k-2\right)+1\right)^{-2c_{2}}$
using similar steps as (\ref{eq:low_step}). We have 
\[
\lim_{k\rightarrow\infty}\frac{\exp\left(-\frac{q_{\mathrm{a}}}{2}\left(k-1\right)+K_{0}-1\right)}{\left(q_{\mathrm{a}}\left(k-2\right)+1\right)^{-2c_{2}}}=0,
\]
meaning that the upper bound of $\mathbb{E}(\varDelta_{k})$ decreases
much faster than the lower bound of $\overline{\beta^{2}}_{k-1}$.
Therefore, there must exist some bounded constants $K_{1}\geq K_{0}$
and $\widetilde{C}>0$, such that $\mathbb{E}(\varDelta_{k})\leq\widetilde{C}\overline{\beta^{2}}_{k-1}$,
$\forall K\geq K_{1}$.

Denote $e_{k}'=\varDelta_{k}-\mathbb{E}(\varDelta_{k})$, then $\{\sum_{k=K_{1}}^{K'}e_{k}'\}_{K'\geq K_{1}}$
is martingale because of $\mathbb{E}(e_{k}')=0$ and the independence
of $\varDelta_{k}$ and $\varDelta_{k'}$ for any $k\neq k'$. Obviously,
$0\leq\varDelta_{k}\leq\alpha_{\Phi}^{2}\gamma_{0}^{2}$, thus $\left|e_{k}'\right|\leq\left|\varDelta_{k}\right|<\infty$.
We can use Doob's martingale inequality to prove $\left|\sum_{k=K_{1}}^{\infty}e_{k}'\right|<\infty$
a.s., with similar steps as the proof of Proposition~\ref{lem:stochastic_noise}.
In the end, we have 
\begin{align}
\sum_{k=K_{1}}^{\infty}\varDelta_{k}\!=\!\sum_{k=K_{1}}^{\infty}\!\mathbb{E}\left(\varDelta_{k}\right)\!+\!\sum_{k=K_{1}}^{\infty}e_{k}'\! & \leq\!\sum_{k=K_{1}}^{\infty}\!\overline{\beta^{2}}_{k-1}\!+\left|\sum_{k=K_{1}}^{\infty}\!e_{k}'\right|\nonumber \\
 & <\infty\qquad\qquad\mathrm{a.s.}
\end{align}
in which $\sum_{k=1}^{\infty}\overline{\beta^{2}}_{k}<\infty$ by
Proposition\ \ref{lem:gamma_sum}. As $\varDelta_{k}\geq0$ by definition,
we also have $\left|\sum_{k=K_{1}}^{\infty}\varDelta_{k}\right|<\infty$
a.s., which concludes the proof.

\subsection{\label{subsec:Proof-sketch-of_th}Proof sketch of Theorem~\ref{prop:converge_rp}}

We perform the summation of (\ref{eq:evo_d}) from $k=K_{0}$ to $k=K$:
\begin{align}
 & d_{K+1}=d_{K_{0}}\!+\frac{2\sigma_{\Phi}^{2}}{\lambda}\!\sum_{k=K_{0}}^{K}\!\overline{\beta\gamma}_{k}\left(\boldsymbol{a}_{k}-\boldsymbol{a}^{*}\right)^{T}\!\cdot\!\left(\nabla F\!\left(\boldsymbol{a}_{k}\right)\!+\boldsymbol{b}_{k}\right)\nonumber \\
 & \!+\!\frac{1}{N}\!\sum_{k=K_{0}}^{K}\!\!\left\Vert \widehat{\boldsymbol{g}}_{k}\right\Vert ^{2}\!\!+\!\frac{2}{N}\!\sum_{k=K_{0}}^{K}\!\!\left(\boldsymbol{a}_{k}-\boldsymbol{a}^{*}\right)^{T}\!\!\cdot\!\boldsymbol{e}_{k}\!+\!\sum_{k=K_{0}}^{K}\!\!\varDelta_{k+1}.\label{eq:evo_s}
\end{align}
According to Lemma~\ref{lem:bound_g}, 
\begin{align}
 & \frac{1}{N}\sum_{k=1}^{\infty}\mathbb{E}\left(\left\Vert \widehat{\boldsymbol{g}}_{k}\right\Vert ^{2}\right)\leq C\sum_{k=1}^{\infty}\overline{\beta^{2}}_{k}<\infty,\label{eq:bound_b1}
\end{align}
as $\sum_{k=1}^{\infty}\overline{\beta^{2}}_{k}<\infty$ by Proposition~\ref{lem:gamma_sum}.
We can deduce that 
\begin{equation}
\frac{1}{N}\sum_{k=K_{0}}^{\infty}\left\Vert \widehat{\boldsymbol{g}}_{k}\right\Vert ^{2}<\infty,\quad\mathrm{a.s.}\label{eq:bound1}
\end{equation}
otherwise (\ref{eq:bound_b1}) cannot hold. Besides we also have $\frac{2}{N}\left|\sum_{k=K_{0}}^{K}\left(\boldsymbol{a}_{k}-\boldsymbol{a}^{*}\right)^{T}\cdot\boldsymbol{e}_{k}\right|<\infty$
a.s. and $\left|\sum_{k=K_{0}}^{K}\varDelta_{k+1}\right|<\infty$
a.s. by Propositions~\ref{lem:stochastic_noise} and~\ref{lem:bias_projection}.

From Theorem~\ref{prop:bias}, we know that $\left|b_{i,k}\right|\rightarrow0$,
$\forall i\in\mathcal{N}$. In other words, for an arbitrary small
positive value $\varepsilon$, there exists $K'$ such that $\left\Vert \nabla F\left(\boldsymbol{a}_{k}\right)+\boldsymbol{b}_{k}\right\Vert \geq\left(1-\varepsilon\right)\left\Vert \nabla F\left(\boldsymbol{a}_{k}\right)\right\Vert $.
By the concavity of $F$, we have $\left(\boldsymbol{a}_{k}-\boldsymbol{a}^{*}\right)^{T}\cdot\nabla F\left(\boldsymbol{a}_{k}\right)\leq0$,
thus 
\begin{align}
 & \frac{2\sigma_{\Phi}^{2}}{\lambda}\sum_{k=1}^{K}\overline{\beta\gamma}_{k}\left(\boldsymbol{a}_{k}-\boldsymbol{a}^{*}\right)^{T}\cdot\left(\nabla F\left(\boldsymbol{a}_{k}\right)+\boldsymbol{b}_{k}\right)\nonumber \\
 & \leq\frac{2\sigma_{\Phi}^{2}}{\lambda}\left(1-\varepsilon\right)\sum_{k=0}^{K}\overline{\beta\gamma}_{k}\left(\boldsymbol{a}_{k}-\boldsymbol{a}^{*}\right)^{T}\cdot\nabla F\left(\boldsymbol{a}_{k}\right).\label{eq:c3}
\end{align}
The following steps of the proof is the same to the classical proof
in \cite{robbins1951stochastic}. The basic idea is that, if $\boldsymbol{a}_{k}$
does not converge to $\boldsymbol{a}^{*}$, then due to $\sum_{k=0}^{\infty}\overline{\beta\gamma}_{k}\rightarrow\infty$,
we have 
\begin{equation}
\sum_{k=0}^{\infty}\overline{\beta\gamma}_{k}\left(\boldsymbol{a}_{k}-\boldsymbol{a}^{*}\right)^{T}\cdot\nabla F\left(\boldsymbol{a}_{k}\right)<-\infty,\label{eq:hy}
\end{equation}
which leads to $\lim_{K\rightarrow\infty}d_{K+1}<-\infty$ by the
above equations (\ref{eq:bound_b1}), (\ref{eq:c3}), and Proposition~\ref{lem:stochastic_noise}.
However $d_{K+1}$ should be positive by definition. Therefore, there
should be $\lim_{k\rightarrow\infty}\nabla F\left(\boldsymbol{a}_{k}\right)=\mathbf{0}$
and $\lim_{k\rightarrow\infty}\boldsymbol{a}_{k}=\boldsymbol{a}^{*}$
a.s., which concludes the proof.

\subsection{\label{subsec:Proof-of-Lemma_D}Proof of Lemma \ref{lem:inequality_D}}

The relation between $d_{k+1}$ and $d_{k}$ has been presented in~(\ref{eq:evo_d}).
In this proof, we aim to deduce an upper bound of $D_{k+1}=\mathbb{E}\left(d_{k+1}\right)$,
which should be a function of $D_{k}=\mathbb{E}\left(d_{k}\right)$.
By performing the expectation on all the random terms of (\ref{eq:evo_d}),
we have 
\begin{align}
 & D_{k+1}\leq D_{k}\!+\mathbb{E}\!\left(\frac{1}{N}\left\Vert \widehat{\boldsymbol{g}}_{k}\right\Vert ^{2}+\frac{2}{N}\left(\boldsymbol{a}_{k}-\boldsymbol{a}^{*}\right)^{T}\!\cdot\!\boldsymbol{e}_{k}+\varDelta_{k+1}\right)\nonumber \\
 & \qquad+\frac{2\sigma_{\Phi}^{2}}{\lambda}\overline{\beta\gamma}_{k}\mathbb{E}\left(\left(\boldsymbol{a}_{k}-\boldsymbol{a}^{*}\right)^{T}\!\cdot\!\left(\nabla F\left(\boldsymbol{a}_{k}\right)+\boldsymbol{b}_{k}\right)\right).\label{eq:Dk+1_Dk}
\end{align}
Since $\mathbb{E}\left(\boldsymbol{e}_{k}\right)=\mathbf{0}$, $\mathbb{E}\left(\varDelta_{k+1}\right)\leq\widetilde{C}\overline{\beta^{2}}_{k}$
and the upper bound of $\mathbb{E}(\left\Vert \widehat{\boldsymbol{g}}_{k}\right\Vert ^{2})$
has been given by Lemma~\ref{lem:bound_g}, we get 
\begin{equation}
\mathbb{E}\left(\!\frac{1}{N}\left\Vert \widehat{\boldsymbol{g}}_{k}\right\Vert ^{2}\!+\!\frac{2}{N}\left(\boldsymbol{a}_{k}-\boldsymbol{a}^{*}\right)^{T}\!\cdot\!\boldsymbol{e}_{k}+\varDelta_{k+1}\!\right)\!\leq C\overline{\beta^{2}}_{k},\label{eq:e_bound}
\end{equation}
with $C=C'+\widetilde{C}$. We then need to bound the last term $\mathbb{E}((\boldsymbol{a}_{k}-\boldsymbol{a}^{*})^{T}\!\cdot\!(\nabla F(\boldsymbol{a}_{k})+\boldsymbol{b}_{k}))$.
With the bound of $\left|b_{i,k}\right|$, we have 
\begin{align}
 & \left(\boldsymbol{a}_{k}-\boldsymbol{a}^{*}\right)^{T}\!\cdot\!\boldsymbol{b}_{k}\nonumber \\
 & \leq\sum_{i=1}^{N}\left|a_{i,k}-a_{i}^{*}\right|\left|b_{i,k}\right|\leq\frac{\alpha_{\Phi}^{3}\alpha_{G}}{2\sigma_{\Phi}^{2}}\sum_{i=1}^{N}\left|a_{i,k}-a_{i}^{*}\right|w_{i,k}\nonumber \\
 & <\!\frac{\alpha_{\Phi}^{3}\alpha_{G}q_{\mathrm{a}}\psi_{k}}{2\sigma_{\Phi}^{2}\overline{\beta\gamma}_{k}}\!\sqrt{N\!\sum_{i=1}^{N}(a_{i,k}-a_{i}^{*})^{2}}\!=\!\frac{\alpha_{\Phi}^{3}\alpha_{G}\lambda\psi_{k}}{2\sigma_{\Phi}^{2}\overline{\beta\gamma}_{k}}\!\sqrt{d_{k}}.\label{eq:ab_bound}
\end{align}
Note that we use (\ref{eq:w_bound1}) to bound $w_{i,k}$, \emph{i.e.},
$w_{i,k}\leq q_{\mathrm{a}}\overline{\beta\gamma}_{k}^{-1}\psi_{k}$
with $\psi_{k}$ defined in (\ref{eq:PHI_def}). Since $\mathbb{E}\left(\sqrt{d_{k}}\right)\leq\sqrt{\mathbb{E}\left(d_{k}\right)}=\sqrt{D_{k}}$,
by taking expectation on both sides of (\ref{eq:ab_bound}), we get
\begin{equation}
\mathbb{E}\left(\left(\boldsymbol{a}_{k}-\boldsymbol{a}^{*}\right)^{T}\!\cdot\!\boldsymbol{b}_{k}\right)<\frac{\alpha_{\Phi}^{3}\alpha_{G}\lambda\psi_{k}}{2\sigma_{\Phi}^{2}\overline{\beta\gamma}_{k}}\sqrt{D_{k}}.\label{eq:order1_1}
\end{equation}
Meanwhile, according to Assumption~\ref{Assumption:strong_concave},
we have 
\begin{equation}
\mathbb{E}\left(\left(\boldsymbol{a}_{k}-\boldsymbol{a}^{*}\right)^{T}\!\cdot\!\nabla F\left(\boldsymbol{a}_{k}\right)\right)\leq-\alpha_{F}ND_{k}\label{eq:order1_2-1}
\end{equation}
Substituting (\ref{eq:e_bound}), (\ref{eq:order1_1}) and (\ref{eq:order1_2-1})
into (\ref{eq:Dk+1_Dk}), we get 
\begin{align*}
 & D_{k+1}\!\leq\!\left(1\!-\!2\sigma_{\Phi}^{2}\alpha_{F}q_{\mathrm{a}}^{-1}\overline{\beta\gamma}_{k}\right)\!D_{k}\!+\alpha_{G}\alpha_{\Phi}^{3}\psi_{k}\sqrt{D_{k}}+C\overline{\beta^{2}}_{k},
\end{align*}
which concludes the proof.

\subsection{\label{subsec:proof_convergencerate}Proof of Theorem~\ref{prop:rate}}

We present the proof of (\ref{eq:Dk_up1}) and of (\ref{eq:Dk_up2})
in Appendix~\ref{subsec:Proof_rate1} and in Appendix~\ref{subsec:Proof_rate2}
respectively.

\subsubsection{\label{subsec:Proof_rate1}Proof of (\ref{eq:Dk_up1})}

The proof realized by induction. First of all, we can easily get $D_{K_{0}}\leq\vartheta^{2}\psi_{K_{0}}^{2}\theta_{K_{0}}^{-2}$
by definition of $\vartheta$. The main problem is to verify whether
$D_{k+1}\leq\vartheta^{2}\psi_{k+1}^{2}\theta_{k+1}^{-2}$ can be
obtained from $D_{k}\leq\vartheta^{2}\psi_{k}^{2}\theta_{k}^{-2}$,
$\forall k\geq K_{0}$.

Suppose that $D_{k}\leq\vartheta^{2}\psi_{k}^{2}\theta_{k}^{-2}$
is true, then by (\ref{eq:dk+1_dk}) we have 
\begin{align}
D_{k+1} & \leq\left(1-A\theta_{k}\right)\frac{\psi_{k}^{2}}{\theta_{k}^{2}}\vartheta^{2}+B\frac{\psi_{k}^{2}}{\theta_{k}}\vartheta+C\upsilon_{k},
\end{align}
as $1-A\theta_{k}\geq0$, $\forall k\geq K_{0}$. The problem turns
to prove the existence of a constant $\vartheta\in\mathbb{R}^{+}$
that ensures 
\begin{align}
D_{k+1}\leq\!\left(1\!-A\theta_{k}\right)\!\frac{\psi_{k}^{2}}{\theta_{k}^{2}}\vartheta^{2}\!+\!B\frac{\psi_{k}^{2}}{\theta_{k}}\vartheta\!+\!C\upsilon_{k} & \leq\vartheta^{2}\frac{\psi_{k+1}^{2}}{\theta_{k+1}^{2}}.\label{eq:rate_proof_1}
\end{align}
which can be rewrite as

\begin{equation}
\left(A-\chi_{k}\right)\vartheta^{2}-B\vartheta-C\upsilon_{k}\theta_{k}\psi_{k}^{-2}\geq0,\label{eq:para_1}
\end{equation}
in which $\chi_{k}=\left(1-\frac{\psi_{k+1}^{2}\theta_{k+1}^{-2}}{\psi_{k}^{2}\theta_{k}^{-2}}\right)\theta_{k}^{-1}$
as defined in (\ref{eq:chi_w}). By solving (\ref{eq:para_1}), we
obtain 
\[
\vartheta\geq\overline{\vartheta}_{k}=\frac{B}{2\left(A-\chi_{k}\right)}+\sqrt{\left(\frac{B}{2\left(A-\chi_{k}\right)}\right)^{2}+C\frac{\upsilon_{k}\theta_{k}\psi_{k}^{-2}}{A-\chi_{k}}},
\]
as $A-\chi_{k}>0$ by assumption and $\vartheta>0$. The last step
is to find an upper bound of $\overline{\vartheta}_{k}$ that is independent
of $k$. By definition (\ref{eq:chi_w}), we have $\epsilon_{1}\geq\chi_{k}$
and $\epsilon_{2}\geq\upsilon_{k}\theta_{k}\psi_{k}^{-2},$ $\forall k\geq K_{0}$.
According to the monotonicity of $\overline{\vartheta}_{k}$ w.r.t.
$\chi_{k}$ and $\upsilon_{k}\theta_{k}\psi_{k}^{-2}$, we have the
upper bound of $\overline{\vartheta}_{k}$, $\forall k\geq K_{0}$,
\begin{align}
\overline{\vartheta}_{k} & \leq\frac{B}{2\left(A-\epsilon_{1}\right)}+\sqrt{\frac{B^{2}}{4\left(A-\epsilon_{1}\right)^{2}}+\frac{C\epsilon_{2}}{A-\epsilon_{1}}}.
\end{align}
We can conclude that, $D_{k}\leq\vartheta^{2}\gamma_{k}^{2}$ leads
to $D_{k+1}\leq\vartheta^{2}\gamma_{k+1}^{2}$ if $\vartheta$ satisfies
(\ref{eq:para_0}), \emph{i.e.}, 
\begin{equation}
\vartheta\geq\sup_{k\geq K_{0}}\overline{\vartheta}_{k}=\frac{B+\sqrt{B^{2}+4C\epsilon_{2}\left(A-\epsilon_{1}\right)}}{2\left(A-\epsilon_{1}\right)}.\label{eq:rate_proof_2-1}
\end{equation}

\subsubsection{\label{subsec:Proof_rate2}Proof of of (\ref{eq:Dk_up2}) }

Similar steps can be used to prove (\ref{eq:Dk_up2}). First, we have
$D_{K_{0}}\leq\varrho^{2}\upsilon_{K_{0}}\theta_{K_{0}}^{-1}$ by
definition of $\varrho$. Then for any $k\geq K_{0}$, we should show
that $D_{k}\leq\varrho^{2}\upsilon_{k}\theta_{k}^{-1}$ leads to $D_{k+1}\leq\varrho^{2}\upsilon_{k+1}\theta_{k+1}^{-1}$.

Suppose that $D_{k}\leq\varrho^{2}\upsilon_{k}\theta_{k}^{-1}$ is
true, from (\ref{eq:dk+1_dk}), we have 
\begin{align}
D_{k+1} & \leq\left(1-A\theta_{k}\right)\frac{\upsilon_{k}}{\theta_{k}}\varrho^{2}+B\psi_{k}\sqrt{\frac{\upsilon_{k}}{\theta_{k}}}\varrho+C\upsilon_{k}.
\end{align}
To show $D_{k+1}\leq\varrho^{2}\upsilon_{k+1}\theta_{k+1}^{-1}$,
the following has to be true: 
\begin{equation}
\left(1-A\theta_{k}\right)\frac{\upsilon_{k}}{\theta_{k}}\varrho^{2}+B\psi_{k}\sqrt{\frac{\upsilon_{k}}{\theta_{k}}}\varrho+C\upsilon_{k}\leq\varrho^{2}\frac{\upsilon_{k+1}}{\theta_{k+1}}.\label{eq:rate_proof3}
\end{equation}
We rewrite (\ref{eq:rate_proof3}) as

\begin{equation}
\left(A-\varpi_{k}\right)\varrho^{2}-B\frac{\psi_{k}}{\sqrt{\theta_{k}\upsilon_{k}}}\varrho-C\geq0,\label{eq:rate_proof4}
\end{equation}
where by definition (\ref{eq:small_p}), $\varpi_{k}=\left(1-\frac{\upsilon_{k+1}\theta_{k+1}^{-1}}{\upsilon_{k}\theta_{k}^{-1}}\right)\theta_{k}^{-1}.$
As $A-\varpi_{k}\geq0$ and $\varrho>0$, (\ref{eq:rate_proof4})
can be solved, \emph{i.e.}, 
\[
\varrho\geq\overline{\varrho}_{k}=\frac{B\frac{\psi_{k}}{\sqrt{\theta_{k}\upsilon_{k}}}+\sqrt{\left(B\frac{\psi_{k}}{\sqrt{\theta_{k}\upsilon_{k}}}\right)^{2}+4C\left(A-\varpi_{k}\right)}}{2\left(A-\varpi_{k}\right)}.
\]
Consider $\epsilon_{3}$ and $\epsilon_{4}$ given in (\ref{eq:small_p}),
\emph{i.e.}, $\epsilon_{3}\geq\varpi_{k}$ and $\epsilon_{4}\geq\frac{\psi_{k}^{2}}{\theta_{k}\upsilon_{k}}$,
$\forall k\geq K_{0}$. We can derive the upper bound of $\overline{\varrho}_{k}$,
\begin{align}
\overline{\varrho}_{k} & \leq\frac{B\sqrt{\epsilon_{4}}+\sqrt{B^{2}\epsilon_{4}+4C\left(A-\epsilon_{3}\right)}}{2\left(A-\epsilon_{3}\right)},\quad\forall k\geq K_{0}.
\end{align}
Therefore, if $\varrho$ satisfies (\ref{eq:para_0-1}), \emph{i.e.},
\begin{equation}
\varrho\geq\sup_{k\geq K_{0}}\overline{\varrho}_{k}=\frac{B\sqrt{\epsilon_{4}}+\sqrt{B^{2}\epsilon_{4}+4C\left(A-\epsilon_{3}\right)}}{2\left(A-\epsilon_{3}\right)},\label{eq:rate_proof_2-1-1}
\end{equation}
then (\ref{eq:rate_proof3}) is true and $D_{k+1}\leq\varrho^{2}\upsilon_{k+1}\theta_{k+1}^{-1}$
holds, which concludes the proof.

\subsection{\label{subsec:Proof-of-Lemma_order}Proof of Lemma~\ref{lem:bounds_collect}}

Recall that $p_{k,\xi}=e^{-\frac{1}{2}\xi^{2}q_{\mathrm{a}}\left(k-1\right)}$,
$\overline{k}'=q\left(k-1\right)+1$, and $\overline{k}_{\xi}=\left\lfloor \left(1-\xi\right)q_{\mathrm{a}}\left(k-1\right)\right\rfloor +2$.
We use Lemma~\ref{lem:bound_chern} to get
\begin{align}
 & \theta_{k}<p_{k,\xi}\beta_{1}\gamma_{1}+\beta_{\overline{k}_{\xi}}\gamma_{\overline{k}_{\xi}},\upsilon_{k}<q_{\mathrm{a}}\left(p_{k,\xi}\beta_{1}^{2}+\beta_{\overline{k}_{\xi}}^{2}\right),\nonumber \\
 & \psi_{k}<\!\left(3\lambda^{2}\!+\!6\lambda q_{\mathrm{a}}\!+\!q_{\mathrm{a}}\right)\!\beta_{1}\gamma_{1}^{2}p_{k,\xi}\!+\!\left(\lambda^{2}\!+\!2\lambda q_{\mathrm{a}}\!+\!q_{\mathrm{a}}\right)\!\beta_{\overline{k}_{\xi}}\gamma_{\overline{k}_{\xi}}^{2}\label{eq:bounds_s}
\end{align}
The exponential term $p_{k,\xi}$ decreases much faster than $\beta_{\overline{k}_{\xi}}\gamma_{\overline{k}_{\xi}}$,
$\beta_{\overline{k}_{\xi}}\gamma_{\overline{k}_{\xi}}^{2}$ and $\beta_{\overline{k}_{\xi}}^{2}$.
Thus, for any $\xi'>0$, there exists $K'$ such that $\forall k\geq K'$,
one has $\theta_{k}<\left(1+\xi'\right)\beta_{\overline{k}_{\xi}}\gamma_{\overline{k}_{\xi}}$,
$\upsilon_{k}<\left(1+\xi'\right)q_{\mathrm{a}}\beta_{\overline{k}_{\xi}}^{2}$,
and $\psi_{k}<\left(1+\xi'\right)\left(\lambda+1\right)^{2}\beta_{\overline{k}_{\xi}}\gamma_{\overline{k}_{\xi}}^{2}$
from (\ref{eq:bounds_s}). Meanwhile, similar to (\ref{eq:low_step}),
by Jensen's inequality, we have 
\begin{equation}
\theta_{k}\geq\beta_{\overline{k}'}\gamma_{\overline{k}'};\:\:\:\upsilon_{k}\geq q_{\mathrm{a}}\beta_{\overline{k}'}^{2};\:\:\:\psi_{k}>\lambda^{2}\beta_{\overline{k}'}\gamma_{\overline{k}'}^{2}.\label{eq:lower_jenson}
\end{equation}
Hence, $\upsilon_{k}\theta_{k}^{-1}$ can be bounded as follows: 
\begin{align}
\frac{\upsilon_{k}}{\theta_{k}} & \overset{\left(a\right)}{<}\left(1+\xi'\right)q_{\mathrm{a}}\frac{\beta_{0}}{\gamma_{0}}\left(\frac{\left(1-\xi\right)q_{\mathrm{a}}\left(k-1\right)+1}{q_{\mathrm{a}}\left(k-1\right)+1}\right)^{-c_{1}-c_{2}}\nonumber \\
 & \qquad\times\left(\frac{\left(1-\xi\right)q_{\mathrm{a}}\left(k-1\right)+1}{k-1+1}\right)^{-c_{1}+c_{2}}k^{-c_{1}+c_{2}}\nonumber \\
 & \overset{\left(b\right)}{<}\left(1+\xi'\right)\beta_{0}\gamma_{0}^{-1}\left(1-\xi\right)^{-2c_{1}}q_{\mathrm{a}}^{1-c_{1}+c_{2}}k^{-c_{1}+c_{2}},
\end{align}
where $\left(a\right)$ is by $x-1<\left\lfloor x\right\rfloor \leq x$,
$\forall x\geq0$ and $\left(b\right)$ is by using $(\frac{yx+1}{x+1})^{-z}<y^{-z}$,
$\forall0<y<1$ and $\forall x,z\in\mathbb{R}^{+}$, so that 
\begin{align*}
 & \left(\!\frac{\left(1-\xi\right)\!q_{\mathrm{a}}\!\left(k-1\right)\!+1}{q_{\mathrm{a}}\left(k-1\right)+1}\!\right)^{\!\!-c_{1}\!-c_{2}}\!\!\left(\!\frac{\left(1-\xi\right)\!q_{\mathrm{a}}\!\left(k-1\right)\!+1}{\left(k-1\right)+1}\!\right)^{\!\!-c_{1}\!+c_{2}}\\
 & <\left(1-\xi\right)^{-c_{1}-c_{2}}\left(\left(1-\xi\right)q_{\mathrm{a}}\right)^{-c_{1}+c_{2}}=\left(1-\xi\right)^{-2c_{1}}q_{\mathrm{a}}^{-c_{1}+c_{2}}.
\end{align*}
Using similar steps, (\ref{eq:rate_2})-(\ref{eq:const_2}) can be
proved as well.

\subsection{\label{subsec:Proof_last}Proof of Lemma~\ref{lem:constant_term}}

In order to prove Lemma~\ref{lem:constant_term}, we mainly need
to show that $\lim_{k\rightarrow\infty}\beta_{0}\gamma_{0}\varpi_{k}<+\infty$
and $\lim_{k\rightarrow\infty}\beta_{0}\gamma_{0}\chi_{k}<+\infty$,
as both numerators and denominators of $\varpi_{k}$ and $\chi_{k}$
are bounded and vanishing. We present a basic lemma in Appendix~\ref{subsec:A-useful-lemma},
then the convergence of the upper bounds of $\varpi_{k}$ and of $\chi_{k}$
are investigated in Appendix~\ref{subsec:Convergence-ofw} and in
Appendix~\ref{subsec:Convergence-ofw}, respectively.

\subsubsection{\label{subsec:A-useful-lemma}A useful lemma }

We mainly prove the following lemma: 
\begin{lem}
\label{lem:diff_bound}Consider a sequence $z_{\ell}=\ell^{-c}$ with
$c>0$. Define $\overline{z}_{k}=\sum_{\ell=0}^{k-1}z_{\ell+1}q_{\mathrm{a}}^{\ell+1}\left(1-q_{\mathrm{a}}\right)^{k-\ell-1}\binom{k-1}{\ell}$
and 
\begin{align}
\overline{z}'_{k} & =\sum_{\ell=0}^{k-1}\frac{z_{\ell+1}}{\ell+1}q_{\mathrm{a}}^{\ell+1}\left(1-q_{\mathrm{a}}\right)^{k-1-\ell}\binom{k-1}{\ell},\label{eq:alpha_p}
\end{align}
then 
\begin{equation}
\overline{z}_{k}-\overline{z}_{k+1}<cq_{\mathrm{a}}\overline{z}'_{k}.\label{eq:diff_c}
\end{equation}
\end{lem}
\begin{IEEEproof}
We rewrite $\overline{z}_{k+1}$ as follows 
\begin{align}
 & \overline{z}_{k+1}\overset{\left(a\right)}{=}z_{1}q_{\mathrm{a}}\left(1-q_{\mathrm{a}}\right)^{k}+\sum_{\ell=1}^{k-1}z_{\ell+1}q_{\mathrm{a}}^{\ell+1}\left(1-q_{\mathrm{a}}\right)^{k-\ell}\binom{k-1}{\ell}\nonumber \\
 & \qquad+\sum_{\ell=1}^{k-1}z_{\ell+1}q_{\mathrm{a}}^{\ell+1}\left(1-q_{\mathrm{a}}\right)^{k-\ell}\binom{k-1}{\ell-1}+z_{k+1}q_{\mathrm{a}}^{k+1}\nonumber \\
 & =\sum_{\ell=0}^{k-1}\!\left(\left(1-q_{\mathrm{a}}\right)z_{\ell+1}+q_{\mathrm{a}}z_{\ell+2}\right)q_{\mathrm{a}}^{\ell+1}\!\left(1-q_{\mathrm{a}}\right)^{k-1-\ell}\!\binom{k-1}{\ell},\label{eq:rela}
\end{align}
where $\left(a\right)$ is by $\binom{k}{\ell}=\binom{k-1}{\ell}+\binom{k-1}{\ell-1}$.
From (\ref{eq:rela}), we have 
\begin{align}
 & \overline{z}_{k}-\overline{z}_{k+1}\overset{\left(a\right)}{=}\sum_{\ell=0}^{k-1}\left(z_{\ell+1}-z_{\ell+2}\right)q_{\mathrm{a}}^{\ell+2}\left(1-q_{\mathrm{a}}\right)^{k-1-\ell}\binom{k-1}{\ell}\nonumber \\
 & \overset{\left(b\right)}{<}\!cq_{\mathrm{a}}\!\sum_{\ell=0}^{k-1}\!\frac{z_{\ell+1}}{\ell+1}q_{\mathrm{a}}^{\ell+1}\!\left(1-q_{\mathrm{a}}\right)^{k-1-\ell}\!\binom{k-1}{\ell}\!=\!cq_{\mathrm{a}}\overline{z}'_{k},\label{eq:diff_c-1}
\end{align}
where $\left(b\right)$ is by $\frac{z_{\ell+1}-z_{\ell+2}}{z_{\ell+1}}=1-(1+\frac{1}{\ell+1})^{-c}<\frac{c}{\ell+1},$
such bound is tight when $\ell$ is large, as $\lim_{x\rightarrow0}\frac{1-(1+x)^{-c}}{cx}=1$.
It is worth mentioning that $\overline{z}_{k}-\overline{z}_{k+1}>0$
can be directly proved from (\ref{eq:diff_c-1})-$\left(a\right)$,
since $z_{\ell+1}>z_{\ell+2}$, $\forall\ell$. Such result is valid
for any decreasing sequence. 
\end{IEEEproof}

\subsubsection{\label{subsec:Convergence-ofw}Convergence of $\varpi_{k}$}

Applying Lemma~(\ref{lem:diff_bound}) and by replacing $z_{\ell}=\ell^{-c}$
with $\beta_{\ell}^{2}=\beta_{0}^{2}\ell^{-2c_{1}}$, we have 
\begin{align}
 & \upsilon_{k}\theta_{k+1}-\upsilon_{k+1}\theta_{k}<\left(\upsilon_{k}-\upsilon_{k+1}\right)\theta_{k}\nonumber \\
 & <2c_{1}\theta_{k}\sum_{\ell=0}^{k-1}\frac{\beta_{\ell+1}^{2}}{\ell+1}q_{\mathrm{a}}^{\ell+2}\left(1-q_{\mathrm{a}}\right)^{k-1-\ell}\binom{k-1}{\ell}.
\end{align}
We use Lemma~\ref{lem:bound_chern} to get, for any $0<\xi<1$, 
\begin{align*}
\upsilon_{k}\theta_{k+1}\!-\!\upsilon_{k+1}\theta_{k} & \!<\!2c_{1}q_{\mathrm{a}}^{2}\beta_{0}^{2}\theta_{k}(p_{k,\xi}\!+\!\left(\left(1-\xi\right)q_{\mathrm{a}}k\right)^{-2c_{1}-1}),
\end{align*}
from which and (\ref{eq:lower_jenson}) we can deduce
\begin{align}
\varpi_{k} & <\frac{2c_{1}q_{\mathrm{a}}^{2}\beta_{0}^{2}\left(p_{k,\xi}+\left(\left(1-\xi\right)q_{\mathrm{a}}k\right)^{-2c_{1}-1}\right)}{q_{\mathrm{a}}\beta_{0}^{2}\left(q_{\mathrm{a}}\left(k-1\right)+1\right)^{-2c_{1}}\beta_{0}\gamma_{0}\left(q_{\mathrm{a}}k+1\right)^{-c_{1}-c_{2}}}\nonumber \\
 & <\frac{2c_{1}q_{\mathrm{a}}\left(p_{k,\xi}+\left(\left(1-\xi\right)q_{\mathrm{a}}k\right)^{-2c_{1}-1}\right)}{\beta_{0}\gamma_{0}\left(q_{\mathrm{a}}k+1\right)^{-3c_{1}-c_{2}}}=\varpi_{k}^{+}.
\end{align}
We have $\lim_{k\rightarrow\infty}p_{k,\xi}\left(q_{\mathrm{a}}k+1\right)^{3c_{1}+c_{2}}=0$
and 
\[
\lim_{k\rightarrow\infty}\frac{\left(\left(1-\xi\right)q_{\mathrm{a}}k\right)^{-2c_{1}-1}}{\left(q_{\mathrm{a}}k+1\right)^{-3c_{1}-c_{2}}}=\begin{cases}
0 & \!\!\textrm{if }c_{1}+c_{2}<1\\
\left(1-\xi\right)^{-2c_{1}-1} & \!\!\textrm{if }c_{1}+c_{2}=1
\end{cases}
\]
Hence $\lim_{k\rightarrow\infty}\beta_{0}\gamma_{0}\varpi_{k}^{+}\leq2c_{1}q_{\mathrm{a}}\left(1-\xi\right)^{-2c_{1}-1}$
as $c_{1}+c_{2}\leq1$. We can deduce that $\beta_{0}\gamma_{0}\varpi_{k}$
is bounded and $\varpi_{k}<A$ can be true $\forall k$, as long as
the value of $\beta_{0}\gamma_{0}$ is large enough.

\subsubsection{\label{subsec:Convergence-ofw-1}Convergence of $\chi_{k}$}

We can use similar steps to show that $\beta_{0}\gamma_{0}\chi_{k}$
is bounded. We need to evaluate 
\begin{align}
 & \psi_{k}-\psi_{k+1}<\overline{\beta\gamma^{2}}_{k}-\overline{\beta\gamma^{2}}_{k+1}+2N\left(\overline{\beta\gamma}_{k}\overline{\gamma}_{k}-\overline{\beta\gamma}_{k+1}\overline{\gamma}_{k+1}\right)\nonumber \\
 & +\left(N-1\right)^{2}\left(\overline{\beta}_{k}\overline{\gamma^{2}}_{k}-\overline{\beta}_{k+1}\overline{\gamma^{2}}_{k+1}\right)\overset{\left(a\right)}{<}q_{\mathrm{a}}\left(c_{1}+2c_{2}\right)\overline{\beta\gamma^{2}}'_{k}\nonumber \\
 & +2Nq_{\mathrm{a}}\!\left(\!\left(c_{1}+c_{2}\right)\!\overline{\beta\gamma}'_{k}\overline{\gamma}_{k}+c_{2}\overline{\beta\gamma}_{k}\overline{\gamma}'_{k}-\!\left(c_{1}+c_{2}\right)\!c_{2}q_{\mathrm{a}}\overline{\beta\gamma}'_{k}\overline{\gamma}'_{k}\right)\nonumber \\
 & +\left(N-1\right)^{2}q_{\mathrm{a}}\left(2c_{2}\overline{\beta}_{k}\overline{\gamma^{2}}'_{k}+c_{1}\overline{\beta}'_{k}\overline{\gamma^{2}}_{k}-2c_{1}c_{2}q_{\mathrm{a}}\overline{\beta}'_{k}\overline{\gamma^{2}}'_{k}\right)\nonumber \\
 & \overset{\left(b\right)}{<}\left(c_{1}+2c_{2}\right)\!\beta_{0}\gamma_{0}^{2}q_{\mathrm{a}}\left(\lambda^{2}+2q_{\mathrm{a}}\right)\left(\left(\left(1-\xi\right)q_{\mathrm{a}}k\right)^{-c_{1}-2c_{2}-1}\right.\nonumber \\
 & \left.\hspace{16em}+p_{k,\xi}C''\right),\label{eq:diff_y}
\end{align}
where $(a)$ is obtained by applying Lemma~\ref{lem:diff_bound},
the terms $\overline{\beta\gamma^{2}}'_{k}$, $\overline{\beta\gamma}'_{k}$
and $\overline{\gamma}'_{k}$ are defined in the same way as $\overline{a}'_{k}$
in (\ref{eq:alpha_p}). We can use Lemma~\ref{lem:bound_chern} to
show $\left(b\right)$. Note that the explicit expression of the upper
bound is quite long, we introduce a bounded constant $C''$ instead.
The bound (\ref{eq:diff_y}) is reasonably tight as $p_{k,\xi}=\exp\left(-\frac{1}{2}\xi q_{\mathrm{a}}k\right)$
is negligible before $\overline{k}_{\xi}^{-c_{1}-2c_{2}-1}$ when
$k$ goes large.

Based on (\ref{eq:diff_y}) and (\ref{eq:lower_jenson}), we can get
\begin{align*}
 & \frac{\chi_{k}}{1+\frac{\psi_{k+1}\theta_{k}}{\psi_{k}\theta_{k+1}}}=\frac{\psi_{k}\theta_{k+1}-\psi_{k+1}\theta_{k}}{\psi_{k}\theta_{k+1}\theta_{k}}<\frac{\left(\psi_{k}-\psi_{k+1}\right)\theta_{k}}{\psi_{k}\theta_{k+1}\theta_{k}}\\
 & <\frac{\left(c_{1}+2c_{2}\right)q_{\mathrm{a}}\left(\left(\left(1-\xi\right)q_{\mathrm{a}}k\right)^{-c_{1}-2c_{2}-1}+p_{k,\xi}C''\right)}{\left(1+2q_{\mathrm{a}}\lambda^{-2}\right)^{-1}\beta_{0}\gamma_{0}\left(q_{\mathrm{a}}k+1\right)^{-2c_{1}-3c_{2}}}=\chi_{k}^{+}.
\end{align*}
Since $\lim_{k\rightarrow\infty}\left(\frac{\psi_{k+1}\theta_{k}}{\psi_{k}\theta_{k+1}}+1\right)=2$
and 
\[
\lim_{k\rightarrow\infty}\chi_{k}^{+}=\begin{cases}
0 & \textrm{if }c_{1}+c_{2}<1,\\
\frac{\left(c_{1}+2c_{2}\right)q_{\mathrm{a}}\left(\left(1-\xi\right)\right)^{-c_{1}-2c_{2}-1}}{\left(1+2q_{\mathrm{a}}\lambda^{-2}\right)^{-1}\beta_{0}\gamma_{0}} & \textrm{if }c_{1}+c_{2}=1,
\end{cases}
\]
we can conclude that $\lim_{k\rightarrow\infty}\beta_{0}\gamma_{0}\chi_{k}$
is bounded. Therefore $\chi_{k}<A$ can be true when $\beta_{0}\gamma_{0}$
is large enough. 
\bibliographystyle{IEEEtran}
\bibliography{BiblioWenjie}

\end{document}